\documentclass[11pt,oneside,reqno]{amsart}
\usepackage{geometry}
\usepackage[utf8]{inputenc}
\usepackage{amstext,latexsym,amsbsy,amsmath,amssymb,amsthm,mathtools,relsize,geometry,enumerate}
\usepackage{hyperref}
\hypersetup{
  colorlinks   = true, 
  urlcolor     = blue, 
  linkcolor    = red, 
  citecolor   = red 
}
\theoremstyle{definition}
\usepackage{graphicx}
\setkeys{Gin}{width=\linewidth,totalheight=\textheight,keepaspectratio}
\graphicspath{{./images/}}

\usepackage{multicol}
\usepackage{booktabs}
\usepackage{mathrsfs}

\usepackage[svgnames]{xcolor}

\newcommand{\nbb}{\mathbb{N}}
\newcommand{\rbb}{\mathbb{R}}

\newcommand{\cbb}{\mathbb{C}}

\newcommand{\CM}{\mathcal{CM}}
\newcommand{\Sc}{\mathcal{S}}
\newcommand{\F}[1]{\mathcal{F}\left[#1\right]}

\newcommand{\Kcos}{\mathcal{K}_{\cos}}
\newcommand{\Ksin}{\mathcal{K}_{\sin}}

\newcommand{\Khat}{\widehat{K}}

\newcommand{\la}{\langle}
\newcommand{\ra}{\rangle}
\newcommand{\dom}{\text{Dom}}

\newcommand{\ptil}{\widetilde{p}}
\newcommand{\qtil}{\widetilde{q}}

\newcommand{\phitil}{\widetilde{\varphi}}

\newcommand{\ubf}{\boldsymbol{u}}

\newcommand{\intx}{\int_0^t x(s)\mathrm{d} s}
\newcommand{\daw}{\textnormal{daw}}

\newcommand{\bbR}{\mathbb{R}}

\newcommand{\bbC}{\mathbb{C}}
\newcommand{\bbE}{\mathbb{E}}

\newcommand{\Ctilde}{\widetilde{C}}

\newcommand{\Wdot}{\dot{W}}

\newcommand{\sm}{\setminus\{0\}}
\newcommand{\erfc}{\textnormal{erfc}}
\newcommand{\erf}{\textnormal{erf}}

\renewcommand{\d}{\mathrm{d}}

\newcommand{\f}{\varphi}

\newcommand{\E}{\mathbb{E}}

\newcommand{\close}{\!\!\!}

\newcommand{\imag}{{\mathbf i}}

\numberwithin{equation}{section}
\theoremstyle{definition}
\theoremstyle{plain}
\newtheorem{theorem}{Theorem}[section]

\newtheorem{lemma}[theorem]{Lemma}
\newtheorem{assumption}[theorem]{Assumption}

\newtheorem{definition}[theorem]{Definition}

\newtheorem{remark}[theorem]{Remark}

\numberwithin{equation}{section}

\title{The generalized Langevin equation in harmonic potentials: Anomalous diffusion and equipartition of energy}

\author{Gustavo~Didier$^1$ and Hung D.~Nguyen$^2$}

\address{$^1$ Department of Mathematics, Tulane University, New Orleans, Louisiana, USA}

\address{$^2$ Department of Mathematics, University of California, Los Angeles, California, USA}

\begin{document}
\maketitle
\begin{abstract}

We consider the generalized Langevin equation (GLE) in a harmonic potential with power law decay memory. We study the anomalous diffusion of the particle's displacement and velocity. By comparison with the free particle situation in which the velocity was previously shown to be either diffusive or subdiffusive, we find that, when trapped in a harmonic potential, the particle's displacement may either be diffusive or superdiffusive. Under slightly stronger assumptions on the memory kernel, namely, for kernels related to the broad class of completely monotonic functions, we show that both the free particle and the harmonically bounded GLE satisfy the equipartition of energy condition. This generalizes previously known results for the GLE under particular kernel instances such as the generalized Rouse kernel or (exactly) a power law function.
\end{abstract}


\section{Introduction}

The classical Langevin equation describes the movement of a foreign particle freely suspended in Newtonian, viscous fluids. If the particle is further subjected to a harmonic potential $U(x)=\gamma x^2/2$, where $\gamma$ reflects the strength of the oscillator, the Langevin equation system is given by
\begin{equation} \label{eq:Langevin:2D}
\begin{aligned}
\dot{x}(t) &=v(t),\\
m\,\dot{v}(t)&=-\lambda v(t)-\gamma x(t)+\sqrt{2\lambda k_BT}\,\dot{W}(t).
\end{aligned}
\end{equation}
In \eqref{eq:Langevin:2D}, $(x(t),v(t))$ is a two-dimensional process, $m$ is the particle’s mass, $\lambda>0$ represents the viscous drag coefficient, $k_B$ is the Boltzman constant, $T$ is the temperature and $\{W(t)\}_{t\in\rbb}$ is a two--sided standard Brownian motion. However, unlike in a classical Langevin framework, fluid viscoelasticity induces time correlation between the foreign particle movement and molecular bombardment
\cite{didier2012statistical,didier2017asymptotic,didier2020asymptotic,kneller2011generalized,
kubo1966fluctuation,mason1995optical,mori1965transport}. To capture this memory effect,~\eqref{eq:Langevin:2D} is modified into the so--named generalized Langevin equation (GLE) system \cite{hohenegger2018reconstructing,kou2008stochastic,kou2004generalized}, namely,
\begin{equation} \label{eq:GLE:2D}
\begin{aligned}
\dot{x}(t) &=v(t),\\
m\,\dot{v}(t)&=-\lambda v(t)-\gamma x(t)-\beta\int_{-\infty}^t\close K(t-s)v(s) \d s+\sqrt{\beta k_BT}F(t)+\sqrt{2\lambda k_BT}\,\dot{W}(t).
\end{aligned}
\end{equation}
In~\eqref{eq:GLE:2D}, the function $K:\rbb\to\rbb^+$ is an even memory kernel that characterizes the delayed response of the fluid medium to the particle's past movement \cite{glatt2020generalized,mckinley2018anomalous}. In turn, $\{F(t)\}_{t\in\rbb}$ is a zero mean, stationary, Gaussian process that is linked to $K(t)$ via the relation
\begin{equation}\label{eq:fluctuation-dissipation}
\E[F(t)F(s)]=K(t-s).
\end{equation}
The equality in \eqref{eq:fluctuation-dissipation} expresses the so--called \emph{fluctuation–dissipation} relationship between $K$ and $F$. In other words, such relationship is the requirement that, in an equilibrium state, the covariance observed in thermal fluctuations be determined by the underlying memory kernel \cite{kubo1966fluctuation,mori1965transport,mori1965continued}.

In this paper, we provide two main sets of results on the long term behavior of a particle whose dynamics are given by the system~\eqref{eq:GLE:2D}. Namely, under broad assumptions, $(i)$ we asymptotically characterize the particle's (ensemble) mean squared displacement (MSD) assuming $\gamma > 0$ and $(ii)$ we establish that equipartition of energy holds assuming $\gamma \geq 0$ (which includes the free particle instance $\gamma = 0$ as in~\eqref{eq:GLE:1D} below). We now provide a more detailed description of each set of results.

By setting $\gamma = 0$ in~\eqref{eq:GLE:2D}, we arrive at
\begin{equation} \label{eq:GLE:1D}
m\,\dot{v}(t) = -\lambda v(t) -\beta\int_{-\infty}^t\close K(t-s)v(s) \d s+\sqrt{\beta k_BT}F(t) + \sqrt{2 \lambda k_B T} \dot{W}(t).
\end{equation}
Expression \eqref{eq:GLE:1D} is the GLE for a particle moving freely in a viscoelastic medium. Historically, this instance of the GLE was first proposed and studied in the seminal work~\cite{kubo1966fluctuation} and later popularized in \cite{mason1995optical,mori1965continued}. In the last several decades, ~\eqref{eq:GLE:1D} has attracted a great deal of attention due to its ability to model what is known as \emph{anomalous diffusion} \cite{kou2008stochastic,mckinley2018anomalous,morgado2002relation}. To be more precise, write $f(t)\sim g(t)$, $t\to\infty$, when, for some $c \in (0,\infty)$, $\lim_{t\to\infty}f(t)/g(t)=c$. A stochastic process is said to exhibit \textit{diffusive} behavior if its MSD grows linearly in time, i.e., $\E|\int_0^t v(s)\d s|^2\sim t$ as $t\to\infty$. Otherwise, if the growth rate is given by $t^\alpha$, where either $\alpha<1$ or $\alpha>1$, then the process is called \textit{sub}diffusive or \textit{super}diffusive, respectively. It was once a longstanding conjecture that the anomalously diffusive behavior of the stationary solution of~\eqref{eq:GLE:1D} was dictated by the decaying rate of the memory kernel $K$ \cite{morgado2002relation}. There have been several attempts to establish such conjecture by means of the asymptotic analysis of either Laplace \cite{kneller2011generalized,kupferman2004fractional,
morgado2002relation} or Fourier transforms \cite{didier2020asymptotic,kou2008stochastic,
mckinley2018anomalous}. Recently, anomalous diffusion for~\eqref{eq:GLE:1D} was fully characterized in terms of the memory kernel $K$. In other words, if $K$ is integrable, then it can be shown that the second moment $\E|\int_0^t v(s)\d s|^2$ grows linearly in time. On the other hand, if there exists $\alpha\in(0,1]$ such that $K(t)\sim t^{-\alpha}$ as $t\to\infty$, then for $\alpha\in(0,1)$, $\E|\int_0^t v(s)\d s|^2\sim t^\alpha$ \cite{mckinley2018anomalous}. Moreover, for $\alpha=1$, $\E|\int_0^t v(s)\d s|^2\sim t/\log(t)$ as $t\to\infty$ \cite{didier2020asymptotic}.

Anomalous diffusion has been mostly investigated for free particles. Nevertheless, there are many viscoelastic fluid systems in which the particle is trapped by a damped harmonic motion under the action of a stationary noise term that follows the fluctuation--dissipation relationship. In recent work~\cite{desposito2008memory,vinales2006anomalous}, similar systems to~\eqref{eq:GLE:2D} -- with the memory kernel restricted to the interval $[0,t]$, instead of $(-\infty,t]$ -- have been examined. Using a combination of Laplace analysis and Tauberian theorems, asymptotic expressions for the velocity autocorrelation functions were established in terms of the large scale (time) asymptotics of the memory kernel and the correlation function of the random force.

 In this paper, we employ the framework of weakly stationary random operators (\cite{mckinley2018anomalous}; see also \cite{didier2020asymptotic,gelfand1955generalized, ito1954stationary,  yaglom1957some}) to construct stationary solutions for the system~\eqref{eq:GLE:2D}. Moreover, following up on results for the MSD of the system~\eqref{eq:GLE:1D} \cite{didier2020asymptotic,kou2008stochastic,mckinley2018anomalous}, we use Fourier analysis \cite{soni1975slowly,soni1975slowlyII} to characterize the asymptotic behavior of the MSD of the bivariate stationary--increment process $\int_0^t (x(s),v(s))\d s$ in terms of the asymptotic decay rate of $K(t)$. Notably, whereas the process $v(t)$ as in~\eqref{eq:GLE:1D} may either be diffusive or subdiffusive depending on the memory kernel, in this paper we show that, for a large class of memory kernels $K$ (see Assumption~\ref{cond:K:general}), the process $x(t)$ in~\eqref{eq:GLE:2D} is either diffusive or superdiffusive (see Theorem~\ref{thm:asymptotic-growth}).

In the second set of main results, under slightly stronger assumptions on the memory kernel, we investigate the so--named \textit{equipartition of energy} condition for the solution pair $(x(t),v(t))$ for~\eqref{eq:GLE:2D} as well as for the solution $v(t)$ for~\eqref{eq:GLE:1D}. In Statistical Mechanics, it is well known that a stationary process in thermodynamical equilibrium~\cite{chandler1987introduction,
hill1986introduction,reif1965fundamentals} must satisfy such condition, i.e., any degree of freedom (e.g., particle position or velocity) appearing quadratically in the energy contributes $k_B T /2$ to the average kinetic energy of the system. However, the equipartition condition may hold even for out-of-equilibrium systems \cite{nichol2012equipartition}. Since such systems are commonly found in nature, the search for generalized equipartition laws and nonequilibrium relations is still a quite active research topic \cite{argun2016non,maggi2014generalized,to2014boltzmann}. One key motivation for studying the equipartition property in the framework of the GLE stems from the fact that \eqref{eq:GLE:2D} is a biophysical model~\cite{kou2008stochastic}; hence, it is of a matter of interest in practice whether or not an equilibrium condition generally holds.

To the best of the authors' knowledge, results on the equipartition of energy for instances of the GLE seem to have first been established in \cite{kou2008stochastic} based on memory kernels of the form
\begin{equation}\label{e:t^(-alpha)}
t^{-\alpha}, \quad \alpha\in (0,1).
\end{equation}
For the free particle case~\eqref{eq:GLE:1D}, it was shown that
\begin{equation}\label{e:E[mv(0)^2]=kBT_intro}
\E[m\, v(0)^2]=k_BT.
\end{equation}
In turn, under a harmonic potential as in \eqref{eq:GLE:2D} (with $\gamma > 0$), it was further proven that
\begin{equation}\label{e:equipartition_2D}
\E[\gamma\, x(0)^2]=k_BT =\E[m\, v(0)^2].
\end{equation}
In other words, relations \eqref{e:E[mv(0)^2]=kBT_intro} and \eqref{e:equipartition_2D} show that equipartition of energy holds in each case. More recently~\cite{hohenegger2017equipartition,hohenegger2017fluid,
hohenegger2018reconstructing}, relation \eqref{e:E[mv(0)^2]=kBT_intro} was established in the case of a free particle GLE~\eqref{eq:GLE:1D} assuming the so--named \textit{generalized Rouse} class of memory kernels, i.e.,
\begin{equation}\label{e:finite_sum_of_exp}
\frac{1}{N}\sum^{N}_{n=1} e^{-|t|/\tau_n},
\end{equation}
where $\tau_1 < \hdots < \tau_N$ are called \textit{relaxation times}. For such kernels, Fourier transforms are known in explicit form. This naturally allows for the use of contour integration in the complex plane and the calculation of the second moments of $x(t)$ and $v(t)$.

Note that, for a general potential $U(x)$, the system~\eqref{eq:GLE:2D} is recast in the form
\begin{equation} \label{eq:GLE:2D:potential}
\begin{aligned}
\dot{x}(t) &=v(t),\\
m\,\dot{v}(t)&=-\lambda v(t)-U'(x(t))-\beta\int_{-\infty}^t\close K(t-s)v(s) \d s+\sqrt{\beta k_BT}F(t)+\sqrt{2\lambda k_BT}\,\dot{W}(t).
\end{aligned}
\end{equation}
For several kernel instances having the form of an infinite sums of exponentials, the so--named Mori--Zwanzig formalism \cite{fricks2009time, goychuk2009viscoelastic, ottobre2011asymptotic,zwanzig2001nonequilibrium} can be used to produce a Markovian approximation to~\eqref{eq:GLE:2D:potential} which in turn admits a stationary distribution \cite{glatt2020generalized,pavliotis2014stochastic}. In particular, relation \eqref{e:equipartition_2D} holds under harmonic potentials and kernels $K$ that are either integrable \cite{ottobre2011asymptotic,pavliotis2014stochastic} or exhibit power law decay $K(t)\sim t^{-\alpha}$ for all $\alpha>1/2$ \cite{glatt2020generalized, herzog:mattingly:nguyen:2021}. The question of whether~\eqref{e:equipartition_2D} holds -- even under harmonic potentials -- for $K(t)\sim t^{-\alpha}$, $\alpha\in(0,1/2]$, remains open~\cite{herzog:mattingly:nguyen:2021}.

In this paper, we tackle the problem of establishing equipartition of energy for both~\eqref{eq:GLE:1D} and~\eqref{eq:GLE:2D}. Namely, we show that relations \eqref{e:E[mv(0)^2]=kBT_intro} and \eqref{e:equipartition_2D} hold under the former (see Theorem \ref{thm:equipartition-of-energy:1D:gamma=0:power-law.exponential}) and the latter (see Theorem~\ref{thm:equipartition-of-energy:2D:gamma>0:powerlaw.exponential}) systems, respectively. In both cases, we assume memory kernels either coming from the large class of \emph{completely monotonic} functions (cf.\ Definition~\ref{def:complete-monotone}) or which can be expressed as $\f(t^2)$, where $\f$ is a completely monotonic function. In particular, the former class includes the kernels \eqref{e:t^(-alpha)} and \eqref{e:finite_sum_of_exp}, whereas the latter class includes Gaussian and Cauchy kernels, namely, $e^{-t^2}$ and $(t^2+1)^{-\alpha}$, respectively \cite{spiechowicz2018quantum,spiechowicz2019superstatistics,spiechowicz2021energy}. Besides its great generality, the class of completely monotonic functions is made up of Laplace transforms of positive Radon measures, which is very convenient for the purpose of establishing analytical results (cf.\ Theorem~\ref{thm:completely-monotone:Bernstein}).

The paper is organized as follows. In Section~\ref{sec:preliminaries}, we introduce the notation as well as the assumptions. In Section~\ref{sec:main-results}, we state the main results of the paper, including Theorem~\ref{thm:asymptotic-growth} on the anomalous diffusion of~\eqref{eq:GLE:2D} and Theorem~\ref{thm:equipartition-of-energy:2D:gamma>0:powerlaw.exponential} on equipartition of energy. We address the well--posedness of~\eqref{eq:GLE:2D} as well as the proofs of the main results in Section~\ref{sec:proofs}. In the Appendix, we review the framework of stationary distributions that is employed in the construction of solutions for~\eqref{eq:GLE:2D}. We also recapitulate several properties of Fourier transforms of the memory kernels that are useful in establishing the main theorems. 

\section{Assumptions and preliminaries} \label{sec:preliminaries}
For a function $f:\rbb\to\cbb$, we define the Fourier transform of $f$ and its inverse as
\begin{displaymath}
\widehat{f}(\omega)=\int_\rbb f(t) e^{-\imag t\omega} \d t, \text{ and }	 \check{f}(t)=\frac{1}{2\pi}\int_\rbb f(\omega) e^{\imag t\omega}\d \omega.
\end{displaymath}
We will also make use of the Fourier cosine and sine transforms
$$
\Kcos(\omega)=\int_0^\infty\close K(t)\cos(\omega t)\d t\quad\text{and}\quad\Ksin(\omega)=\int_0^\infty\close K(t)\sin(\omega t)\d t,
$$
where the two integrals are understood in the sense of improper integrals. Let $\Sc$ be the Schwartz space of all smooth functions whose derivatives are rapidly decreasing. Recall that its dual space $\Sc'$ is the so--named class of tempered distributions on $\Sc$. For a given tempered distribution $g\in\Sc'$, we write $\F{g}$ to denote the Fourier transform of $g$ in $\Sc'$. Namely, for all $\varphi\in\Sc$,
\begin{align*}
\langle g,\widehat{\varphi}\rangle = \langle \F{g},\varphi\rangle,
\end{align*}
where $\la g,\widehat{\f} \ra$ denotes the action of a tempered distribution $g$ on a Schwartz function $\widehat{\f}$. It is well known that this transformation is a one--to--one relation in $\Sc'$.

Throughout the paper, we make the following assumptions on the memory kernel (cf.\ \cite{didier2020asymptotic, mckinley2018anomalous}).

\begin{assumption} \label{cond:K:general}
Let $K:\rbb\to \rbb \cup \{\infty\}$ be a real--valued function for $t \neq 0$ and which may be infinite at $t = 0$. We assume that
\begin{itemize}
\item[(I)] \label{cond1}
\begin{itemize}
\item[(a)] \label{cond1a} $K \in L^1_{loc}(\rbb)$ is symmetric around zero and positive for all nonzero $t$;
\item[(b)] \label{cond1b} $K(t) \to 0$ as $t \to \infty$ and is eventually decreasing;
\item[(c)] \label{cond1c} the improper integral $\Kcos(\omega)=\int_0^\infty \! K(t)\cos(\omega t) \, \d t$ is positive for all nonzero $\omega$.
\end{itemize}

\item[(II)] \label{cond2} Furthermore, $K(t)$ satisfies either
\begin{itemize}

\item[(a)]  $K(t)\in L^1(\rbb)$; or
\item[(b)]  $K(t)\sim t^{-1}$ as $t\to\infty$; or
\item[(c)] there exists $\alpha\in(0,1)$ such that $K(t)\sim t^{-\alpha}$ as $t\to\infty$.

\end{itemize}

\end{itemize}
\end{assumption}

Weakly stationary operators generalize stationary distributions in the sense of~\cite{gelfand1955generalized,ito1954stationary}. The conceptual details can be found in Appendix~\ref{sec:stationary-operator}. We now make use of (weakly) stationary operators to construct a weak solution for the system~\eqref{eq:GLE:2D}. The procedure consists in reexpressing the system \eqref{eq:GLE:2D} in terms of operators as applied to test functions, and then extracting (covariance) relations that will enter into the definition of a weak solution. Since \eqref{eq:GLE:2D} is a linear Gaussian system, then such covariance relations fully characterize the weak solution.

We begin by formally multiplying both sides the first equation of~\eqref{eq:GLE:2D} by a test function $\f\in\Sc$. Then, after integration by parts, we obtain
\begin{equation}\label{e:-int_x(t)varphi'(t)dt=int_v(t)varphi(t)dt}
-\int_\rbb x(t)\f'(t)\d t =\int_\rbb v(t)\f(t)\d t.
\end{equation}
Moreover, again by integration by parts,
\begin{equation}\label{e:-int_v(t)varphi'(t)dt=int_v'(t)varphi(t)dt}
\int_\rbb v'(t)\f(t)\d t = -\int_\rbb v(t)\f'(t)\d t.
\end{equation}
Also, for $K(t)$ as in \eqref{eq:GLE:2D} and for a test function $\f\in\Sc$, let
\begin{equation}\label{e:K+(t)}
K^+(t)=K(t)1_{[0,\infty)}(t).
\end{equation}
Then, based on relations \eqref{e:-int_x(t)varphi'(t)dt=int_v(t)varphi(t)dt}, \eqref{e:-int_v(t)varphi'(t)dt=int_v'(t)varphi(t)dt} and \eqref{e:K+(t)}, we can formally write
\begin{align*}
	- m\int_\rbb v(t)\varphi'(t)\d t &= - \lambda \int_{\rbb} v(t) \f(t) \d t -\gamma \int_\rbb x(t)\f(t)\d t- \beta \int_\rbb v(t)\int_\rbb K^+(s)\varphi(t+s)\d s\d t \\
	& \qquad + \sqrt{\beta k_B T} \int_\rbb F(t)\varphi(t)\d t + \sqrt{2\lambda k_B T} \int_{\rbb} \f(t) \d W(t).
\end{align*}
By grouping together terms in $v$ and terms in $x$,
\begin{equation} \label{eq:integration-by-part:v-x}
\begin{aligned}
\int_\rbb v(t)\big(-m\varphi'(t)+\lambda\f(t)&+\beta (K^+*\phitil)(t)\big)\d t+\int_\rbb x(t)\big(\gamma\f(t)\big)\d t  \\
	&=  \sqrt{\beta k_BT} \int_\rbb F(t)\varphi(t)\d t + \sqrt{2\lambda k_BT} \int_{\rbb} \f(t) \d W(t),
\end{aligned}
\end{equation}
where $\phitil(t):=\f(-t)$.

So, let $L^2(\Omega)$ be the space of squared integrable, complex--valued random variables. Also let $\Phi=(X,V):\text{Dom}(\Phi)\subset\Sc'\to L^2(\Omega)^2$ (i.e., $d =2$) be a weakly stationary operator as in Definition~\ref{def:weak-operator}. In the formal relation \eqref{e:-int_x(t)varphi'(t)dt=int_v(t)varphi(t)dt}, we may interpret $X$ and $V$ as operators acting on test functions $\f\in\Sc$. In particular, the (Gaussian) operator $\Phi=(X,V)$ is fully characterized by its covariance structure, which we describe next.

First, note that \eqref{e:-int_x(t)varphi'(t)dt=int_v(t)varphi(t)dt} yields an intrinsic connection between the correlation structures of $X$ and $V$, namely,
\begin{align}\label{e:corr_X_vs_corr_V}
\E\Big[\la X,-\f_1'\ra\overline{\la X,-\f_2'\ra}\Big]&=\E\Big[\la V,\f_1\ra\overline{\la V,\f_2\ra}\Big], \quad \f_1,\,\f_2\in\Sc.
\end{align}
In regard to the cross--correlation between $X$ and $V$, again from the integral equation \eqref{e:-int_x(t)varphi'(t)dt=int_v(t)varphi(t)dt} we obtain
\begin{align}\label{e:corr_X_vs_cross-corr_XandV}
\E\Big[\la X,-\f_1'\ra\overline{\la X,-\f_2'\ra}\Big]&=\E\Big[\la X,-\f_1'\ra\overline{\la V,\f_2\ra}\Big].
\end{align}
Moreover, on the right-hand side of \eqref{eq:integration-by-part:v-x}, the functions $F:\Sc\to L^2(\Omega)$ and $\Wdot:\Sc\to L^2(\Omega)$ are understood as stationary random distributions in the sense of Definition~\ref{def:stationary-distribution}. Their autocorrelation functions are given by, respectively,
\begin{equation} \label{form:Wdot:covariance}
\E\Big[\langle \dot W,\varphi_1\rangle\overline{ \langle \dot W,\varphi_2\rangle}\Big]=\int_\rbb \varphi_1(t)\varphi_2(t)\d t=\frac{1}{2\pi}\int_\rbb\widehat{\f_1}(\omega)\overline{\widehat{\f_2}(\omega)}\d\omega, \\
\end{equation}
and
\begin{equation} \label{form:F:covariance}
\E \Big[\langle F,\varphi_1\rangle\overline{\langle F,\varphi_2\rangle}\Big]  = \int_\rbb K(t)\left( \varphi_1*\widetilde{\varphi}_2 \right)(t)\d t =\frac{1}{2\pi}\int_\rbb 2\Kcos(\omega)\widehat{\f_1}(\omega)\overline{\widehat{\f_2}(\omega)}\d\omega.
\end{equation}
In \eqref{form:F:covariance}, the last equality follows from the fact that $2\Kcos$ is the Fourier transform of $K$ in the sense of distributions (cf.\ Lemma~\ref{lem:temper-distribution:Kcos}). In other words, the spectral measure of $\dot W$ as in Theorem~\ref{thm:weak-stationary-distribution:characterization} is the Lebesgue measure, and that of $K$ is $\pi^{-1}\Kcos(\omega)\d\omega$. Define the operator
\begin{equation}\label{form:-m.phi'+lambda.phi+K^+*phitilde}
\Psi(\f)=-m\varphi'+\lambda\f+\beta (K^+*\phitil).
\end{equation}
Then, we can conveniently recast~\eqref{eq:integration-by-part:v-x} in the form
\begin{equation}\label{eq:integration-by-part:v-x_in_terms_of_operators}
\la V,\Psi(\f)\ra +\la X,\gamma\f\ra = \sqrt{k_BT} \, \la \sqrt{\beta}F+ \sqrt{2\lambda}\dot{W},\f\ra, \quad \f \in {\mathcal S}.
\end{equation}
In particular, relation \eqref{eq:integration-by-part:v-x_in_terms_of_operators} can be used in characterizing the covariance structure of the left-hand side of \eqref{eq:integration-by-part:v-x_in_terms_of_operators} in terms of the covariance structure of the noise terms $F$ and $\dot W$. In other words, for $\f_1,\f_2 \in {\mathcal S}$,
\begin{align}\label{e:LHS_2DGLE_second_moments}
\E\Big[ \big(\la V,\Psi(\f_1)\ra&+\la X,\gamma\f_1\ra\big) \overline{(\la V,\Psi(\f_2)\ra+\la X,\gamma\f_2\ra ) }\Big] \nonumber\\
&= k_BT\,\E\Big[ \la \sqrt{\beta}F+ \sqrt{2\lambda}\dot{W},\f_1\ra \overline{\la \sqrt{\beta}F+ \sqrt{2\lambda}\dot{W},\f_2\ra}\Big].
\end{align}

We now add the standard assumption that the two thermal forcing terms $F$ and $\Wdot$ are uncorrelated.
\begin{assumption}\label{cond:F-and-Wdot:independent}
Let $\Wdot$ and $F$ be the stationary random distributions as in Definition~\ref{def:stationary-distribution} whose covariance functions are given by~\eqref{form:Wdot:covariance} and \eqref{form:F:covariance}. $F$ and $\Wdot$ are uncorrelated, i.e., for all $\f_1,\,\f_2\in\Sc$,
\begin{align*}
\E[\la \Wdot,\f_1\ra\overline{\la F,\f_2\ra}]=0.
\end{align*}
\end{assumption}

In light of relations \eqref{e:corr_X_vs_corr_V}, \eqref{e:corr_X_vs_cross-corr_XandV}, \eqref{e:LHS_2DGLE_second_moments}, as well as of  Assumption~\ref{cond:F-and-Wdot:independent}, we are now in a position to define weak solutions for~\eqref{eq:GLE:2D}.
\begin{definition} \label{def:weak-solution}
Under Assumptions \ref{cond:K:general} and \ref{cond:F-and-Wdot:independent}, let $\Phi=(X,V):\emph{\dom}(\Phi)\subset\Sc'\to L^2(\Omega)^2$ be a stationary operator as in Definition~\ref{def:weak-operator}. Then $\Phi=(X,V)$ is called a weak stationary solution for equation~\eqref{eq:GLE:2D} if the following conditions are satisfied.
\begin{itemize}

\item[(a)] For all $\varphi\in\Sc$, $\E|\la V,\Psi(\f)\ra|^2<\infty$, where $\Psi(\f)$ is the transformation as in~\eqref{form:-m.phi'+lambda.phi+K^+*phitilde}.

\item[(b)] For any $\f_1,\f_2\in \Sc$,
\begin{align}\label{form:covariance:E[X,X]=E[V,V]=E[X,V]}
\E\Big[\la X,-\f_1'\ra\overline{\la X,-\f_2'\ra}\Big]&=\E\Big[\la V,\f_1\ra\overline{\la V,\f_2\ra}\Big]=\E\Big[\la X,-\f_1'\ra\overline{\la V,\f_2\ra}\Big],
\end{align}
\begin{equation} \label{form:covariance:E[V+X,phi]=E[F+W,phi]}
\begin{aligned}
\text{and}\qquad\E\Big[ \big(\la V,\Psi(\f_1)\ra&+\la X,\gamma\f_1\ra\big) \overline{\la V,\Psi(\f_2)\ra+\la X,\gamma\f_2\ra }\Big]\\&= k_BT\,\E\Big[ \la \sqrt{\beta}F+ \sqrt{2\lambda}\dot{W},\f_1\ra \overline{\la \sqrt{\beta}F+ \sqrt{2\lambda}\dot{W},\f_2\ra}\Big]\\
&=k_BT\Big(\E\Big[ \beta\la F,\f_1\ra \overline{\la F,\f_2\ra}\Big]+\E\Big[2\lambda \la \dot{W},\f_1\ra \overline{\la \dot{W},\f_2\ra}\Big]\Big).
\end{aligned}
\end{equation}
\end{itemize}
\end{definition}

In Section~\ref{sec:proofs}, we show that, for a weak stationary solution $\Phi=(X,V)$ of~\eqref{eq:GLE:2D}, its spectral densities can be computed explicitly, as pointed out in~\cite{kou2008stochastic}. In other words, let
\begin{align}
r_{11}(\omega)&=\frac{2\lambda+\beta\Khat(\omega)}{\big|\gamma-m\omega^2+\imag \omega(\lambda+\beta\widehat{K^+}(\omega))\big|^2},
 \label{form:spectral-density:x}\\
 r_{22}(\omega)&=\omega^2\,r_{11}(\omega),
 \label{form:spectral-density:v}\\
\text{and}\qquad r_{12}(\omega)&=\overline{r_{21}(\omega)}=\imag \omega\, r_{11}(\omega) \label{form:spectral-density:xv}.
\end{align}
Then, by Lemma~\ref{lem:r_(ij):L1}, there exists a unique stationary operator $\Phi$, cf.\ Definition~\ref{def:weak-operator}, associated with a $2\times 2$ Hermitian positive definite matrix of measures $\nu$ such that
$$
\nu(\d\omega)= k_BT(2\pi)^{-1}(r_{ij}(\omega)\d\omega)_{1\le i,j\le 2}
$$
for $r_{ij}$ as in \eqref{form:spectral-density:x}, \eqref{form:spectral-density:v} or \eqref{form:spectral-density:xv}.\\

For results on equipartition of energy, as mentioned in the Introduction, we consider kernels that are related to the so--named class of completely monotonic functions, denoted by $\CM$. We recall their definition next.
\begin{definition}\label{def:complete-monotone}
A function $K:(0,\infty)\to[0,\infty)$ is called completely monotonic if $K \in C^\infty(0,\infty)$ and $(-1)^n K^{(n)}(t)\geq 0$ for all $n\geq 0$, $t>0$.
\end{definition}

So, we make the following additional assumption on the memory kernels.
\begin{assumption} \label{cond:CM}
Let $K:\rbb\to \rbb \cup \{\infty\}$ be a real--valued function for $t \neq 0$ and which may be infinite at $t = 0$. We assume that either
\begin{itemize}
\item[(a)] $K\in \CM$; or

\item[(b)] $K(t)=\f(t^2)$, where $\f\in \CM$.
\end{itemize}

\end{assumption}
As briefly discussed in the Introduction, the former class includes exact power-law and sum-of-exponential kernels as in~\eqref{e:t^(-alpha)} and~\eqref{e:finite_sum_of_exp}, respectively, whereas the latter class includes Gaussian and Cauchy kernels, namely, $e^{-t^2}$ and $(t^2+1)^{-\alpha}$, respectively \cite{spiechowicz2018quantum,spiechowicz2019superstatistics,spiechowicz2021energy}.  Besides the broad scope of the $\CM$ class, dealing with completely monotonic functions involves the technically convenient fact that they can be represented as Laplace transforms of Radon measures on $[0,\infty)$ (cf.\ Theorem~\ref{thm:completely-monotone:Bernstein}). As a consequence, one is able to express the Fourier transforms of the memory kernels described in Assumption~\ref{cond:CM} based on the Radon measures (cf.\ Lemmas~\ref{lem:complete-monotone:Kcos-Ksin} and~\ref{lem:complete-monotone:Kcos-Ksin:K(t)=phi(t^2)}). For this reason, we are able to extend these transforms to the complex plane and calculate contour integrals involving the completely monotonic functions in question.

\section{Main results}\label{sec:main-results}

In this section, we state the main results of the paper. In Theorem \ref{thm:well-posed}, we establish the existence of weakly stationary solutions for~\eqref{eq:GLE:2D}. In Theorem \ref{thm:asymptotic-growth}, we characterize the mean squared displacement of $\int^{t}_{0}x(s)ds$ and $\int^{t}_{0}v(s)ds$ for weak solutions of ~\eqref{eq:GLE:2D}. Starting from the broad class of completely monotonic kernels, in Theorems \ref{thm:equipartition-of-energy:1D:gamma=0:power-law.exponential} and \ref{thm:equipartition-of-energy:2D:gamma>0:powerlaw.exponential}, respectively, we establish the equipartition relation in the GLE framework for free particles or particles under a harmonic potential.

We start off with the existence of solutions.
\begin{theorem}\label{thm:well-posed}
Under Assumptions~\ref{cond:K:general} and~\ref{cond:F-and-Wdot:independent}, then $\Phi=(X,V)$ is a weakly stationary solution of~\eqref{eq:GLE:2D} as in Definition~\ref{def:weak-solution} if and only if the spectral measure $\nu(\d\omega)= k_BT(2\pi)^{-1}(r_{ij}(\omega)\d\omega)_{1\le i,j\le 2}$ is given by relations~\eqref{form:spectral-density:x}, \eqref{form:spectral-density:v} and \eqref{form:spectral-density:xv}.
\end{theorem}
\noindent The proof of Theorem~\ref{thm:well-posed} is discussed in detail in Section~\ref{sec:wellposed}.

\begin{remark}\label{rem:weak-solution:gamma=0} When $\gamma=0$, equation~\eqref{eq:GLE:2D} is reduced to~\eqref{eq:GLE:1D}, whose weak solution $V:\emph{Dom}(V)\subset \Sc'\to L^2(\Omega)$ is defined as satisfying a relation similar to~\eqref{form:covariance:E[V+X,phi]=E[F+W,phi]}, namely,
\begin{align*}
\E\Big[ \la V,\Psi(\f_1)\ra\overline{\la V,\Psi(\f_2)\ra }\Big]=k_BT\Big(\E\Big[ \beta\la F,\f_1\ra \overline{\la F,\f_2\ra}\Big]+\E\Big[2\lambda \la \dot{W},\f_1\ra \overline{\la \dot{W},\f_2\ra}\Big]\Big).
\end{align*}
The existence of such $V$ for equation~\eqref{eq:GLE:1D} was previously studied in \cite{didier2020asymptotic,mckinley2018anomalous}. In particular, the spectral measure of $V$ is also given by $r_{22}$ as in~\eqref{form:spectral-density:v} with $\gamma=0$.
\end{remark}

Next, we turn to the topic of characterizing of the anomalously diffusive behavior of solutions to~\eqref{eq:GLE:1D}. For this purpose, we consider the integrated bivariate process resulting from the solutions encountered in Theorem \ref{thm:well-posed}. More precisely, in view of Lemma~\ref{lem:r_(ij):L1} (see Section \ref{lem:r_(ij):L1}) together with Remark~\ref{rem:weak-oprator:point-process}, since $r_{11}$ and $r_{22}$ are both integrable, we can define the bivariate process $(x(t),v(t))$ associated with the weak stationary solution $(X,V)$ as in Definition~\ref{def:form:u(t)}. Namely, we set
\begin{equation} \label{form:x(t)=<X,delta_t>}
x(t):=\la X,\delta_t\ra\quad\text{and}\quad v(t):=\la V,\delta_t\ra,
\end{equation}
where $\delta_t$ is the Dirac $\delta$ distribution centered at $t$. Moreover, it can be shown that $(x(t),v(t))$ is a $\bbR^2$--valued process and has a continuous modification (see Lemma~\ref{lem:x(t)-v(t):continuous}). It follows that we may define the integrals $\int_0^t x(s)\d s$ and $\int_0^t v(s)\d s$ in the usual Riemann-Lebesgue sense. Note that these integrals do agree with $\la X,1_{[0,t]}\ra$ and $\la V,1_{[0,t]}\ra$, respectively (see Remark~\ref{rem:int_0^t.x(s)ds=<X,1_[0,t]>}).

As explained in the Introduction, for the case of a free particle as in~\eqref{eq:GLE:1D} ($\gamma=0$), for a large class of memory kernels the process $v(t)$ may either be diffusive or subdiffusive \cite{didier2020asymptotic,mckinley2018anomalous} depending on the asymptotic decay of $K(t)$ as $t\to\infty$. In contrast, the process $x(t)$ defined in~\eqref{form:x(t)=<X,delta_t>} may be either diffusive or \textit{super}diffusive. This is all precisely stated in the following theorem.
\begin{theorem}\label{thm:asymptotic-growth}
Let $(x(t),v(t))$ be the bivariate process associated with $(X,V)$ the weak stationary solution of~\eqref{eq:GLE:2D} as in Theorem~\ref{thm:well-posed}. Then, under Assumptions~\ref{cond:K:general} and~\ref{cond:F-and-Wdot:independent}, the following holds.
\begin{itemize}
\item[(a)] For all $t\in\rbb$, $\E\big[\int_0^t x(s)\d s\int_0^tv(s)\d s\big]=0$.

\item[(b)] As $t\to\infty$, $\E\big|\int_0^t v(s)\d s\big|^2\to2\,\E|x(0)|^2$.

\item[(c)] $
\text{If } K(t)\begin{cases}\in L^1(\rbb),\\
\sim t^{-1},\quad t\to\infty,\\
\sim t^{-\alpha},\,\alpha\in(0,1),\quad t\to\infty,\end{cases}\hspace{-0.3cm}\text{then  }\E\big|\int_0^t x(s)\d s\big|^2\sim\begin{cases} t,\\ t\,\log (t),\\ t^{2-\alpha} \end{cases}\hspace{-0.3cm}\text{as  } t\to\infty.
 $
\end{itemize}
\end{theorem}
The claim in Theorem~\ref{thm:asymptotic-growth}, (a), is not surprising in view of the fact that, for several other GLE instances in stationarity, $x(t)$ is uncorrelated with $v(t)$ \cite{glatt2020generalized,pavliotis2014stochastic}. Also, the appearance of $x(0)$ in Theorem~\ref{thm:asymptotic-growth}, (b), may be intuitively explained based on the observation that $v(t)$ can be regarded as the derivative of $x(t)$. Thus, formally,
$$
\E \Big|\int_0^t v(s)\d s\Big|^2= \E |x(t)-x(0)|^2 \sim 2\E|x(0)|^2, \quad t \rightarrow \infty,
$$
where the asymptotic equivalence is a consequence of the fact that $x(t)$ is Gaussian, (weakly) stationary and mixing \cite{rosinski:zak:1996}. However, note that establishing the asymptotic growth of $\E|\int_0^t x(s)\d s|^2$ requires a careful characterization of the spectral density $r_{11}$ in terms of the asymptotics of $K(t)$. The proof of Theorem~\ref{thm:asymptotic-growth} can be found in Section~\ref{sec:asymptotic.growth}.

We now turn to equipartition of energy. First, we discuss the case $\gamma=0$, namely, a free particle as defined by equation~\eqref{eq:GLE:1D}. In what follows, we state the result for~\eqref{eq:GLE:1D} under kernels $K$ either in the class $\CM$ or such that $K(t) = \varphi(t^2)$, $\varphi\in\CM$. As discussed in the Introduction, this generalizes the results in \cite[Formula (2.7)]{hohenegger2017equipartition} and \cite[Theorem 4.1]{kou2008stochastic}.
\begin{theorem} \label{thm:equipartition-of-energy:1D:gamma=0:power-law.exponential}
Suppose that $\gamma=0$ and that Assumptions~\ref{cond:K:general},~\ref{cond:F-and-Wdot:independent} and~\ref{cond:CM} are satisfied. Let $v(t)$ be the process associated with $V$, a weak solution for~\eqref{eq:GLE:1D} in the sense of Remark~\ref{rem:weak-solution:gamma=0}. Then,
\begin{align}\label{e:E[m,v(0)^2]=k_BT}
\E[m \, v(0)^2] = k_BT.
\end{align}
\end{theorem}

The proof of Theorem~\ref{thm:equipartition-of-energy:1D:gamma=0:power-law.exponential} can be found in Section~\ref{sec:equipartition:1D:gamma=0}.

In the following theorem, we describe the analogous result for the case $\gamma > 0$, namely, a harmonically bounded particle as defined by~\eqref{eq:GLE:2D}. Its proof is presented in Section~\ref{sec:equipartition:2D:gamma>0}.

\begin{theorem} \label{thm:equipartition-of-energy:2D:gamma>0:powerlaw.exponential}
Suppose that $\gamma>0$ and that Assumptions~\ref{cond:K:general},~\ref{cond:F-and-Wdot:independent} and~\ref{cond:CM} are satisfied. Let $(x(t),v(t))$ be the process associated with $\Phi=(X,V)$, a weak solution for~\eqref{eq:GLE:2D} as in~\eqref{form:x(t)=<X,delta_t>}. Then,
\begin{align*}
\E[\gamma\, x(0)^2] =\E[m\, v(0)^2] = k_BT.
\end{align*}
\end{theorem}

\section{Proofs of the main results}\label{sec:proofs}
Throughout the rest of the paper, $c$ denotes a generic positive constant. The main parameters that $c$ depends on will appear between parenthesis, e.g., $c(T,q)$ is a function of $T$ and $q$.
\subsection{Wellposedness} \label{sec:wellposed}
In this section, we provide the proof of Theorem~\ref{thm:well-posed} giving the existence of weak solution for~\eqref{eq:GLE:2D}. We start with the following result, which asserts that $\{r_{ij}\}$ is a spectral density of a weak stationary operator $\Phi$.
\begin{lemma}\label{lem:r_(ij):L1}
Let $\nu(\d\omega)=k_BT(2\pi)^{-1}(r_{ij}(\omega)\d\omega)_{1\le i,j\le 2}$ where $r_{ij}$ is as in~\eqref{form:spectral-density:x}, \eqref{form:spectral-density:v} or \eqref{form:spectral-density:xv}. Then, $\nu$ is the spectral measure of a stationary operator as in Definition~\ref{def:weak-operator}.
\end{lemma}
\begin{proof}
By symmetry, the Fourier transform $\widehat{K}$ of $K$ satisfies $\widehat{K}=2\Kcos$. Thus, we can rewrite $r_{11}$ in~\eqref{form:spectral-density:x} as
\begin{equation} \label{form:spectral-density:x:Kcos-Ksin}
r_{11}(\omega) = \frac{2(\lambda+\beta
\Kcos(\omega))}{|\gamma-m\omega^2+\beta\omega\Ksin(\omega)|^2+\omega^2|\lambda+\beta\Kcos(\omega)|^2}.
\end{equation}
By Assumption~\ref{cond:K:general} (I) (c), $\Kcos$ is positive. Hence, $r_{11}$ is positive a.e., and so is $r_{22}(\omega)=\omega^2 r_{11}(\omega)$. In view of \eqref{form:spectral-density:x}--\eqref{form:spectral-density:xv}, $\nu$ is a Hermitian nonnegative definite matrix a.e.

Next, we claim that both $r_{11}$ and $r_{22}$ are integrable. To see this, by symmetry again, we only need to consider $\omega\in [0,\infty)$. In addition, due to continuity, we only need to check integrability at $\omega\to\infty$ and around the origin. On one hand, as $\omega\to\infty$, we invoke~\eqref{lim:Kcos-Ksin:omega->infinity} to conclude that $\Kcos(\omega)$ and $\Ksin(\omega)$ converge to zero. It follows that $r_{11}(\omega)$ is dominated by $\omega^{-4}$, which also implies that $r_{22}$ is dominated by $\omega^{-2}$. This proves integrability at infinity. On the other hand, when $\omega$ is near the origin, there are three cases to be considered, depending on the behavior of $K(t)$.\\

\noindent \textbf{Case 1}: $K$ is integrable, cf. Assumption~\ref{cond:K:general} (II) (a). By virtue of Lemma~\ref{lem:Kcos-Ksin:omega->0} (a), it is clear that
\begin{align} \label{lim:r_[11].r_[22]:omega->0:diffusive}
r_{11}(\omega)\to \frac{2\lambda}{\gamma^2}\quad\text{and}\quad r_{22}(\omega)\to 0,\quad\text{as}\quad\omega\to 0.
\end{align}

\noindent \textbf{Case 2}: $K\sim t^{-1}$ as $t\to\infty$, cf. Assumption~\ref{cond:K:general} (II) (b). From~\eqref{form:spectral-density:x:Kcos-Ksin}, we have
\begin{align*}
\frac{r_{11}(\omega)}{|\log(\omega)|} =\frac{2\lambda/|\log(\omega)|+2\beta\Kcos(\omega)/|\log (\omega)|}
{|\gamma-m\omega^2+\beta\omega\Ksin(\omega)|^2+\omega^2|\lambda+\beta\Kcos(\omega)|^2}.
\end{align*}
By~\eqref{lim:Kcos-Ksin:omega->0:critial},
\begin{align} \label{lim:r_[11].r_[22]:omega->0:critical}
r_{11}(\omega)\sim |\log(\omega)|\quad\text{and}\quad r_{22}(\omega)\to 0,\quad\text{as }\omega\to 0.
\end{align}

\noindent \textbf{Case 3}: For some $\alpha\in(0,1)$, $K\sim t^{-\alpha}$ as $t\to\infty$, cf. Assumption~\ref{cond:K:general} (II) (c). Similarly to Case 2, from~\eqref{form:spectral-density:x:Kcos-Ksin}, we obtain
\begin{align*}
\frac{r_{11}(\omega)}{\omega^{\alpha-1}} =\frac{2\lambda\omega^{1-\alpha}+2\beta\omega^{1-\alpha}\Kcos(\omega)}
{|\gamma-m\omega^2+\beta\omega\Ksin(\omega)|^2+\omega^2|\lambda+\beta\Kcos(\omega)|^2}.
\end{align*}
In view of~\eqref{lim:Kcos-Ksin:omega->0:subdiffusion},
\begin{align} \label{lim:r_[11].r_[22]:omega->0:subdiffusive}
r_{11}(\omega)\sim \omega^{\alpha-1}\quad\text{and}\quad r_{22}(\omega)=\omega^2 r_{11}(\omega)\to 0,\quad\text{as }\omega\to 0.
\end{align}
In all three cases, both $r_{11}$ and $r_{22}$ are integrable near the origin. Since they are also integrable at $\infty$, they are integrable on $\rbb$, as claimed.

As a consequence, in view of~\eqref{form:spectral-density:xv}, by Cauchy-Schwarz inequality
\begin{align*}
\int_\rbb |r_{12}(\omega)|\d\omega=\int_\rbb |r_{21}(\omega)|\d\omega\le \Big(\int_{\rbb}r_{11}(\omega)\d\omega\Big)^{1/2}\Big(\int_\rbb r_{22}(\omega)\d\omega\Big)^{1/2}<\infty.
\end{align*}
It follows that $k_BT(2\pi)^{-1} (r_{ij})_{1\le i,j\le 2}$ satisfies inequality~\eqref{cond:weak-stationary-distribution:int.f(omega)/(1+omega^2)^p<infty} with $p=0$. By virtue of Theorem~\ref{thm:weak-stationary-distribution:characterization}, this implies the existence of a unique stationary distribution $G$ whose spectral density is $k_BT(2\pi)^{-1} (r_{ij})$. Furthermore since $\nu$ is absolutely continuous with respect to the Lebesgue measure, there exists a unique weak stationary operator $\Phi:\dom(\Phi)\subset\Sc'\to L^2(\Omega)^2$ extending $G$ as in Definition~\ref{def:weak-operator}. Thus, the proof is complete.
\end{proof}

Theorem~\ref{thm:well-posed} asserts that $\Phi$ is, indeed, a weak solution of \eqref{eq:GLE:2D}. The argument is based on that of~\cite[Theorem~4.5]{mckinley2018anomalous} tailored to our setting.

\begin{proof}[Proof of Theorem~\ref{thm:well-posed}] ($\Rightarrow$) Let $\Phi=(X,V)$ be a stationary operator associated with a spectral measure $\nu(\d\omega)=k_BT(2\pi)^{-1} (r_{ij}(\omega)\d\omega)_{1\le i,j\le 2}$. Suppose $\Phi$ is a weak solution for \eqref{eq:GLE:2D}. For $\f\in\Sc$, consider $\Psi(\f)$ as in~\eqref{form:-m.phi'+lambda.phi+K^+*phitilde}. Its Fourier transform in $\Sc'$ is given by
\begin{align} \label{form:-m.phi'+lambda.phi+K^+*phitilde:Fourier-transform}
\F{\Psi(\f)}&= \F{-m\varphi'+\lambda\f+\beta (K^+*\phitil)} =(\overline{\imag m\omega+\lambda+\beta\widehat{K^+}})\cdot\widehat{\varphi}.	
\end{align}
For any $\f_1,\,\f_2\in\Sc$, in view of~\eqref{form:covariance:E[X,X]=E[V,V]=E[X,V]} together with~\eqref{form:stationary-distribution:E[<G,phi1><G,phi2>]} for stationary operators, we have
\begin{align}\label{e:c*int_w^2_phi1_phi2_r11}
\frac{k_BT}{2\pi}\int_\rbb \omega^2\widehat{\f_1}(\omega)\overline{\widehat{\f_2}(\omega)}r_{11}(\omega)\d\omega&=\frac{k_BT}{2\pi}\int_\rbb \widehat{\f_1}(\omega)\overline{\widehat{\f_2}(\omega)}r_{22}(\omega)\d\omega \nonumber\\
&=-\imag \frac{k_BT}{2\pi}\int_\rbb \omega\widehat{\f_1}(\omega)\overline{\widehat{\f_2}(\omega)}r_{12}(\omega)\d\omega.
\end{align}
Recall that the Fourier transform is an automorphism on ${\mathcal S}$ \cite{strichartz2003guide}. Hence, we can rewrite \eqref{e:c*int_w^2_phi1_phi2_r11} as
\begin{align}\label{e:c*int_w^2_phi1_phi2_r11_after_Fourier}
\int_\rbb \omega^2\f_1(\omega)\f_2(\omega)r_{11}(\omega)\d\omega=\int_\rbb \f_1(\omega)\f_2(\omega)r_{22}(\omega)\d\omega=-\imag \int_\rbb \omega\f_1(\omega)\f_2(\omega)r_{12}(\omega)\d\omega.
\end{align}
Since \eqref{e:c*int_w^2_phi1_phi2_r11_after_Fourier} holds for any $\f_1,\,\f_2\in\Sc$, we conclude that, a.e.,
\begin{equation} \label{eq:omega^2.r_[11]=r_[22]}
r_{22}(\omega)=\omega^2 r_{11}(\omega), \quad r_{12}(\omega)=\imag \omega r_{11}(\omega)\quad\text{and}\quad r_{21}(\omega)=-\imag \omega r_{11}(\omega).
\end{equation}
Note that, in \eqref{eq:omega^2.r_[11]=r_[22]}, the last equality follows from the fact that $\nu$ is a Hermitian measure, so that $r_{21}=\overline{r_{12}}$.

It remains to show that $r_{11}$ is given by~\eqref{form:spectral-density:x} or, equivalently, by \eqref{form:spectral-density:x:Kcos-Ksin}. On one hand, by~\eqref{form:-m.phi'+lambda.phi+K^+*phitilde:Fourier-transform}, \eqref{eq:omega^2.r_[11]=r_[22]} and a simple calculation,
\begin{align*}
\MoveEqLeft[5]\E\Big[\big(\la V,\Psi(\f_1)\ra+\la X,\gamma \f_1\ra\big)\overline{(\la V,\Psi(\f_2)\ra+\la X,\gamma \f_2\ra)}\Big]\\
&= \frac{k_BT}{2\pi}\int_{\rbb} \big|\gamma-m\omega^2+\imag \omega(\lambda+\beta\widehat{K^+}(\omega))\big|^2\widehat{\f_1}(\omega)\overline{\widehat{\f_2}(\omega)}r_{11}(\omega)\d\omega.
\end{align*}
On the other hand, together with~\eqref{form:Wdot:covariance} and~\eqref{form:F:covariance}, the zero correlation assumption between $F$ and $\Wdot$ (see Assumption~\ref{cond:F-and-Wdot:independent}) implies that
\begin{align*}
\MoveEqLeft[5]k_BT\,\E\Big[ \la \sqrt{\beta}F+ \sqrt{2\lambda}\dot{W},\f_1\ra \overline{\la \sqrt{\beta}F+ \sqrt{2\lambda}\dot{W},\f_2\ra}\Big]\\
&=k_BT\beta\E\Big[ \la F,\f_1\ra \overline{\la F,\f_2\ra}\Big]+2k_BT\lambda\E\Big[ \la \dot{W},\f_1\ra \overline{\la \dot{W},\f_2\ra}\Big]\\&=\frac{k_BT}{2\pi}\int_\rbb \Big(2\lambda+\beta\widehat{K}(\omega)\Big)\widehat{\varphi_1}(\omega) \overline{\widehat{\f_2}(\omega) }\d\omega.
\end{align*}
In view of relation~\eqref{form:covariance:E[V+X,phi]=E[F+W,phi]}, for all $\f_1,\,\f_2\in\Sc$ we readily obtain
\begin{align*}
\int_{\rbb} \big|\gamma-m\omega^2+\imag \omega(\lambda+\beta\widehat{K^+}(\omega))\big|^2\widehat{\f_1}(\omega)\overline{\widehat{\f_2}(\omega)}r_{11}(\omega)\d\omega=\int_\rbb \left(2\lambda+\beta\widehat{K}(\omega)\right)\widehat{\varphi_1}(\omega) \overline{\widehat{\f_2}(\omega) }\d\omega.
\end{align*}
It follows that~\eqref{form:spectral-density:x} holds, namely,
\begin{align*}
r_{11}(\omega)=\frac{2\lambda+\beta\widehat{K}(\omega)}{\big|\gamma-m\omega^2+\imag \omega(\lambda+\beta\widehat{K^+}(\omega))\big|^2} \quad \textnormal{a.e.}
\end{align*}

($\Leftarrow$) Suppose $\Phi=(X,V)$ is the weakly stationary operator whose spectral density is given by~\eqref{form:spectral-density:x}--\eqref{form:spectral-density:xv}. We first check condition (a) in Definition \ref{def:weak-solution}.
In fact, by \eqref{form:-m.phi'+lambda.phi+K^+*phitilde:Fourier-transform},  
\begin{align}\label{e:E|<V,Psi(phi)>|^2}
\E|\la V,\Psi(\f)\ra|^2=\frac{k_BT}{2\pi}\int_\rbb |\imag m\omega+\lambda+\beta\widehat{K^+}(\omega)|^2|\widehat{\f}(\omega)|^2\omega^2 r_{11}(\omega)\d\omega.
\end{align}
Similarly to the proof of Lemma~\ref{lem:r_(ij):L1}, it suffices to consider the integrand in \eqref{e:E|<V,Psi(phi)>|^2} as $\omega$ tends to infinity and for $\omega$ around the origin. On one hand, since $r_{11}\sim \omega^{-4}$ as $\omega\to\infty$, it is clear that the integrand \eqref{e:E|<V,Psi(phi)>|^2} is dominated by $\widehat{\f}$, which is integrable. On the other hand, in view of Lemma~\ref{lem:Kcos-Ksin:omega->0}, $|\imag m\omega+\lambda+\beta\widehat{K^+}(\omega)|^2|\omega^2$ tends to zero as $\omega\to 0$. It follows that, around the origin, the integrand is dominated by $ r_{11}(\omega)$ which is integrable (see the proof of Lemma~\ref{lem:r_(ij):L1}).

To verify condition (b) in Definition~\ref{def:weak-solution}, one can adapt the calculation in part (a) so as to arrive at~\eqref{form:covariance:E[X,X]=E[V,V]=E[X,V]} and~\eqref{form:covariance:E[V+X,phi]=E[F+W,phi]}. The proof is thus complete.
\end{proof}

\begin{lemma} \label{lem:x(t)-v(t):continuous}
Let $(x(t),v(t))=\la \Phi,\delta_t\ra$ be the stochastic process defined by~\eqref{form:x(t)=<X,delta_t>}. Then $(x(t),v(t))$ is a well defined real stationary bivariate process. Moreover, $(x(t),v(t))$ has a continuous modification.
\end{lemma}
\begin{proof}
Establishing that the bivariate stochastic process $(x(t),v(t))$ is well defined is equivalent to showing that $\delta_t\in\dom(\Phi)$. In turn, the latter is equivalent to proving that $r_{11}$ and $r_{22}$ are integrable, cf.\ Remark~\ref{rem:weak-oprator:point-process}, which is established in the proof of Lemma~\ref{lem:r_(ij):L1}. In addition, since $r_{11}$ and $r_{22}$ are even functions, $x(t)$ and $v(t)$ are, indeed, real--valued (weakly) stationary processes \cite{ito1954stationary}.

Recall that, by \cite[Chapter 9.3]{cramer1967stationary}, if there exists a constant $a>3$ such that
\begin{equation}\label{e:cramer_and_leadbetter}
\int_0^\infty\close \left[\log(1+\omega)\right]^a\big( r_{11}(\omega)+r_{22}(\omega)\big) \d\omega <\infty,
\end{equation}
then $(x(t),v(t))$ has a continuous modification. In fact, following the proof of Lemma~\ref{lem:r_(ij):L1}, $r_{11}$ and $r_{22}$ are dominated by $\omega^{-4}$ and $\omega^{-2}$, respectively, as $\omega \rightarrow \infty$. Also, both functions are integrable around the origin. As a consequence, \eqref{e:cramer_and_leadbetter} does hold for any $a>3$. Therefore, $(x(t),v(t))$ has a continuous modification, as claimed.
\end{proof}

\begin{remark}\label{rem:int_0^t.x(s)ds=<X,1_[0,t]>} Since the bivariate, stationary stochastic process $(x(t),v(t))$ has a continuous modification, then we can define the integral $\int_0^t (x(s),v(s))\d s$ in the usual Riemann--Lebesgue sense. However, integration over $t$ may also be defined by means of the action $\la (X,V),1_{[0,t]}\ra$. Moreover, it can be shown that, for all $t\geq 0$,
$$
\E\Big[\Big(\int_0^t (x(s),v(s))\d s-\la (X,V),1_{[0,t]}\ra\Big)^*\Big(\int_0^t (x(s),v(s))\d s-\la (X,V),1_{[0,t]}\ra\Big)\Big]=0.
$$
This implies that, for every $t\ge 0$, these two notions of integration agree a.s.
\end{remark}

\subsection{Anomalous diffusion of $(x(t),v(t))$} \label{sec:asymptotic.growth}
In this section, we prove Theorem~\ref{thm:asymptotic-growth} on the asymptotic behavior of $\int_0^t(x(s), v(s))\d s$. While the result for $\E|\int_0^t v(s)\d s|^2$ and the cross--covariance between $x(t)$ and $v(t)$ are relatively straightforward, the asymptotics of $\bbE|\int_0^t x(s)\d s|^2$ requires a more careful analysis depending on three cases of $K$ as in Assumption~\ref{cond:K:general} (II). The approach that we are going to employ is similar to those in \cite[Section 5]{didier2020asymptotic} and \cite[Section 6]{mckinley2018anomalous}. For the reader's convenience, we first summarize the method to characterize the growth rate of $\E|\int_0^t x(s)\d s|^2$.\vspace{0.3cm}

\noindent \textbf{Step 1}: we relate the large (time) scale behavior of the memory $K$ to the behavior of $\Kcos(\omega)$ and $\Ksin(\omega)$ as $\omega\to 0$. This result appears in Lemma~\ref{lem:Kcos-Ksin:omega->0}. \vspace{0.1cm}

\noindent \textbf{Step 2}: similarly to the proof of Lemma~\ref{lem:r_(ij):L1}, we obtain the near--zero behavior of the spectral densities $r_{11}(\omega)$, the spectral density for $x(t)$ as in~\eqref{form:spectral-density:x} through that of $\Kcos(\omega)$ and $\Ksin(\omega)$ as $\omega\to 0$;\vspace{0.1cm}

\noindent \textbf{Step 3}: the behavior of $r_{11}(\omega)$ as $\omega \rightarrow 0$ and the Dominated Convergence Theorem are used to characterize the asymptotic growth of $\E\big|\int_0^tx(s)\d s\big|^2$.

\begin{proof}[Proof of Theorem~\ref{thm:asymptotic-growth}]

(a) Recall that $r_{12}=\imag \omega r_{11}(\omega)$ by relation~\eqref{form:spectral-density:xv}. By~\eqref{form:stationary-distribution:E[<G,phi1><G,phi2>]} for the operator $\Phi$,
\begin{align}\label{e:cross-corr_E_int_x_int_y}
\E\Big[\int_0^t x(s)\d s\overline{\int_0^t v(y)\d y}\Big]&=\int_0^t\int_0^t\E\big[\la X,\delta_s\ra\overline{\la V,\delta_y\ra}\big]\d s\d y \nonumber\\
&=\frac{k_BT}{2\pi}\int_0^t\int_0^t\int_\rbb e^{-\imag (s-y)\omega}r_{12}(\omega)\d\omega \d s\d y \nonumber\\
&=\frac{k_BT}{2\pi}\int_0^t\int_0^t\int_\rbb e^{-\imag (s-y)\omega}\imag \omega  r_{11}(\omega)\d\omega \d s\d y\nonumber\\
&=\frac{k_BT}{2\pi}\int_\rbb \Big|\frac{e^{\imag t\omega}-1}{\omega}\Big|^2 \imag \omega r_{11}(\omega)\d\omega=0.
\end{align}
The last equality in \eqref{e:cross-corr_E_int_x_int_y} is a consequence of the fact that the integrand is an odd function. This establishes (a).

(b) Similarly to part (a), we compute the second moment of $\int_0^t v(s)\d s$ using formula $r_{22}=\omega^2 r_{11}$ as in~\eqref{form:spectral-density:v} and covariance function~\eqref{form:stationary-distribution:E[<G,phi1><G,phi2>]} for $\Phi$. In fact,
\begin{align*}
\E\Big|\int_0^t v(s)\d s\Big|^2&=\int_0^t\int_0^t\E\big[\la V,\delta_s\ra\overline{\la V,\delta_y \ra}\big]\d s\d y\\
&=\frac{k_BT}{2\pi}\int_0^t\int_0^t\int_\rbb e^{-\imag (s-y)\omega}\omega^2  r_{11}(\omega)\d\omega \d s\d y\\
&= \frac{k_BT}{2\pi}\int_\rbb 2(1-\cos(t\omega)) r_{11}(\omega)\d\omega\\
&=2\,\E|x(0)|^2-\frac{k_BT}{\pi}\int_\rbb \cos(t\omega)r_{11}(\omega)\d\omega.
\end{align*}
Since $r_{11}$ is integrable by virtue of the proof of Lemma~\ref{lem:r_(ij):L1}, its Fourier cosine transform converges to zero as $t$ tends to infinity. This establishes part (b).

(c) As in the proofs of parts (a) and (b), note that the second moment of $\int_0^t x(s)\d s$ can be written explicitly as
\begin{equation} \label{eq:E|int_0^t.x(s)ds|^2}
\begin{aligned}
\E\Big|\int_0^tx(s)\d s\Big|^2=\frac{2k_BT}{\pi}\int_0^\infty \frac{1-\cos(t\omega)}{\omega^2} r_{11}(\omega)\d\omega,
\end{aligned}
\end{equation}
where $r_{11}$ is the even function given by~\eqref{form:spectral-density:x:Kcos-Ksin}. Now, there are three situations depending on the asymptotic behavior of $K$ as characterized in Assumption~\ref{cond:K:general} (II).\\

\noindent \textbf{Case 1}: $K$ is integrable (Assumption~\ref{cond:K:general} (II) (a)). By a change of variable $u:=t\omega$ in~\eqref{eq:E|int_0^t.x(s)ds|^2}, we obtain
\begin{equation} \label{eq:E|int_0^t.x(s)ds|^2:u=t.omega}
\E\Big|\intx\Big|^2=t\frac{2k_BT}{\pi}\int_0^\infty \frac{1-\cos(u)}{u^2}
r_{11}\left(\frac{u}{t}\right)\d u.
\end{equation}
\[\]
Similarly to the proof of Lemma~\ref{lem:r_(ij):L1}, on one hand, as $\omega$ tends to infinity, $r_{11}$ converges to zero. On the other hand, by virtue of relation~\eqref{lim:r_[11].r_[22]:omega->0:diffusive}, $r_{11}(\omega)$ converges to $2\lambda/\gamma^2$ as $\omega \rightarrow 0$. In other words, $r_{11}$ is bounded on $[0,\infty)$. As a consequence, by the Dominated Convergence Theorem, we arrive at the limit
\begin{align*}
\frac{1}{t}\hspace{0.5mm}\E\Big|\intx\Big|^2=\frac{2k_BT}{\pi}\int_0^\infty \frac{1-\cos(u)}{u^2}
r_{11}\left(\frac{z}{t}\right)\d z\to \frac{4k_BT\lambda}{\pi\gamma^2}\int_0^\infty \frac{1-\cos(u)}{u^2}
\d u,
\end{align*}
as $t \rightarrow \infty$.\\

\noindent \textbf{Case 2}: $K\sim t^{-1}$ as $t\to\infty$ (Assumption~\ref{cond:K:general} (II) (b)). In this situation, $r_{11}(\omega)\sim |\log(\omega)|$ as $\omega\to 0$ (see~\eqref{lim:Kcos-Ksin:omega->0:critial}). In particular, $\sup_{\omega\in(0,1/2)}r_{11}(\omega)/|\log(\omega)|$ is finite.

Starting from~\eqref{eq:E|int_0^t.x(s)ds|^2:u=t.omega}, recast
\begin{equation} \label{eqn:thm:asymptotic:critical:1}
\begin{aligned}
\frac{\pi}{2k_BT\,t\log(t)}  \E\Big|\intx\Big|^2
=\frac{1}{\log(t)}\int_0^\infty\frac{1-\cos(u)}{u^2}
r_{11}\left(\frac{u}{t}\right)\d u.
\end{aligned}
\end{equation}
We want to show that the right-hand side of \eqref{eqn:thm:asymptotic:critical:1} converges to a finite limit as $t \rightarrow \infty$. To this end, we first decompose the integral into three terms, i.e.,
\begin{align*}
\frac{1}{\log(t)}\int^{\infty}_0\frac{1-\cos(u)}{u^2}
r_{11}\left(\frac{u}{t}\right)\d u & = \frac{1}{\log(t)}\Big\{\int_0^{e^{-2}}\close+\int_{e^{-2}}^{t/2}+\int_{t/2}^\infty \Big\} \frac{1-\cos(u)}{u^2}
r_{11}\left(\frac{u}{t}\right)\d u\\
&=I_1(t)+I_2(t)+I_3(t).
\end{align*}
 With regard to $I_3$, recall from the proof of Lemma~\ref{lem:r_(ij):L1} that $r_{11}(\omega)\sim \omega^{-4}$ as $\omega\to\infty$. Then,
\begin{align}\label{e:lim_I3(t)}
0 \leq \hspace{1mm}I_3(t)=\frac{1}{\log(t)}\int_{t/2}^\infty \frac{1-\cos(u)}{u^2}
r_{11}\left(\frac{u}{t}\right)\d u&\leq \frac{1}{\log(t)}\int_{t/2}^\infty  \frac{1-\cos(u)}{u^2}\d u \cdot \sup_{\omega\ge 1/2}r_{11}(\omega)\\
&\leq \frac{C}{t\log(t)} \rightarrow 0, \quad t\to\infty.
\end{align}

Concerning $I_1(t)$, rewrite
\begin{align*}
I_1(t)&=\frac{1}{\log(t)}\int_{0}^{e^{-2}} \frac{1-\cos(u)}{u^2}
r_{11}\left(\frac{u}{t}\right)\d u=\int_{0}^{e^{-2}} \frac{1-\cos(u)}{u^2}
\cdot\frac{\log(t/u)}{\log(t)}\cdot\frac{
 r_{11}(u/t)}{|\log(u/t)|}\d u.
\end{align*}
Note that, for sufficiently large $t$ and for all $u\in(0,e^{-2})$,
$$
\frac{\log(t/u)}{\log(t)}\le |\log(u)|.
$$
Together with~\eqref{lim:r_[11].r_[22]:omega->0:critical}, this implies that
\begin{align*}
\frac{\log(t/u)}{\log(t)}\cdot\frac{
 r_{11}(u/t)}{|\log(u/t)|}\le |\log(u)|\sup_{0<\omega<1/2}\frac{r_{11}(\omega)}{|\log(\omega)|}.
\end{align*}
It follows from Lemma~\ref{lem:Kcos-Ksin:omega->0}, (b), combined with the Dominated Convergence Theorem, that
\begin{align}\label{e:lim_I1(t)}
\lim_{t\to\infty}I_1(t)=\int_0^{e^{-2}} \frac{1-\cos(u)}{u^2}
\d u \cdot \lim_{\omega\to0}\frac{r_{11}(\omega)}{|\log(\omega)|}\in (0,\infty).
\end{align}

Regarding $I_2(t)$, similarly to $I_1(t)$, we note that, for all $u\in(e^{-2},t/2)$,
$$
\frac{\log(t/u)}{\log(t)}\le \frac{\log(t)+2}{\log(t)}<2.
$$
So,
\begin{align*}
\frac{\log(t/u)}{\log(t)}\cdot\frac{
 r_{11}(u/t)}{|\log(u/t)|}\le 2\sup_{0<\omega<1/2}\frac{r_{11}(\omega)}{|\log(\omega)|}.
\end{align*}
In light of the Dominated Convergence Theorem together with Lemma~\ref{lem:Kcos-Ksin:omega->0}, (b), we obtain
\begin{align}\label{e:lim_I2(t)}
\lim_{t\to\infty}I_2(t)=\int_{e^{-2}}^\infty \frac{1-\cos(u)}{u^2}
\d u \cdot \lim_{\omega\to0}\frac{r_{11}(\omega)}{|\log(\omega)|}\in (0,\infty).
\end{align}
The asymptotic expression for $\E\Big|\intx\Big|^2$ now follows from~\eqref{eqn:thm:asymptotic:critical:1}, \eqref{e:lim_I3(t)}, \eqref{e:lim_I1(t)} and \eqref{e:lim_I2(t)}.

\noindent \textbf{Case 3}: For some $\alpha\in(0,1)$, $K\sim t^{-\alpha}$ as $t\to\infty$ (Assumption~\ref{cond:K:general} (II) (c)). Note that \eqref{eq:E|int_0^t.x(s)ds|^2:u=t.omega} may be rewritten as
\begin{equation*}
\frac{1}{t^{2-\alpha}}\E\Big|\intx\Big|^2=\frac{2k_BT}{\pi}\int_0^\infty \frac{1-\cos(u)}{u^{3-\alpha}}\cdot
\frac{r_{11}(u/t)}{(u/t)^{\alpha-1}}\d u.
\end{equation*}
On one hand, for large $\omega$, $r_{11}(\omega) \leq C \omega^{-4}$. Thus, $r_{11}(\omega)/\omega^{\alpha-1} \rightarrow 0$ as $\omega \rightarrow \infty$. On the other hand, as $\omega \rightarrow 0$, relation~\eqref{lim:r_[11].r_[22]:omega->0:subdiffusive} implies that $r_{11}(\omega)/\omega^{\alpha-1}$ has a finite limit. In particular, this also implies that $r_{11}(\omega)/\omega^{\alpha-1}$ is bounded on $(0,\infty)$. In light of the Dominated Convergence Theorem together with~\eqref{lim:r_[11].r_[22]:omega->0:subdiffusive}, we obtain
\begin{equation*}
\frac{1}{t^{2-\alpha}}\hspace{0.5mm}\E\Big|\intx\Big|^2=\frac{2}{\pi}\int_0^\infty \frac{1-\cos(u)}{u^{3-\alpha}}\cdot
\frac{r_{11}(u/t)}{(u/t)^{\alpha-1}}\d u\to  c\in(0,\infty),
\end{equation*}
as $t \rightarrow \infty$. This completes the proof.
\end{proof}

\subsection{Equipartition of Energy} \label{sec:equipartition}

In what follows, we provide the proofs of Theorems~\ref{thm:equipartition-of-energy:1D:gamma=0:power-law.exponential} and~\ref{thm:equipartition-of-energy:2D:gamma>0:powerlaw.exponential}. So, let
$$
\cbb^+=\{u+\imag v:u\in\rbb, v\ge 0\}\quad \text{and}\quad\cbb^-=\{u+\imag v:u\in\rbb, v\le 0\}
$$
be the upper half and lower half complex plane, respectively. Also, let
$$
\cbb^*=\{z:\Re(z)\le 0\}
$$
be the left half plane of nonpositive real part in $\cbb$.

\subsubsection{\bf{Free-particle case ($\boldsymbol{\gamma=0)}$} } \label{sec:equipartition:1D:gamma=0}

In this subsection, we consider the case of a free particle as in equation~\eqref{eq:GLE:1D}. Our approach builds upon the work in \cite{hohenegger2017equipartition,kou2008stochastic}.

We introduce $f_1(z)$, the complex--valued function given by
\begin{equation} \label{form:f_1(z)}
f_1(z)=\frac{1}{\lambda +\beta (\Kcos(z)-\imag\Ksin(z)) + \imag mz}.
\end{equation}
The function $f_1(z)$ is closely related to the expressions for spectral densities $r_{22}$ and $r_{11}$, respectively, as in~\eqref{form:spectral-density:v} and~\eqref{form:spectral-density:x:Kcos-Ksin}, and will be used in the proof of Lemma 4.4 (see also \eqref{form:E[v(0)^2]} in the proof of Lemma~\ref{lem:equiparition:1D:gamma=0}).

\begin{remark} \label{rem:f_1:well-defined}
Note that, whereas $\Kcos(\omega)$ and $\Ksin(\omega)$ are well-defined for $\omega\in\rbb\setminus\{0\}$ (see Lemmas~\ref{lem:complete-monotone:Kcos-Ksin} and~\ref{lem:complete-monotone:Kcos-Ksin:K(t)=phi(t^2)}), $\Kcos(z)$ and $\Ksin(z)$ need not be for every $z\in\cbb\setminus\{0\}$. Hence, in formula~\eqref{form:f_1(z)}, $\Kcos(z)-\imag\Ksin(z)$ is understood as the integrals in either~\eqref{form:complete-monotone:Kcos-Ksin} or \eqref{form:complete-monotone:Kcos-Ksin:K(t)=phi(t^2)} extended to $\cbb$, depending on either $K\in\CM$ or $K=\f(t^2)$, $\f\in\CM$, respectively. Later in the proof of Theorem~\ref{thm:equipartition-of-energy:1D:gamma=0:power-law.exponential}, we will see that $\Kcos(z)-\imag\Ksin(z)$ is actually analytic on suitable subspaces of $\cbb$.
\end{remark}

For a large constant $R>0$, define, respectively, the outer circle and inner half circle in $\cbb^+$ as
\begin{equation} \label{form:curve:C_R.and.C_(1/R)}
 C_R^+=\{Re^{\imag \theta}:0\le \theta\le \pi\} \quad\text{and}\quad C_{1/R}^+= \{e^{\imag \theta}/R:0 \le \theta\le \pi\}.
\end{equation}
Further define their counterparts in $\cbb^-$ as
\begin{equation} \label{form:curve:C_R^-.and.C_(1/R)^-}
 C_R^-=\{Re^{\imag \theta}:-\pi\le \theta\le 0\} \quad\text{and}\quad C_{1/R}^-= \{e^{\imag \theta}/R:-\pi\le \theta\le 0\}.
\end{equation}
Also, let
\begin{equation}\label{form:curve:C(R)}
C(R)=[-R,-1/R]\cup C^-_{1/R}\cup[1/R,R]\cup C_R^-
\end{equation}
be a closed curve in $\cbb^-$, oriented clockwise.

Before discussing the proof of Theorem~\ref{thm:equipartition-of-energy:1D:gamma=0:power-law.exponential}, it is illuminating to recapitulate some technical aspects of previous work. In~\cite{hohenegger2017equipartition}, establishing \eqref{e:E[mv(0)^2]=kBT_intro} for the case of generalized Rouse kernels involved considering a complex--valued function similar to $f_1$ as in~\eqref{form:f_1(z)} and its contour integrals on the upper half plane $\cbb^+$. The argument relies heavily on a careful analysis of the locations of the poles of the functions involved. In turn, in \cite{kou2008stochastic}, establishing \eqref{e:E[mv(0)^2]=kBT_intro} for the class of memory kernels \eqref{e:t^(-alpha)} involved employing an integration trick via a smart change of variables.

Nevertheless, neither approach is available in the more general framework of this paper, which involves memory kernels that either are in $\CM$ or which have the form $\f(t^2)$ for $\f\in \CM$. As in~\cite{hohenegger2017equipartition}, we investigate contour integrals of $f_1(z)$ as in~\eqref{form:f_1(z)}. However, we shift the analysis to the \textit{lower} half complex plane $\cbb^-$. As it turns out, unlike in \cite{hohenegger2017equipartition}, dealing with poles is not needed when $K\in\CM$ since, in this case, the function $f_1(z)$ is analytic in $\cbb^-\sm$.

For the reader's convenience, we summarize the idea of the proof of Theorem~\ref{thm:equipartition-of-energy:1D:gamma=0:power-law.exponential}. The argument essentially consists of three steps as follows.\vspace{0.3cm}

\noindent \textbf{step 1}: We first consider $f_1(z)$ as in~\eqref{form:f_1(z)} and show that this function is analytic on $\cbb^-\sm$. This is established via the auxiliary results Lemma~\ref{lem:p(x).q_1(x)} and Lemma~\ref{lem:ptilde.qtilde_1}, respectively, for the cases $K\in\CM$ and $K(t)=\f(t^2)$, $\f\in\CM$.\vspace{0.1cm}

\noindent \textbf{step 2}: Next, we consider the contour integral on $\cbb^-\sm$ given by
\begin{equation}\label{e:in_f1(z)dz=0}
\int_{C(R)}\close f_1(z)\d z=\Big\{\int_{-R}^{-1/R}\close+\int_{C^-_{1/R}}\close+\int_{1/R}^{R}+\int_{C_R^-}\Big\}f_1(z)\d z=0.
\end{equation}
In \eqref{e:in_f1(z)dz=0}, the second equality holds by the analyticity of $f_1(z)$, as established in step 1. Then, we show that, as $R\to\infty$, the sum of the first and third integrals in \eqref{e:in_f1(z)dz=0} converges to $\E[m\,  v(0)^2]$, whereas the sum of the two remaining integrals converges to $-k_BT$. This establishes equipartition of energy for $v(t)$. This is discussed in detail in the proof of another auxiliary result, namely, Lemma~\ref{lem:equiparition:1D:gamma=0}, which states sufficient conditions on $f_1$, $\Kcos$ and $\Ksin$ for equipartition of energy to hold for the system~\eqref{eq:GLE:1D}.\vspace{0.1cm}

\noindent \textbf{step 3}: We prove Theorem~\ref{thm:equipartition-of-energy:1D:gamma=0:power-law.exponential} by verifying the assumptions of Lemma~\ref{lem:equiparition:1D:gamma=0}, while making use of Lemma~\ref{lem:p(x).q_1(x)} and Lemma~\ref{lem:ptilde.qtilde_1} depending on whether $K\in {\mathcal CM}$ or $K(t) = \varphi(t^2)$, $\varphi \in {\mathcal CM}$, respectively.\\

For the sake of clarity, the proofs of Lemmas ~\ref{lem:equiparition:1D:gamma=0}--\ref{lem:ptilde.qtilde_1} will be deferred to the end of this section. We start by stating Lemma \ref{lem:equiparition:1D:gamma=0}, where equipartition of energy is established directly based on assumptions on $f_1$, $\Kcos$ and $\Ksin$.
\begin{lemma} \label{lem:equiparition:1D:gamma=0} Suppose that $\gamma=0$. Let $v(t)$ be the process associated with the weak solution $V$ of~\eqref{eq:GLE:1D} and $f_1(z)$ be as in~\eqref{form:f_1(z)}. Suppose that
\begin{equation} \label{cond:lim.Kcos-iKsin/z=0:C^-}
\lim_{|z|\to\infty, \hspace{0.5mm}z \in \cbb^-}\frac{|\Kcos(z)-\imag \Ksin(z)|}{|z|}=0,
\end{equation}
and
\begin{equation} \label{cond:lim:zf_1(z)=0:C^-}
\lim_{|z|\to0, \hspace{0.5mm}z \in \cbb^-}|zf_1(z)|=0.
\end{equation}
Furthermore, suppose that, for all large enough $R>0$,
\begin{equation} \label{cond:int_C(R) f_1(z)=0}
\int_{C(R)}\close f_1(z)\d z=0,
\end{equation}
where $C(R)$ is the curve~\eqref{form:curve:C(R)}. Then,
\begin{equation}\label{e:E[m,v(0)^2]=k_BT_lemma}
\E[m\, v(0)^2]=k_BT.
\end{equation}
\end{lemma}

Next, we state Lemma \ref{lem:p(x).q_1(x)}, which is employed in showing that $f_1$ as in~\eqref{form:f_1(z)} is analytic on $\cbb^-\sm$ when $K\in\CM$.
\begin{lemma} \label{lem:p(x).q_1(x)}
Let $\mu$ be the representation measure on $[0,\infty)$ for $K\in\CM$ as in Theorem~\ref{thm:completely-monotone:Bernstein}. Let $p(z)$ and $q_1(z)$ be the complex--valued functions defined on $\cbb^*\setminus\{0\}$ and given by
\begin{equation} \label{form:p.q_1}
p(z)=\int_0^\infty\close \frac{1}{z-x}\mu(\d x)\quad\text{and}\quad q_1(z)=\lambda-mz-\beta p(z).
\end{equation}
\begin{itemize}
\item [(a)] Then, the function $p(z)$ is analytic on $\cbb^*\setminus\{0\}$. Moreover, for $z\in \cbb^*\setminus\{0\}$, it satisfies
\begin{equation} \label{lim:p/z=lim.zp=0}
\lim_{|z|\to\infty}|p(z)/z|=0 = \lim_{|z|\to 0}|z \cdot p(z)|,
\end{equation}
and, for $|z| \le 1$ in $\bbC^*\sm$
\begin{equation} \label{ineq:lambda+betap(z)>lambda}
|\lambda-\beta p(z)|\ge \frac{\beta}{\sqrt{2}}\int_0^\infty\close \frac{1}{x+1}\mu(\d x).
\end{equation}
\item [(b)] The function $q_1(z)$ in \eqref{form:p.q_1} is analytic in $\cbb^*\setminus\{0\}$ and $q_1(z)$ does not admit any complex root in $\cbb^*\setminus\{0\}$.
\end{itemize}
\end{lemma}

In Lemma \ref{lem:ptilde.qtilde_1}, covering the case where $K=\f(t^2)$, $\f\in\CM$, the analysis involves the special class of \emph{error functions}. For the reader's convenience, we briefly recapitulate some related notions.

Recall that the so--named \textit{complementary error function} is given by
\begin{equation}\label{e:erfc}
\erfc(z) := \frac{2}{\sqrt{\pi}} \int^{\infty}_{z}\close e^{-t^2}dt = 1 - \erf(z), \quad z \in \bbC,
\end{equation}
where the \textit{error function} admits the MacLaurin series representation
\begin{equation} \label{e:erf}
\erf(z) := \frac{2}{\sqrt{\pi}} \int^{z}_0 e^{-t^2}dt = \frac{2}{\sqrt{\pi}} \sum^{\infty}_{n=0} \frac{(-1)^n z^{2n+1}}{n! (2n+1)}, \quad z \in \bbC.
\end{equation}
In particular, both $\erf$ and $\erfc$ are entire functions. Now consider the function
\begin{equation} \label{form:w(z)}
w(z) = e^{-z^2} \erfc(-\imag z), \quad z \in \bbC,
\end{equation}
also called \textit{Faddeeva function} or \textit{plasma dispersion function}. The function $w(z)$ also admits the Hilbert transform representation \cite[expression (8)]{fettis:caslin:cramer:1973}
\begin{equation}\label{e:w(z)_Hilbert}
w(z) = \frac{\imag}{\pi} \int_{\bbR} \frac{e^{-t^2}}{z - t} dt, \quad \Im(z) > 0.
\end{equation}
When $z = x \in \bbR$, \eqref{e:w(z)_Hilbert} should be modified to
\begin{equation}\label{e:w(x)}
w(x) = e^{-x^2} + \frac{2 \imag}{\sqrt{\pi}} \hspace{0.5mm}\textnormal{daw}(x),
\end{equation}
where the so--named \textit{Dawson integral} is given by \cite[pp.\ 1497--1498]{weideman:1994}
\begin{equation} \label{form:dawson(z)}
\textnormal{daw}(z) = e^{-z^2} \int^{z}_0 e^{t^2}dt, \quad z \in \bbC.
\end{equation}
Having introduced these special functions, we are now in a position to state Lemma \ref{lem:ptilde.qtilde_1}, which is employed in showing that $f_1$ as in~\eqref{form:f_1(z)} is analytic in $\cbb^-\sm$ when $K=\f(t^2)$, $\f\in\CM$.
\begin{lemma} \label{lem:ptilde.qtilde_1}
Suppose $K(t)=\f(t^2)$ where $\f\in\CM$. Let $\mu$ be the representation measure on $[0,\infty)$ for $\f$ as in Theorem~\ref{thm:completely-monotone:Bernstein}. Let $\ptil(z)$ and $\qtil_1(z)$ be the complex--valued functions defined on $\cbb^-\setminus\{0\}$ and given by
\begin{equation} \label{form:ptilde.qtilde_1}
\ptil(z)=\frac{\sqrt{\pi}}{2}\int_0^\infty\close \frac{1}{\sqrt{x}}w\Big(\!-\frac{z}{2\sqrt{x}}\Big)\mu(\d x)\quad\text{and}\quad \qtil_1(z)=\lambda+\imag mz+\beta \ptil(z).
\end{equation}
\begin{itemize}
\item [(a)] Then, the function $\ptil(z)$ is analytic on $\cbb^-\setminus\{0\}$. Moreover, for $z\in \cbb^-\setminus\{0\}$, it satisfies
\begin{equation} \label{lim:ptilde/z=lim.zptilde=0}
\lim_{|z|\to\infty}|\ptil(z)/z|=0 = \lim_{|z|\to 0}|z \cdot \ptil(z)|,
\end{equation}
and, for $|z| \le 1$ in $\cbb^-\sm$
\begin{equation} \label{ineq:lambda+beta.ptilde(z)>0}
|\lambda+\beta \ptil(z)|\ge \beta\frac{\sqrt{\pi}}{4}\int_0^\infty\close\frac{1}{\sqrt{x}}e^{-\frac{1}{2x}}\mu(\d x).
\end{equation}
\item [(b)] The function $\qtil_1(z)$ in \eqref{form:ptilde.qtilde_1} is analytic in $\cbb^-\setminus\{0\}$ and $\qtil_1(z)$ does not admit any complex root in $\cbb^-\setminus\{0\}$.
\end{itemize}
\end{lemma}

After stating Lemmas~\ref{lem:equiparition:1D:gamma=0}--\ref{lem:ptilde.qtilde_1}, we provide the proof of Theorem~\ref{thm:equipartition-of-energy:1D:gamma=0:power-law.exponential}.
\begin{proof}[Proof of Theorem~\ref{thm:equipartition-of-energy:1D:gamma=0:power-law.exponential}]

We first consider the case where $K\in\CM$. The proof is based on verifying the assumptions of Lemma~\ref{lem:equiparition:1D:gamma=0}, while making use of Lemma~\ref{lem:p(x).q_1(x)}.

With regard to condition~\eqref{cond:lim.Kcos-iKsin/z=0:C^-}, note that, by virtue of formula~\eqref{form:complete-monotone:Kcos-Ksin} extended to $\cbb^-\setminus\{0\}$,
\begin{align*}
\Kcos(z)-\imag \Ksin(z)&=\int_0^\infty\close \frac{x-\imag z}{x^2+z^2}\mu(\d x)=-\int_0^\infty\close\frac{1}{-\imag z-x}\mu(\d x)=-p(-\imag z),
\end{align*}
where $p(x)$ is as in~\eqref{form:p.q_1}. Also, since $z\in\cbb^-\setminus\{0\}$, it is clear that $-\imag z\in\cbb^*\setminus\{0\}$. In light of Lemma~\ref{lem:p(x).q_1(x)}, (a), cf.~\eqref{lim:p/z=lim.zp=0}, we conclude that \eqref{cond:lim.Kcos-iKsin/z=0:C^-} holds.

Turning to the limit~\eqref{cond:lim:zf_1(z)=0:C^-}, first recall that $q_1(z)$ is given by~\eqref{form:p.q_1}. Next, let $f_1(z)$ be as in~\eqref{form:f_1(z)}. Recast
\begin{align}\label{e:f_1(z)=1/q_1(-iz)}
f_1(z)&=\frac{1}{\lambda +\beta \int_0^\infty\frac{1}{\imag z+x}\mu(\d x) +\imag mz}=\frac{1}{\lambda-\beta p(-\imag z)-m(-\imag z)}=\frac{1}{q_1(-\imag z)}.
\end{align}
Then, for $z\in\cbb^-\sm$ such that $|z|$ is small enough, relation \eqref{ineq:lambda+betap(z)>lambda} in Lemma~\ref{lem:p(x).q_1(x)}, (a), implies that
\begin{align}\label{e:f_1(z)=<o(1)}
|zf_1(z)|\le \frac{|z|}{|\lambda-\beta p(-\imag z)|-m|z|}\le \frac{|z|}{\frac{\beta}{\sqrt{2}}\int_0^\infty \frac{1}{x+1}\mu(\d x)-m|z|}.
\end{align}
Since the upper bound in \eqref{e:f_1(z)=<o(1)} converges to zero as $|z|\to 0$ in $\cbb^-\sm$, then condition \eqref{cond:lim:zf_1(z)=0:C^-} holds.

However, by Lemma~\ref{lem:p(x).q_1(x)}, (b), $q_1(-\imag z)$ is analytic as a function of $z \in \cbb^-\setminus\{0\}$ and, also, does not admit any complex root in the same domain. Therefore, by expression \eqref{e:f_1(z)=1/q_1(-iz)}, $f_1(z)$ is analytic in $\cbb^-\setminus\{0\}$. In particular, condition \eqref{cond:int_C(R) f_1(z)=0} holds. Consequently, by Lemma~\ref{lem:equiparition:1D:gamma=0}, relation \eqref{e:E[m,v(0)^2]=k_BT} is established.

We now turn to the case $K(t)=\f(t^2)$ where, $\f\in\CM$. Similarly to the previous case, we need to verify the assumptions of Lemma~\ref{lem:equiparition:1D:gamma=0} while making use of the auxiliary results in Lemma~\ref{lem:ptilde.qtilde_1}.
 Concerning limit~\eqref{cond:lim.Kcos-iKsin/z=0:C^-}, in view of expressions~\eqref{form:complete-monotone:Kcos-Ksin:K(t)=phi(t^2)} and~\eqref{form:ptilde.qtilde_1}, we immediately obtain for $z\in \bbC^-\sm$
 \begin{align*}
 \frac{|\Kcos(z)-\imag \Ksin(z)|}{|z|}=\frac{|\ptil(z)|}{|z|}\to 0,\quad |z|\to\infty,
 \end{align*}
 by virtue of Lemma~\ref{lem:ptilde.qtilde_1}, (a), cf.~\eqref{lim:ptilde/z=lim.zptilde=0}.

Considering the limit~\eqref{cond:lim:zf_1(z)=0:C^-}, first recall that $\ptil(z)$ and $\qtil_1(z)$ are given by~\eqref{form:ptilde.qtilde_1}. Next, observe that $f_1(z)$ as in~\eqref{form:f_1(z)} can be rewritten as
\begin{align*}
f_1(z)=\frac{1}{\qtil_1(z)}=\frac{1}{\lambda+\beta\ptil(z)+\imag mz}.
\end{align*}
In light of the estimate~\eqref{ineq:lambda+beta.ptilde(z)>0}, for $z \in \bbC^-\sm$ and sufficiently small $|z|$, we obtain
\begin{align*}
|zf_1(z)|\le \frac{|z|}{|\lambda+\beta\ptil(z)|-m|z|}\le  \frac{|z|}{\beta\frac{\sqrt{\pi}}{4}\int_0^\infty\frac{1}{\sqrt{x}}e^{-\frac{1}{2x}}\mu(\d x)-m|z|} \rightarrow 0, \quad |z| \rightarrow 0.
\end{align*}
This establishes ~\eqref{cond:lim:zf_1(z)=0:C^-}.

In regard to condition~\eqref{cond:int_C(R) f_1(z)=0}, fix $R > 0$. By virtue of Lemma~\ref{lem:ptilde.qtilde_1}, (b), since $\qtil_1(z)$ is analytic and does not admit any root in $\bbC^-\sm$, then there exists a simply connected region $D \supseteq C(R)$ such that $\qtil_1(z) \neq 0$, for all $z \in D$. Hence, $f_1(z)=1/\qtil_1(z)$ is analytic on $D$, implying that condition \eqref{cond:int_C(R) f_1(z)=0} holds.

As a consequence, by Lemma~\ref{lem:equiparition:1D:gamma=0}, relation~\eqref{e:E[m,v(0)^2]=k_BT} is established. This concludes the proof.
\end{proof}

We now turn to the proofs of the auxiliary results. First, we provide the proof of Lemma~\ref{lem:equiparition:1D:gamma=0}.
\begin{proof}[Proof of Lemma~\ref{lem:equiparition:1D:gamma=0}]
Recall that the spectral density $r_{22}$ for $v(t)$ is given by~\eqref{form:spectral-density:v} and~\eqref{form:spectral-density:x:Kcos-Ksin}. Together with the covariance function~\eqref{form:stationary-distribution:E[<G,phi1><G,phi2>]} and based on the condition that $\gamma=0$, we have
 \begin{align} \label{form:E[v(0)^2]}
\E[v(0)^2]&=\frac{k_BT}{\pi}\int_0^\infty\close r_{22}(\omega)\d\omega=\frac{k_BT}{\pi}\int_0^\infty\close \frac{2\big(\lambda+\beta\Kcos(\omega)\big)}{|m\omega-\beta \Ksin(\omega)|^2+|\lambda+\beta\Kcos(\omega))|^2}\d\omega.
\end{align}
It therefore suffices to prove that $\int_0^\infty r_{22}(\omega)\d\omega=\pi m^{-1}$.

Now, consider the contour integral of the function $f_1$ as in \eqref{form:f_1(z)} on $C(R)$. Note that we may decompose
\begin{align*}
\int_{C(R)}f_1(z)\d z&=\Big\{\int_{-R}^{-1/R}\close+\int_{C^-_{1/R}}\close+\int_{1/R}^{R}+\int_{C_R^-}\Big\}f_1(z)\d z\\
&=I_1 (R) +I_2(R) +I_3(R) +I_4(R)  ,
\end{align*}
where we recall that $C^-_R$ and $C^-_{1/R}$, respectively, are the outer and inner half circles in $\cbb^-$ as in~\eqref{form:curve:C_R^-.and.C_(1/R)^-}. Using the variable $\omega$ for integration along the real axis, it is straightforward to see that
$$
I_3(R) =\int_{1/R}^R\frac{\d\omega}{\lambda +\beta \Kcos(\omega) + \imag (m\omega-\beta \Ksin(\omega))}.
$$
Concerning $I_1(R)$, note that $\Kcos(\omega)$ and $\Ksin(\omega)$ are even and odd functions, respectively. Thus, by a change of variable $z:=-\omega$, we obtain
\begin{align*}
I_1(R) &=\int_{R}^{1/R}\close \frac{-\d \omega}{\lambda +\beta \Kcos(-\omega) + \imag (m(-\omega)-\beta\Ksin(-\omega))}\\
&=\int_{1/R}^{R} \frac{\d \omega}{\lambda +\beta \Kcos(\omega) - \imag (m\omega-\beta \Ksin(\omega))}.
\end{align*}
It follows immediately that
$$
I_1(R) +I_3(R) =\int_{1/R}^R\frac{2\big(\lambda+\beta\Kcos(\omega)\big)}{|m\omega-\beta \Ksin(\omega)|^2+|\lambda+\beta\Kcos(\omega))|^2}\d\omega.
$$
By the Monotone Convergence Theorem, we obtain
\begin{equation}\label{e:I_1(R)+I_3(R)->int_r22}
I_1(R) +I_3(R)\to \int_0^\infty\close r_{22}(\omega)\d\omega,\quad\text{as }R\to\infty.
\end{equation}
Concerning $I_2(R)$, by making the change of variable
\begin{equation}\label{e:z=R^(-1)e^(i*theta)}
z:=R^{-1}e^{\imag \theta},
\end{equation}
we can write
\begin{align*}
I_2(R)&=\int_{-\pi}^0\frac{R^{-1}e^{\imag \theta}\imag \d\theta}{\lambda +\beta \big[\Kcos(R^{-1} e^{\imag \theta})-\imag  \Ksin(R^{-1} e^{\imag \theta})\big] + \imag mR^{-1} e^{\imag \theta}}\\
&=\int_{-\pi}^0R^{-1}e^{\imag \theta}f_1(R^{-1}e^{\imag \theta})\imag \d\theta.
\end{align*}
Then, by virtue of the Dominated Convergence Theorem together with~\eqref{cond:lim:zf_1(z)=0:C^-},
\begin{equation}\label{e:lim_I_2(R)=0}
\lim_{R \rightarrow \infty}I_2(R) = 0.
\end{equation}
Likewise, with regards to $I_4(R)$, by the change of variable $z:=Re^{\imag \theta}$,
\begin{align}\label{e:I4(R)}
I_4(R)=\int_0^{-\pi}\close\frac{\imag \d\theta}{\lambda R^{-1}e^{-\imag \theta} +\beta R^{-1}e^{-\imag \theta}  \big[\Kcos(Re^{\imag \theta}) - \imag  \Ksin(Re^{\imag \theta})\big]+ \imag m  }.
\end{align}
In view of condition~\eqref{cond:lim.Kcos-iKsin/z=0:C^-}, the integrand in \eqref{e:I4(R)} converges to $m^{-1}$ as $R\to\infty$. By the Dominated Convergence Theorem, we further obtain
\begin{equation}\label{e:I_4(R)->-pi/m}
I_4(R) \to -\pi m^{-1}, \quad\text{as }R\to\infty.
\end{equation}
Together with the limits \eqref{e:lim_I_2(R)=0}, \eqref{e:I_1(R)+I_3(R)->int_r22} and \eqref{e:I_4(R)->-pi/m}, expression~\eqref{cond:int_C(R) f_1(z)=0} yields the limit
\begin{equation} \label{eq:int_C(R)f(z)dz=int_0^infty.r_v+pi.m^(-1)}
0=\lim_{R\to\infty}\int_{C(R)}\close f(z)\d z= \int_0^\infty \close r_{22}(\omega)\d\omega-\pi m^{-1}.
\end{equation}
Hence,
$$
\int_0^\infty \close r_{22}(\omega)\d\omega=\pi m^{-1}.
$$
This establishes \eqref{e:E[m,v(0)^2]=k_BT_lemma}.
\end{proof}

Next, we give the the proof of Lemma~\ref{lem:p(x).q_1(x)}.
\begin{proof}[Proof of Lemma~\ref{lem:p(x).q_1(x)}]
(a) Firstly, with regards to analyticity, letting $z_0\in \cbb^*\setminus\{0\}$, it suffices to show that $p(z)$ as in~\eqref{form:p.q_1} can be expanded for all $z\in B(z_0,|z_0|/2)$. To see this, we compute
\begin{align}\label{e:Stieltjes_transf_series}
\int_0^\infty\close \frac{1}{z-x}\mu(\d x)&=\int_0^\infty\close \frac{1}{(z_0-x)\big(\frac{z-z_0}{z_0-x}+1\big)}\mu(\d x)=\sum_{n\ge 0}\int_0^\infty\close \frac{(-1)^n}{(z_0-x)^{n+1}}\mu(\d x)\cdot (z-z_0)^n,
\end{align}
where the second equality is obtained by interchanging integration and summation. To justify this interchange, we claim that the series in \eqref{e:Stieltjes_transf_series} converges absolutely for all $z\in B(z_0,|z_0|/2)$. Indeed, by writing $z_0=-u+\imag v\in \cbb^*\sm$, with $u \geq 0$, for all $x\ge 0$ we can bound
\begin{equation}\label{ineq:|z-x|>max(|z|,x)}
|z_0 - x|=|-(x+u)+\imag v|\ge \max\{|z_0|,x\}.
\end{equation}
Thus, we can estimate
\begin{align*}
\Big| \int_0^\infty\close \frac{(-1)^n}{(z_0-x)^{n+1}}\mu(\d x)\cdot(z-z_0)^n \Big|&\le \int_0^\infty\close \frac{1}{|-(x+u)+\imag v|^{n+1}}\mu(\d x)\frac{|z_0|^n}{2^n}\\
&= \Big\{\int_0^1+\int_1^\infty\Big\}\frac{1}{|-(x+u)+\imag v|^{n+1}}\mu(\d x)\frac{|z_0|^n}{2^n}\\
&\le \frac{1}{2^n |z_0|}\int_0^1\mu(\d x)+\frac{1}{2^n}\int_1^\infty \frac{1}{x}\mu(\d x).
\end{align*}
This implies that
\begin{align*}
\sum_{n\ge 0}\Big|\int_0^\infty\close \frac{(-1)^n}{(z_0-x)^{n+1}}\mu(\d x)\cdot (z-z_0)^n\Big|\le \frac{1}{|z_0|}\mu([0,1])+\int_1^\infty\frac{1}{x}\mu(\d x)<\infty,
\end{align*}
where the last implication follows from the fact that $\mu([0,1])$ and $\int_1^\infty\frac{1}{x}\mu(\d x)$ are both finite by virtue of~\eqref{ineq:mu(0,1)<infty}-\eqref{ineq:int_1^infty 1/x.mu(dx)<infty}. This establishes the analyticity of $p(z)$.

Next, we turn to~\eqref{lim:p/z=lim.zp=0}. On one hand, by~\eqref{ineq:|z-x|>max(|z|,x)},
\begin{align*}
|p(z)|\le \frac{1}{|z|}\int_0^1\mu(\d x)+\int_1^\infty \frac{1}{x}\mu(\d x),
\end{align*}
implying that $\lim_{|z|\to\infty}|p(z)/z|=0$. On the other hand, the bound ~\eqref{ineq:|z-x|>max(|z|,x)} (with $z \in \bbC^* \backslash\{0\}$ in place of $z_0$) implies that
\begin{align}\label{e:int_|z|/|z-x|mu(dx)_bound}
\int_0^\infty \close\frac{|z|}{|z-x|}\mu(\d x)\le \int_0^1 \frac{|z|}{|z-x|}\mu(\d x)+ \int_1^\infty \frac{|z|}{x}\mu(\d x).
\end{align}
Since $\int_1^\infty x^{-1}\mu(\d x)<\infty$, cf.~\eqref{ineq:int_1^infty 1/x.mu(dx)<infty}, the second term on the right-hand side of \eqref{e:int_|z|/|z-x|mu(dx)_bound} converges to zero as $|z| \rightarrow 0$. Also, by virtue of~\eqref{ineq:|z-x|>max(|z|,x)}, the integrand in the first term on the right-hand side of \eqref{e:int_|z|/|z-x|mu(dx)_bound} is bounded uniformly in $x$. By the Dominated Convergence Theorem, this implies that its limit is also zero. Therefore, $\lim_{|z|\rightarrow 0}|z \cdot p(z)| = 0$, as claimed. This establishes \eqref{lim:p/z=lim.zp=0}.

In regard to~\eqref{ineq:lambda+betap(z)>lambda}, let $z=-u+\imag v\in\cbb^*\setminus\{0\}$, $u\ge 0$, such that $|z|\le 1$. Then, for any $x,u,|v|\ge 0$, 
we claim that
\begin{align*}
\frac{x+u+|v|}{(x+u)^2+v^2}\ge \frac{1}{x+1}.
\end{align*}
Indeed, by multiplying through the denominators, the above inequality is equivalent to
\begin{align*}
x^2+xu+x|v|+x+u+|v|\ge x^2+2xu+u^2+v^2,
\end{align*}
i.e.,
\begin{align} \label{ineq:xv+x+u+v>xu+u^2+v^2}
x|v|+x+u+|v|\ge xu+u^2+v^2.
\end{align}
However, inequality~\eqref{ineq:xv+x+u+v>xu+u^2+v^2} always holds, since $u,|v|\in[0,1]$ and $x\ge 0$.

Therefore, by the elementary inequality $\sqrt{2(a^2+b^2)} \geq |a| + |b|$,
\begin{align*}
|\lambda-\beta p(z)|&=\Big|\lambda+\beta\int_0^\infty\close \frac{u+x}{(u+x)^2+v^2}\mu(\d x)+\imag \beta\int_0^\infty\close \frac{v}{(u+x)^2+v^2}\mu(\d x)\Big|\\
&\ge \frac{1}{\sqrt{2}}\Big(\lambda+\beta\int_0^\infty\close \frac{u+x+|v|}{(u+x)^2+v^2}\mu(\d x)\Big)\\
&\ge \frac{\beta}{\sqrt{2}}\int_0^\infty\close \frac{1}{x+1}\mu(\d x).
\end{align*}
This establishes (a).\\

(b) Since $p(z)$ is analytic, then so is $q_1(z)$. To show that $q_1(z)$ does not admit any root in $\cbb^*\setminus\{0\}$, suppose, by means of contradiction, that, for some $z_0:= -u + \imag v \in \cbb^*\setminus\{0\}$, $q_1 (z_0) = 0$. In particular, $u\ge 0$, $v \in\rbb$. A simple calculation yields
\begin{align*}
0=\Re\big(q_1(z_0)\big)=\lambda+mu+\beta\int_0^\infty\close \frac{u+x}{(u+x)^2+v^2}\mu(\d x)>0,
\end{align*}
a contradiction. This shows (b).

\end{proof}

We finish this subsection by presenting the proof of Lemma~\ref{lem:ptilde.qtilde_1}.
\begin{proof}[Proof of Lemma~\ref{lem:ptilde.qtilde_1}]
(a) Similarly to the proof of Lemma~\ref{lem:p(x).q_1(x)}, fixing $z_0\in \bbC^-\sm$, we want to show $\ptil(z)$ can be expanded in a neighborhood of $z_0$. To see this, we first choose an open disk $B(z_0,\varepsilon)$ centered at $z_0$ with radius $\varepsilon$ such that
\begin{equation}\label{cond:B(z_0,varepsilon)}
\overline{B(z_0,\varepsilon)}\subset\Big\{R e^{\imag\theta}:-\frac{9\pi}{8}<\theta<\frac{\pi}{8},R>0\Big\}.
\end{equation}
For each $x>0$, let $w_x(z)=w\big(-\frac{z}{2\sqrt{x}}\big)$. Since $w(z)$ is entire, then so is $w_x(z)$. In light of Cauchy's integral formula, for all $z_1\in B(z_0,\varepsilon)$, we can write
\begin{align*}
w_x(z_1)&=\frac{1}{2\pi \imag}\sum_{n\ge 0}\int_{\partial B(z_0,\varepsilon)}\frac{1}{(z-z_0)^{n+1}}w_x(z)\d z\, \cdot (z_1-z_0)^n\\
&=\frac{1}{2\pi \imag}\sum_{n\ge 0}\int_{\partial B(z_0,\varepsilon)}\frac{1}{(z-z_0)^{n+1}}w\Big(\!-\frac{z}{2\sqrt{x}}\Big)\d z\, \cdot (z_1-z_0)^n.
\end{align*}
Recall that $\ptil(z)$ is given by~\eqref{form:ptilde.qtilde_1}. Then,
\begin{align}
\frac{2}{\sqrt{\pi}}\ptil(z_1)&=\int_0^\infty\close \frac{1}{\sqrt{x}}w_x(z_1)\mu(\d x)\nonumber \\
&=\frac{1}{2\pi \imag} \int_0^\infty\close \frac{1}{\sqrt{x}}\Big[\sum_{n\ge 0}\int_{\partial B(z_0,\varepsilon)}\frac{1}{(z-z_0)^{n+1}}w\Big(\!-\frac{z}{2\sqrt{x}}\Big)\d z\,\cdot (z_1-z_0)^n\Big]\mu(\d x) \nonumber\\
&=\frac{1}{2\pi \imag} \sum_{n\ge 0}\Big[ \int_0^\infty\close \frac{1}{\sqrt{x}}\int_{\partial B(z_0,\varepsilon)}\frac{1}{(z-z_0)^{n+1}}w\Big(\!-\frac{z}{2\sqrt{x}}\Big)\d z\,\mu(\d x)\Big]\cdot (z_1-z_0)^n \nonumber\\
&=\frac{1}{2\pi \imag} \sum_{n\ge 0}I_n\cdot(z_1-z_0)^n. \label{e:(1/2*pi*i)sum_In(z1-z0)^n}
\end{align}
In the third equality, we formally interchanged the order of integration with respect to $\mu(\d x)$ and the summation. To justify this step, it suffices to show that the series in \eqref{e:(1/2*pi*i)sum_In(z1-z0)^n} converges absolutely for all $z_1\in B(z_0,\varepsilon)$. In fact, considering $I_n$,
\begin{align}
|I_n| &\le c \int_0^\infty \close\frac{1}{\sqrt{x}}\sup_{z\in\partial B(z_0,\varepsilon)}\Big|w\Big(\!-\frac{z}{2\sqrt{x}}\Big)\Big|\mu(\d x)\cdot \frac{1}{\varepsilon^n} \nonumber \\
&=c \Big\{\int_0^1+\int_1^\infty\Big\} \frac{1}{\sqrt{x}}\sup_{z\in\partial B(z_0,\varepsilon)}\Big|w\Big(\!-\frac{z}{2\sqrt{x}}\Big)\Big|\mu(\d x) \cdot \frac{1}{\varepsilon^n}, \label{e:|In|=<sum_two_integrals}
\end{align}
where $c=c(z_0,\varepsilon)$ is a positive constant independent of $n$. We now bound the integrals with respect to $\mu$ on the right-hand side of \eqref{e:|In|=<sum_two_integrals}. On one hand, we invoke~\eqref{ineq:int_1^infty 1/sqrt(x).mu(dx)<infty} together with the fact that $w(z)$ is entire to bound
\begin{align}
\int_1^\infty\close \frac{1}{\sqrt{x}}\sup_{z\in\partial B(z_0,\varepsilon)}\Big|w\Big(\!-\frac{z}{2\sqrt{x}}\Big)\Big|\mu(\d x)\le \int_1^\infty\close\frac{1}{\sqrt{x}}\mu(\d x) \cdot \sup_{z\in B(0,|z_0|+\varepsilon)}|w(z)|<\infty. \label{e:int^{infty}_1_1/sqrt(x)_sup_mu(dx)}
\end{align}
On the other hand, by the choice of $B(z_0,\varepsilon)$ as in~\eqref{cond:B(z_0,varepsilon)}, for all $z\in B(z_0,\varepsilon)$ and $x\in(0,1)$, $$-\frac{z}{2\sqrt{x}}\in \Big\{R e^{\imag\theta}:-\frac{\pi}{8}<\theta<\frac{9\pi}{8}\Big\}. $$
In view of Lemma~\ref{lem:w(z)}, cf.~\eqref{ineq:w(z)}, namely, $|w(z)|\le c/|z|$, we have the bound
\begin{align}\label{e:int^1_0_1/sqrt(x)_sup_mu(dx)}
\int_0^1 \frac{1}{\sqrt{x}}\sup_{z\in\partial B(z_0,\varepsilon)}\Big|w\Big(\!-\frac{z}{2\sqrt{x}}\Big)\Big|\mu(\d x)\le c\mu([0,1]).
\end{align}
By \eqref{e:int^{infty}_1_1/sqrt(x)_sup_mu(dx)} and \eqref{e:int^1_0_1/sqrt(x)_sup_mu(dx)}, we conclude that there exists a constant $c=c(z_0,\varepsilon)>0$, independent of $n$, such that
\begin{equation}\label{e:|In|=<ce^(-n)}
|I_n|\le c\,\varepsilon^{-n}.
\end{equation}
Recall that $z_1\in B(z_0,\varepsilon)$. Relation \eqref{e:|In|=<ce^(-n)} implies that
$$
\sum_{n\ge 0}|I_n|\cdot|z_1-z_0|^n\le c\sum_{n\ge 0}\Big|\frac{z_1-z_0}{\varepsilon}\Big|^n<\infty.
$$
This establishes the analyticity of $\ptil(z_0)$ for all $z_0\in \bbC^-\sm$.

Next, we turn to the limits~\eqref{lim:ptilde/z=lim.zptilde=0}. To show that
\begin{equation}\label{e:lim|p-tilde(z)/z|}
\lim_{|z|\rightarrow \infty}|\ptil(z)/z| = 0,
\end{equation}
first recast
\begin{align*}
\frac{2}{\sqrt{\pi}}\cdot\frac{\ptil(z)}{z}=\frac{1}{z}\Big\{\int_0^1+\int_1^{|z|^2}\close+\int_{|z|^2}^\infty \Big\} \frac{1}{\sqrt{x}}w\Big(\!-\frac{z}{2\sqrt{x}}\Big)\mu(\d x)=I_1(z)+I_2(z)+I_3(z).
\end{align*}
Now, to bound $I_3$, consider $z\in \bbC^-\sm$ such that $|z|>1$. Then, we can invoke~\eqref{ineq:int_1^infty 1/sqrt(x).mu(dx)<infty} to conclude that
\begin{align}\label{e:|I3(z)|->0}
|I_3(z)|\le \frac{1}{|z|}\int_{1}^\infty\close \frac{1}{\sqrt{x}}\mu(\d x) \cdot \sup_{z_1\in B(0,1/2)}|w(z_1)| \rightarrow 0
\end{align}
as $|z|\to\infty$ in $\bbC^-\sm$. With regard to $I_2$, by combining~\eqref{ineq:int_1^infty 1/sqrt(x).mu(dx)<infty} with~\eqref{ineq:w(z)}, we obtain
\begin{align}\label{e:|I2(z)|->0}
|I_2(z)|\le \frac{c}{|z|^2}\int_1^{|z|^2}\close\mu(\d x)\le \frac{c}{|z|}\int_1^{|z|^2}\close\frac{1}{\sqrt{x}}\mu(\d x)\to0,\quad |z|\to\infty.
\end{align}
Likewise,
\begin{align}\label{e:|I1(z)|->0}
|I_1(z)|\le \frac{c}{|z|^2}\mu([0,1])\to 0,\quad |z|\to\infty.
\end{align}
Relation \eqref{e:lim|p-tilde(z)/z|} is now a consequence of \eqref{e:|I3(z)|->0}, \eqref{e:|I2(z)|->0} and \eqref{e:|I1(z)|->0}.

We now show that
\begin{equation}\label{e:|z.p-tilde(z)|->0}
\lim_{|z|\rightarrow 0} |z \cdot \widetilde{p}(z)| = 0.
\end{equation}
Similarly to the argument for \eqref{e:lim|p-tilde(z)/z|}, for $|z|<1$ we write
\begin{align*}
\frac{2}{\sqrt{\pi}}\,z\,\ptil(z)=z\Big\{\int_0^{|z|^2}\close+\int_{|z|^2}^{|z|}+\int_{|z|}^1+\int_1^\infty \Big\} \frac{1}{\sqrt{x}}w\Big(\!-\frac{z}{2\sqrt{x}}\Big)\mu(\d x)=I_4(z)+I_5(z)+I_6(z)+I_7(z).
\end{align*}
By~\eqref{ineq:int_1^infty 1/sqrt(x).mu(dx)<infty},
\begin{align}\label{e:|I7(z)|->0}
|I_7(z)|\le |z|\int_1^\infty\close\frac{1}{\sqrt{x}}\mu(\d x) \cdot \sup_{z_1\in B(0,1)}|w(z_1)|\to0,\quad |z|\to 0.
\end{align}
Likewise, again as $|z|\to 0$,
\begin{align}\label{e:|I6(z)|->0}
|I_6(z)|\le |z|\int_{|z|}^1\frac{1}{\sqrt{x}}\mu(\d x) \cdot \sup_{z_1\in B(0,1)}|w(z_1)|\le \sqrt{|z|}\,\mu([0,1])\sup_{z_1\in B(0,1)}|w(z_1)|\to0
\end{align}
and
\begin{align}\label{e:|I5(z)|->0}
|I_5(z)|\le |z|\int_{|z|^2}^{|z|}\frac{1}{\sqrt{x}}\mu(\d x)\cdot \sup_{z_1\in B(0,1)}|w(z_1)|\le \mu([0,|z|])\sup_{z_1\in B(0,1)}|w(z_1)|\to0.
\end{align}
In the last limit above, we employed the fact that $\lim_{|z|\to 0}\mu([0,|z|])=\mu(\{0\})=0$ as in~\eqref{eqn:mu(0)=0}. Moreover, by combining~\eqref{ineq:int_1^infty 1/sqrt(x).mu(dx)<infty} with~\eqref{ineq:w(z)}, we obtain
\begin{align}\label{e:|I4(z)|->0}
|I_4(z)|\le c\mu([0,|z|^2]),
\end{align}
which converges to $0$ as $|z|\to 0$. Then, relations \eqref{e:|I7(z)|->0}--\eqref{e:|I4(z)|->0} imply~\eqref{lim:ptilde/z=lim.zptilde=0}.

Next, we establish~\eqref{ineq:lambda+beta.ptilde(z)>0}. We break up the proof into two cases, depending on whether or not $z$ is real. We consider the first case $z=u+\imag v\in \bbC^-\sm$ where $v<0$ and $|z|^2=u^2+v^2 < 1$ . In view of expression~\eqref{e:w(z)_Hilbert} together with~\eqref{form:ptilde.qtilde_1}, we write
\begin{align*}
\lambda+\beta\ptil(z)&=\lambda-\beta\frac{\imag}{2\sqrt{\pi}}\int_0^\infty\close\frac{1}{\sqrt{x}}\int_{\rbb}\frac{e^{-t^2}}{ \frac{u}{2\sqrt{x}} +\imag \frac{v}{2\sqrt{x}}+t}\d t\,\mu(\d x).
\end{align*}
Then, in view of the fact that $1\big/ [(\frac{u}{2\sqrt{x}}+t) +\imag \frac{v}{2\sqrt{x}}] = [\frac{u}{2\sqrt{x}}+t - \imag \frac{v}{2 \sqrt{x}}]\big/ [(\frac{u}{2\sqrt{x}}+t)^2 + \frac{v^2}{4x}]$, we can bound
\begin{align*}
|\lambda+\beta\ptil(z)|&\ge\lambda+ \frac{\beta}{2\sqrt{\pi}}\int_0^\infty\close\frac{1}{\sqrt{x}} \Big\{ \int_{\rbb}\frac{e^{-t^2}}{ (\frac{u}{2\sqrt{x}}+t)^2 + \frac{v^2}{4x}}\d t \Big\} \frac{(-v)}{2\sqrt{x}}\mu(\d x)\\
&= \lambda + \frac{\beta}{2\sqrt{\pi}}\int_0^\infty\close\frac{1}{\sqrt{x}}\int_{\rbb}\frac{e^{-\frac{(t|v|-u  )^2}{4x}}}{ t^2 + 1}\d t\,\mu(\d x)\\
&\ge  \frac{\beta}{2\sqrt{\pi}}\int_0^\infty\close\frac{1}{\sqrt{x}}\int_{-1}^1\frac{e^{-\frac{(t|v|-u  )^2}{4x}}}{ t^2 + 1}\d t\,\mu(\d x).
\end{align*}
Now consider the elementary system of inequalities
\begin{align*}
(t|v|-u)^2\le 2(t^2v^2+u^2)\le 2,
\end{align*}
which holds for $|t|<1$ and $u^2+v^2<1$. Then,
\begin{align*}
|\lambda+\beta\ptil(z)|&\ge \frac{\beta}{2\sqrt{\pi}}\int_0^\infty\close\frac{1}{\sqrt{x}}e^{-\frac{1}{2x}}\mu(\d x) \cdot \int_{-1}^1\frac{1}{ t^2 + 1}\d t =\beta \frac{\sqrt{\pi}}{4}\int_0^\infty\close\frac{1}{\sqrt{x}}e^{-\frac{1}{2x}} \mu(\d x),
\end{align*}
where the equality follows from the elementary identity $\int_{-1}^1\frac{1}{t^2+1} dt =\frac{\pi}{2}$. This proves~\eqref{ineq:lambda+beta.ptilde(z)>0} for the case $\Im(z)<0$.

We now consider the second case where $z=u\in\rbb\sm$ and $|u|\le 1$. In view of expression~\eqref{e:w(x)} together with~\eqref{form:ptilde.qtilde_1}, we can write
\begin{align*}
\lambda+\beta \ptil(z)=\lambda+\beta\frac{\sqrt{\pi}}{2}\int_0^\infty\close\frac{1}{\sqrt{x}} \Big[e^{-\frac{u^2}{4x}}+\frac{2\imag}{\sqrt{\pi}}\daw\Big(\!-\frac{u}{2\sqrt{x}}\Big)\Big]\mu(\d x).
\end{align*}
Since $e^{-\frac{u^2}{4x}} \geq e^{-\frac{1}{2x}}$ for $|u|\leq 1$, this implies that
\begin{align*}
|\lambda+\beta \ptil(z)|\ge \beta\frac{\sqrt{\pi}}{2}\int_0^\infty\close\frac{1}{\sqrt{x}} e^{-\frac{u^2}{4x}}\mu(\d x)&\ge \beta\frac{\sqrt{\pi}}{4}\int_0^\infty\close\frac{1}{\sqrt{x}}e^{-\frac{1}{2x}}\mu(\d x).
\end{align*}
This proves~\eqref{ineq:lambda+beta.ptilde(z)>0}. Hence, part (a) is established.\\

\noindent (b) By part (a), $\qtil_1(z)$ is analytic. It remains to show that $\qtil_1(z)$ does not admit any root in $\bbC^-\sm$.

Similarly to the proof of~\eqref{ineq:lambda+beta.ptilde(z)>0}, we consider two cases depending on whether or not $z$ is real.

Let $z= u + \imag v$, $v < 0$, and consider $\qtil_1(z)$ as in \eqref{form:ptilde.qtilde_1}. Note that, by expression \eqref{e:w(z)_Hilbert},
$$
\qtil_1(z)=\lambda + \imag m z + \beta \ptil(z)=\lambda+\imag m(u+\imag v)-\beta\frac{\imag}{2\sqrt{\pi}}\int_0^\infty\close\frac{1}{\sqrt{x}}\int_{\rbb}\frac{e^{-t^2}}{ \frac{u}{2\sqrt{x}} +\imag \frac{v}{2\sqrt{x}}+t}\d t\,\mu(\d x).
$$
So,
$$
\Re(\qtil_1(z))= \lambda+ \frac{\beta}{2\sqrt{\pi}}\int_0^\infty\close\frac{1}{\sqrt{x}}\int_{\rbb}\frac{e^{-t^2}}{ (\frac{u}{2\sqrt{x}}+t)^2 + \frac{v^2}{4x}}\d t\frac{(-v)}{2\sqrt{x}}\mu(\d x)+ m (-v) > 0.
$$
Therefore, $\qtil_1(z)$ has no root in $\{z \in \bbC: \Im(z) < 0\}$.

Now fix $z = u \in \bbR\sm$ (i.e., $v = 0$). Then, expression \eqref{e:w(x)} implies that we can write
$$
\qtil_1(z) =
\lambda + \imag  m u +\beta \ptil(z)=\lambda+\beta\frac{\sqrt{\pi}}{2}\int_0^\infty\close\frac{1}{\sqrt{x}} \Big[e^{-\frac{u^2}{4x}}+\frac{2\imag}{\sqrt{\pi}}\daw\Big(\!-\frac{u}{2\sqrt{x}}\Big)\Big]\mu(\d x)+ \imag  m u .
$$
In particular,
$$
\Re(z)= \lambda +\beta\frac{\sqrt{\pi}}{2}\int_0^\infty\close\frac{1}{\sqrt{x}}e^{-\frac{u^2}{4x}}\mu(\d x)>0.
$$
It follows that $\qtil_1(z)$ has no root in $z\in\rbb\sm$. This concludes the proof.

\end{proof}

\subsubsection{\bf{Harmonically bounded case} ($\boldsymbol{\gamma>0}$)} \label{sec:equipartition:2D:gamma>0}

We now turn to the proof of Theorem~\ref{thm:equipartition-of-energy:2D:gamma>0:powerlaw.exponential}. Similarly to the previous subsection, let
\begin{equation} \label{form:f_2(z)}
f_2(z)=\frac{1}{z\big(\lambda z+\beta z\big[\Kcos(z)+\imag\Ksin(z)\big] + \imag (\gamma-mz^2)\big)}.
\end{equation}

\begin{remark} \label{rem:f_2:well-defined}
Similarly to Remark~\ref{rem:f_1:well-defined}, we note that in formula~\eqref{form:f_2(z)}, $\Kcos(z)+\imag\Ksin(z)$ is understood in the sense of~\eqref{form:complete-monotone:Kcos-Ksin} and \eqref{form:complete-monotone:Kcos-Ksin:K(t)=phi(t^2)} extended to $\cbb$. Furthermore, $\Kcos(z)+\imag\Ksin(z)$ is actually analytic on suitable subspaces of $\cbb$ (see the proof of Theorem~\ref{thm:equipartition-of-energy:2D:gamma>0:powerlaw.exponential}).
\end{remark}

For a large constant $R>0$, recall that $C_R^+ $ and $C_{1/R}^+ $ are, respectively, the outer and inner half circles as in~\eqref{form:curve:C_R.and.C_(1/R)}. Also consider the following closed curve $\Ctilde(R)\subset\cbb^+$ oriented counterclockwise
\begin{equation}\label{form:curve:Ctilde(R)}
\Ctilde(R)=[1/R,R]\cup C_R^+  \cup [-R,-1/R]\cup C_{1/R}^+ .
\end{equation}
Our approach is similar to that in the proof of Theorem~\ref{thm:equipartition-of-energy:1D:gamma=0:power-law.exponential}. We essentially need to show that $f_2$ as in \eqref{form:f_2(z)} is analytic on the upper half plane $\cbb^+$. Once this is accomplished, in view of Cauchy's theorem for contour integrals, we are then able to establish Theorem~\ref{thm:equipartition-of-energy:2D:gamma>0:powerlaw.exponential}, whence equipartition of energy holds for~\eqref{eq:GLE:2D}. Some statements appear in the auxiliary Lemmas \ref{lem:equipartition-of-energy:2D:gamma>0}--\ref{lem:qtilde_2:gamma>0}, whose proofs are deferred to the end of the section.

First, in the following lemma we provide sufficient conditions on $f_2$ and on the Fourier transform of $K$ for equipartition of energy to hold.
\begin{lemma} \label{lem:equipartition-of-energy:2D:gamma>0}
Let $(x(t),v(t))$ be the stationary process associated with the weak solution $(X,V)$ of~\eqref{eq:GLE:2D}. Suppose that, for $z\in\cbb^+\setminus\{0\}$,
\begin{equation} \label{cond:lim.Kcos+iKsin=0}
\lim_{|z|\to\infty}\frac{|\Kcos(z)+\imag \Ksin(z)|}{|z|}=0=\lim_{|z|\to0}\big|z\big(\Kcos(z)+\imag \Ksin(z)\big)\big|.
\end{equation}
Let $f_2(z)$ be as in~\eqref{form:f_2(z)} and let $\Ctilde(R)$ be the curve as in~\eqref{form:curve:Ctilde(R)}. Then, the following holds.
\begin{itemize}
\item [(a)] If, for all large enough $R>0$,
\begin{equation} \label{cond:int_Ctilde(R) f_2(z)=0}
\int_{\Ctilde(R)}\close f_2(z)\d z=0,
\end{equation}
then
\begin{equation}\label{e:E[x(0)^2]=kB*T}
\E[\gamma x(0)^2]=k_BT.
\end{equation}
\item [(b)] If, for all large enough $R>0$,
\begin{equation} \label{cond:int_C(R) z^2f_2(z)=0}
\int_{\Ctilde(R)}\close z^2f_2(z)\d z=0,
\end{equation}
then
\begin{equation}\label{e:E[mv(0)^2]=kB*T}
\E[m v(0)^2]=k_BT.
\end{equation}
\end{itemize}
\end{lemma}

The proof of Theorem~\ref{thm:equipartition-of-energy:2D:gamma>0:powerlaw.exponential} is based on verifying that the assumptions of Lemma~\ref{lem:equipartition-of-energy:2D:gamma>0} are met. To this end, we show that $f_2$ is analytic on $\cbb^+\sm$, which is established based on the following lemmas.
\begin{lemma} \label{lem:q_2:gamma>0} Let $\mu$ be the representation measure on $[0,\infty)$ for $K\in\CM$ as in Theorem~\ref{thm:completely-monotone:Bernstein}. Let $q_2(z)$ be the function defined on $\cbb^*\setminus\{0\}$ and given by
\begin{equation} \label{form:q_2}
q_2(z)=\gamma-\lambda z +mz^2+\beta zp(z),
\end{equation}
where $p(z)=\int_0^\infty \frac{1}{z-x}\mu(\d x)$ is as in~\eqref{form:p.q_1}. Then,
\begin{itemize}
\item [(a)] $q_2(z)$ is analytic on $\cbb^*\setminus\{0\}$; and
\item [(b)] $q_2(z)$ does not admit any root in $\cbb^*\setminus\{0\}$.
\end{itemize}
\end{lemma}

\begin{lemma} \label{lem:qtilde_2:gamma>0} Suppose $K(t)=\f(t^2)$ where $\f\in\CM$. Let $\mu$ be the representation measure on $[0,\infty)$ for $\f$ as in Theorem~\ref{thm:completely-monotone:Bernstein}. Let $\qtil_2(z)$ be the function defined on $\cbb^+\setminus\{0\}$ and given by
\begin{equation} \label{form:qtilde_2}
\qtil_2(z)=\lambda z+\beta z\ptil(-z) +\imag (\gamma-mz^2),
\end{equation}
where $\ptil(z)=\frac{\sqrt{\pi}}{2}\int_0^\infty\close \frac{1}{\sqrt{x}}w\Big(\!-\frac{z}{2\sqrt{x}}\Big)\mu(\d x)$ is as in~\eqref{form:ptilde.qtilde_1}. Then,
\begin{itemize}
\item [(a)] $\qtil_2(z)$ is analytic on $\cbb^+\setminus\{0\}$; and
\item [(b)] $\qtil_2(z)$ does not admit any root in $\cbb^+\setminus\{0\}$.
\end{itemize}
\end{lemma}

We now provide the proof of Theorem~\ref{thm:equipartition-of-energy:2D:gamma>0:powerlaw.exponential}. 
\begin{proof}[Proof of Theorem~\ref{thm:equipartition-of-energy:2D:gamma>0:powerlaw.exponential}] Similarly to the proof of Theorem~\ref{thm:equipartition-of-energy:1D:gamma=0:power-law.exponential}, we first consider the case where $K\in \CM$. It suffices to check that the conditions for Lemma~\ref{lem:equipartition-of-energy:2D:gamma>0} are met using auxiliary results provided in Lemma~\ref{lem:q_2:gamma>0}.

We first verify the limit condition~\eqref{cond:lim.Kcos+iKsin=0}. First note that $z\in \cbb^+\setminus\{0\}$ implies $\imag z\in \cbb^*\setminus\{0\}$. Then, as a consequence of relations~\eqref{form:complete-monotone:Kcos-Ksin} extended to $\cbb^+\setminus\{0\}$,
$$\Kcos(z)+\imag \Ksin(z)=\int_0^\infty\close \frac{x+\imag z}{x^2+z^2}\mu(\d x)=-\int_0^\infty\close\frac{1}{\imag z-x}\mu(\d x)=-p(\imag z),$$
where $p(z)$ is as in~\eqref{form:p.q_1}. Condition~\eqref{cond:lim.Kcos+iKsin=0} now follows immediately from Lemma~\ref{lem:p(x).q_1(x)}, (a), cf.~\eqref{lim:p/z=lim.zp=0}.

To verify the contour integral conditions~\eqref{cond:int_Ctilde(R) f_2(z)=0} and~\eqref{cond:int_C(R) z^2f_2(z)=0} for all large enough $R$, it suffices to prove that $f_2(z)$ is, indeed, analytic on $\cbb^+\setminus\{0\}$. To this end, recast
\begin{align*}
f_2(z)&=\frac{1}{z\big(\lambda z-\beta z\int_0^\infty\frac{1}{\imag z-x}\mu(dx)+\imag (\gamma-mz^2)\big)}\\
&=\frac{1}{\imag z\big( \gamma-\lambda (\imag z)+m(\imag z)^2+\beta(\imag z)\int_0^\infty\frac{1}{\imag z-x}\mu(\d x)    \big)}\\
&=\frac{1}{\imag zq_2(\imag z)},
\end{align*}
where $q_2(\cdot)$ is as in~\eqref{form:q_2}. Since $z\in\cbb^+\sm$, then $\imag z\in\cbb^*\sm$. Also, in view of Lemma~\ref{lem:q_2:gamma>0}, $q_2(\cdot)$ is analytic and does not have poles in $\cbb^*\setminus\{0\}$. It then follows immediately that, for $z\in \cbb^+\setminus\{0\}$, $(\imag z q_2(\imag z))^{-1}=f_2(z)$ is analytic on $\cbb^+\setminus\{0\}$, which clearly implies the contour integral conditions~\eqref{cond:int_Ctilde(R) f_2(z)=0} and~\eqref{cond:int_C(R) z^2f_2(z)=0}.

We now turn to the case where $K(t)=\f(t^2)\in\CM$. We have to verify the assumptions of Lemma~\ref{lem:equipartition-of-energy:2D:gamma>0} by  combining auxiliary results in Lemmas~\ref{lem:ptilde.qtilde_1} and \ref{lem:qtilde_2:gamma>0}. Recall that $\ptil(z)$ and $\qtil_2(z)$, respectively, are given by~\eqref{form:ptilde.qtilde_1} and \eqref{form:qtilde_2}. In light of~\eqref{form:complete-monotone:Kcos-Ksin:K(t)=phi(t^2)}, for $z\in\cbb^+\sm$,
\begin{align*}
\Kcos(z)+\imag\Ksin(z)=\ptil(-z),\quad\text{and}\quad f_2(z)=\frac{1}{z\qtil_2(z)}.
\end{align*}
The limits in~\eqref{cond:lim.Kcos+iKsin=0} are now a consequence of the fact that, for $z\in\cbb^+\sm$,
\begin{align}\label{e:lim_|p-tilde(-z)/z|=0=lim_|z*p-tilde(-z)|}
\lim_{|z|\to \infty}\frac{|\ptil(-z)|}{|z|}=0=\lim_{|z|\to0}\big|z\ptil(-z)\big|,
\end{align}
where \eqref{e:lim_|p-tilde(-z)/z|=0=lim_|z*p-tilde(-z)|} follows from Lemma~\ref{lem:ptilde.qtilde_1}, (a), cf.~\eqref{lim:ptilde/z=lim.zptilde=0}.

Also, by Lemma~\ref{lem:qtilde_2:gamma>0}, the function $f_2(z)=(z\qtil_2(z))^{-1}$ is analytic on $\cbb^+\sm$. It follows that the contour integral conditions~\eqref{cond:int_Ctilde(R) f_2(z)=0} and~\eqref{cond:int_C(R) z^2f_2(z)=0} hold for all large enough $R$. Consequently, the assumptions of Lemma~\ref{lem:equipartition-of-energy:2D:gamma>0} have been verified, which implies that relations~\eqref{e:E[x(0)^2]=kB*T} and~\eqref{e:E[mv(0)^2]=kB*T} hold. Thus, the proof is complete.
\end{proof}

We now provide the proof of Lemma~\ref{lem:equipartition-of-energy:2D:gamma>0}.

\begin{proof}[Proof of Lemma~\ref{lem:equipartition-of-energy:2D:gamma>0}]

(a) First, recall from~\eqref{form:spectral-density:x} and covariance relation~\eqref{form:stationary-distribution:E[<G,phi1><G,phi2>]} that
 \begin{align}\label{e:E[x(0)^2]}
\E[x(0)^2]&=\frac{k_BT}{2\pi}\int_{\rbb} r_{11}(\omega)\d\omega \nonumber\\
&= \frac{k_BT}{\pi}\int_0^\infty \close\frac{2\big(\lambda+\beta\Kcos(\omega)\big)}{\big|\imag (\gamma-m\omega^2-\beta \omega\Ksin(\omega))+\lambda\omega+\beta\omega\Kcos(\omega))\big|^2}\d\omega.
\end{align}
Hence, it suffices to prove that the value of the integral \eqref{e:E[x(0)^2]} is $\pi\gamma^{-1}$.

Let $f_2(z)$ be the function as in~\eqref{form:f_2(z)}. Then, the contour integral of $f_2$ on $\Ctilde(R)$ may be decomposed into
\begin{align*}
\int_{\Ctilde(R)}f_2(z)\d z&=\Big\{\int_{1/R}^R+\int_{C_R^+ }+\int_{-R}^{-1/R}\close+\int_{C_{1/R}^+ }\Big\}f_2(z)\d z\\
&=I_1 (R) +I_2(R) +I_3(R) +I_4(R).
\end{align*}
We now proceed to reexpress and establish the limiting behavior of each integral $I_i(R)$, $i = 1,2,3,4$, as $R \rightarrow \infty$. First, note that
$$
I_1(R) =\int_{1/R}^R \frac{\d \omega}{\omega(\lambda \omega+\beta \omega\Kcos(\omega) + \imag (\gamma-m\omega^2+\beta \omega\Ksin(\omega)))},
$$
where, once again, we use the notation $\omega$ for integration along the real axis. Turning to $I_3(R)$, recall that the function $\Kcos$ is even, whereas $\Ksin$ is odd. Thus, by a change of variable $z:=-\omega$, we can reexpress
\begin{align*}
I_3(R) &=\int_{R}^{1/R}\close \frac{-\d \omega}{-\omega(-\lambda \omega-\beta \omega\Kcos(-\omega) + \imag (\gamma-m\omega^2+\beta (-\omega)\Ksin(-\omega)))}\\
&=-\int_{1/R}^{R} \frac{\d \omega}{\omega(-\lambda \omega-\beta \omega\Kcos(\omega) + \imag (\gamma-m\omega^2+\beta \omega\Ksin(\omega)))}.
\end{align*}
Therefore,
$$
I_1(R) +I_3(R) =\int_{1/R}^R\frac{2\lambda+2\beta\Kcos(\omega)}{|\gamma-m\omega^2+\beta\omega\Ksin(\omega)|^2+\omega^2|\lambda+\beta\Kcos(\omega)|^2}\d\omega.
$$
By virtue of the Monotone Convergence Theorem,
\begin{equation}\label{e:I_1(R)+I_1(R)_->_int_r11}
I_1(R) +I_3(R)\to \int_0^\infty\close r_{11}(\omega)\d\omega,\quad\text{as }R\to\infty.
\end{equation}
With regards to $I_2(R)$, for $z\in C_{R}\subset\cbb^+\setminus\{0\}$, recast $f_2(z)$ as
\begin{align*}
f_2(z)&=\frac{1}{z\big(\lambda z+\beta z(\Kcos(z)+\imag \Ksin(z))+\imag (\gamma-mz^2)  \big)}\\
&=\frac{1}{z^3\big(\lambda z^{-1}+\beta z^{-1}(\Kcos(z)+\imag \Ksin(z))+\imag (\gamma z^{-2}-m)  \big)}.
\end{align*}
By making the change of variable $z:=Re^{\imag \theta}$,
\begin{equation}\label{e:I_2(R)}
\begin{aligned}
&I_2(R)\\
&= \int_0^{\pi} \frac{\imag \,\d\theta  }{R^2e^{\imag 2\theta}\big[\lambda R^{-1}e^{-\imag \theta}+\beta R^{-1}e^{-\imag \theta}[\Kcos(Re^{\imag \theta}))+\imag \Ksin(Re^{\imag \theta})]+\imag (\gamma R^{-2}e^{-\imag 2\theta}-m   )  \big]}.
\end{aligned}
\end{equation}
By~\eqref{cond:lim.Kcos+iKsin=0}, as $R\to\infty$, the integrand in \eqref{e:I_2(R)} converges to $0$. Hence, by the Dominated Convergence Theorem,
\begin{equation}\label{e:I_2(R)_->_0}
I_2(R)\to 0, \quad\text{as }R\to\infty.
\end{equation}
Similarly, by making a change of variable $z=R^{-1}e^{\imag \theta}$, we have
\begin{align}\label{e:I_4(R)}
I_4(R)= \int_{\pi}^0 \frac{\imag \d\theta  }{\lambda R^{-1}e^{\imag \theta}+\beta R^{-1}e^{\imag \theta}[\Kcos(R^{-1}e^{\imag \theta})+\imag \Ksin(R^{-1}e^{\imag \theta})]+\imag (\gamma-m R^{-2}e^{\imag 2\theta}) }.
\end{align}
As $R\to\infty$, again by relation~\eqref{cond:lim.Kcos+iKsin=0}, the integrand in \eqref{e:I_4(R)} converges to $\gamma^{-1}$. In light of the Dominated Convergence Theorem, this implies that
\begin{equation}\label{e:I_4(R)_->_-pi*gamma^(-1)}
I_4(R)\to -\pi\gamma^{-1}, \quad\text{as } R\to\infty.
\end{equation}
Collecting the limits \eqref{e:I_1(R)+I_1(R)_->_int_r11}, \eqref{e:I_2(R)_->_0} and \eqref{e:I_4(R)_->_-pi*gamma^(-1)}, we obtain
$$
0=\lim_{R\to\infty}\int_{C(R)}\close f_2(z)\d z= \int_0^\infty \close r_{11}(\omega)\d\omega-\pi\gamma^{-1}.
$$
Hence,
$$
\E[\gamma x(0)^2]=\gamma\frac{k_BT}{\pi}\int_0^\infty \close r_{11}(\omega)\d\omega=k_BT,
$$
which establishes \eqref{e:E[x(0)^2]=kB*T}.

With regards to $\E[v(0)^2]$, expressions \eqref{form:spectral-density:x} and \eqref{form:spectral-density:v} imply that
\begin{align}\label{e:E[v(0)^2]}
\E[v(0)^2]&=\frac{k_BT}{2\pi}\int_{\rbb}r_{22}(\omega)\d\omega \nonumber \\
&= \frac{k_BT}{\pi}\int_0^\infty \close\frac{2\omega^2\big(\lambda+\beta\Kcos(\omega)\big)}{\big|\imag (\gamma-m\omega^2-\beta \omega\Ksin(\omega))+\lambda\omega+\beta\omega\Kcos(\omega))\big|^2}\d\omega.
\end{align}
We now claim that the integral in \eqref{e:E[v(0)^2]} is equal to $\pi m^{-1}$. Similarly to the argument for $\E[x(0)^2]$, to establish this we consider the contour integral
\begin{align}\label{e:int_C-tilde(R)z^2f(z)dz}
\int_{\Ctilde(R)}z^2f(z)\d z&=\Big\{\int_{1/R}^R+\int_{C_R^+ }+\int_{-R}^{-1/R}\close+\int_{C_{1/R}^+ }\Big\}z^2f_2(z)\d z\\
&=I_5 (R) +I_6(R) +I_7(R) +I_8(R).
\end{align}
By calculations analogous to those for $I_1(R)+I_3(R)$,
\begin{equation}\label{e:I_5(R)+I_7(R)->_int}
I_5(R)+I_7(R)\to \int_0^\infty\close \frac{2\omega^2\big(\lambda+\beta\Kcos(\omega)\big)}{\big|\imag (\gamma-m\omega^2-\beta \omega\Ksin(\omega))+\lambda\omega+\beta\omega\Kcos(\omega))\big|^2}\d\omega, \quad\text{as } R\to\infty,
\end{equation}
where the limit is a consequence of the Monotone Convergence Theorem.

On the other hand, in regard to integration along the outer half circle $C^+_R$, note that we can express
$$
z^2f(z)=\frac{1}{z\big(\lambda z^{-1}+\beta z^{-1}(\Kcos(z)+\imag \Ksin(z))+\imag (\gamma z^{-2}-m)  \big)}.
$$
Thus, by making a change of variable $z= Re^{\imag \theta} $, we obtain
\begin{align*}
I_6(R)&= \int_0^\pi \frac{\imag \,\d\omega}{\lambda R^{-1}e^{-\imag \theta}+\beta R^{-1}e^{-\imag \theta}[\Kcos(Re^{\imag \theta}))+\imag \Ksin(Re^{\imag \theta})]+\imag (\gamma R^{-2}e^{-\imag 2\theta}-m   ) }.
\end{align*}
Consequently, in light of the Dominated Convergence Theorem together with condition~\eqref{cond:lim.Kcos+iKsin=0},
\begin{equation}\label{e:I_6(R)->_-pi*m^(-1)}
I_6(R)\to -\pi m^{-1},\quad\text{as }R\to\infty.
\end{equation}
Likewise, by a change of variable $z=R^{-1}e^{\imag \theta}$,
\begin{align}\label{e:I_8(R)}
I_8(R)&= \int_0^\pi \frac{R^{-3}e^{\imag 3\theta}\imag \,\d\omega}{\lambda R^{-1}e^{\imag \theta}+\beta R^{-1}e^{\imag \theta}[\Kcos(R^{-1} e^{\imag \theta}))+\imag \Ksin(R^{-1}e^{\imag \theta})]+\imag (\gamma -mR^{-2}e^{\imag 2\theta}   ) }.
\end{align}
By~\eqref{cond:lim.Kcos+iKsin=0}, the integrand in \eqref{e:I_8(R)} converges to zero. Therefore, by the Dominated Convergence Theorem,
\begin{equation}\label{e:I_8(R)->0}
I_8(R)\to 0, \quad\text{as }R\to\infty.
\end{equation}
Based on expressions \eqref{e:int_C-tilde(R)z^2f(z)dz}, \eqref{e:I_5(R)+I_7(R)->_int}, \eqref{e:I_6(R)->_-pi*m^(-1)} and \eqref{e:I_8(R)->0}, we obtain the limit
$$
0=\lim_{R\to\infty}\int_{C(R)}\close z^2f(z)\d z=\int_0^\infty\close r_{22}(\omega)\d\omega-\pi m^{-1}.
$$
Hence,
$$
\E[mv(0)^2]=m\frac{k_BT}{\pi}\int_0^\infty\close r_{22}(\omega)\d\omega=k_BT.
$$
This establishes \eqref{e:E[mv(0)^2]=kB*T}.
\end{proof}

We now provide the proof of Lemma \ref{lem:q_2:gamma>0}.

\begin{proof}[Proof of Lemma \ref{lem:q_2:gamma>0}]
In view of Lemma~\ref{lem:p(x).q_1(x)} (a), $p(z)$ is analytic in $\cbb^*\sm$, and so is $q_2(z)$. This shows $(a)$.

We now show $(b)$, i.e., we prove that $q_2(z)$ does not admit any root in $\cbb^*\setminus\{0\}$. Similarly to the proof of Lemma~\ref{lem:p(x).q_1(x)} for $q_1(z)$, first observe that $q_2(z)$ cannot have a (real) negative real root. Indeed, if $z<0$, then
$$
q_2(z)=\gamma+\lambda |z| +m|z|^2+\beta |z|\int_0^\infty\close\frac{1}{|z|+x}\mu(\d x) > 0.
$$
Next, by means of contradiction, suppose that $z_0 =-u+\imag v$, is a root of $q_2(z)$, where $u\ge 0$ and
\begin{equation}\label{e:v_neq_0}
v\neq 0.
\end{equation}
A routine calculation shows that the condition $q_2(z_0) = 0$ is equivalent to
\begin{align*}
0&=\gamma+m(u^2-v^2)-\imag 2uv+\lambda u-\imag \lambda v+\beta(-u+\imag v)\int_0^\infty\close\frac{1}{-u-x+\imag v}\mu(\d x)\\
&=\gamma+m(u^2-v^2)+\lambda u+\int_0^\infty\frac{\beta u (u+x)+\beta v^2}{(u+x)^2+v^2}\mu(\d x)\\
&\qquad -\imag v\Big(2mu+\lambda +\beta \int_0^\infty\close\frac{x}{(u+x)^2+v^2}\mu(\d x) \Big).
\end{align*}
Consequently, $v=0$, which contradicts \eqref{e:v_neq_0}. This establishes $(b)$.
\end{proof}

We now provide the proof of Lemma~\ref{lem:qtilde_2:gamma>0}. The argument is essentially the same as that for proving Lemma~\ref{lem:ptilde.qtilde_1}, (b).

\begin{proof}[Proof of Lemma~\ref{lem:qtilde_2:gamma>0}]
By Lemma~\ref{lem:ptilde.qtilde_1}, (a), $\ptil(z)$ is analytic on $\bbC^-\sm$. Therefore, $\qtil_2(z)=\lambda z+\beta z\ptil(-z) +\imag (\gamma-mz^2)$ is analytic on $\bbC^+\sm$. This establishes (a).

We now turn to (b), i.e., we want to show that $\qtil_2(z)$ does not admit any root in $\bbC^+\sm$. To see this, first recall that $\qtil_1(z)$ is given by~\eqref{form:ptilde.qtilde_1}. Then, $\qtil_2(z)$ as in~\eqref{form:qtilde_2} can be rewritten as
\begin{align*}
\qtil_2(z)=z\Big(\lambda+\beta\ptil(-z)+\imag m(-z)-\imag \frac{\gamma}{-z}\Big  )=z\Big(\qtil_1(-z)-\imag \frac{\gamma}{-z}\Big).
\end{align*}
Thus, it suffices to show that $\qtil_1(z)-\imag \frac{\gamma}{z}$ does not admit any root in $\bbC^-\sm$. The argument for showing this is similar to the one in the proof of Lemma~\ref{lem:ptilde.qtilde_1}, (b), and involves two cases, i.e., for real and non-real $z\in \bbC^-\sm$.

Assume first that $\Im(z)<0$. Write $z=u+\imag v$, $v<0$ and note that, by expression \eqref{e:w(z)_Hilbert},
$$
\qtil_1(z)-\imag \frac{\gamma}{z}=\lambda+\imag m(u+\imag v)-\beta\frac{\imag}{2\sqrt{\pi}}\int_0^\infty\close\frac{1}{\sqrt{x}}\int_{\rbb}\frac{e^{-t^2}}{ \frac{u}{2\sqrt{x}} +\imag \frac{v}{2\sqrt{x}}+t}\d t\,\mu(\d x)-\imag \frac{\gamma}{z}.
$$
After a routine calculation, we obtain
\begin{equation}\label{e:Re(q1-tilde(z)-igamma/z)}
\Re\Big(\qtil_1(z)-\imag\frac{\gamma}{z} \Big)= \lambda+ \frac{\beta}{2\sqrt{\pi}}\int_0^\infty\close\frac{1}{\sqrt{x}}\int_{\rbb}\frac{e^{-t^2}}{ (\frac{u}{2\sqrt{x}}+t)^2 + \frac{v^2}{4x}}\d t\frac{(-v)}{2\sqrt{x}}\mu(\d x)+ m (-v)+\frac{\gamma(-v)}{u^2+v^2} > 0.
\end{equation}
Alternatively, assume $z = u \in \bbR\sm$. Then, by expression \eqref{e:w(x)},
$$
\qtil_1(u) -\imag \frac{\gamma}{u}=\lambda+\beta\frac{\sqrt{\pi}}{2}\int_0^\infty\close\frac{1}{\sqrt{x}} \Big[e^{-\frac{u^2}{4x}}+\frac{2\imag}{\sqrt{\pi}}\daw\Big(\!-\frac{u}{2\sqrt{x}}\Big)\Big]\mu(\d x)+ \imag  m u -\imag \frac{\gamma}{z}.
$$
Hence,
\begin{equation}\label{e:Re(q1-tilde(u)-igamma/u)}
\Re\Big(\qtil_1(u)-\imag\frac{\gamma}{u} \Big)= \lambda +\beta\frac{\sqrt{\pi}}{2}\int_0^\infty\close\frac{1}{\sqrt{x}}e^{-\frac{u^2}{4x}}\mu(\d x)>0.
\end{equation}
By \eqref{e:Re(q1-tilde(z)-igamma/z)} and \eqref{e:Re(q1-tilde(u)-igamma/u)}, we conclude that $\qtil_1(z)-\imag\gamma/z$ has no root in $z\in\cbb^-\sm$, and that neither does $\qtil_2(z)$ in $\cbb^+\sm$. This finishes the proof.
\end{proof}

\section*{acknowledgment}

The authors would like to thank two anonymous reviewers for their helpful comments and suggestions.

\appendix

\section{Stationary Random Operators}  \label{sec:stationary-operator}

In this section, we review and generalize the framework of stationary distributions~\cite{didier2020asymptotic,ito1954stationary,  mckinley2018anomalous,yaglom1957some}. The goal is to construct stationary random operators for the purpose of analyzing the well--posedness of the 2D GLE.

Hereinafter, $^*$ denotes Hermitian transposition and $\|\cdot\|$ denotes the operator norm. ${\mathcal H}_{\geq 0}(d,\bbC)$ and ${\mathcal M}(d,\bbC)$ denote, respectively, the convex cone of Hermitian positive semidefinite matrices and the space of $d \times d$, entry--wise $\bbC$-valued matrices.

Given $d\in\nbb$, $\ubf(t)=(u_1(t),\dots,u_d(t))^T$ denotes a $\bbC^d$--valued stochastic process. We now briefly recall the definitions of weak stationarity and mean squared continuity.
\begin{definition} \label{def:weak-stationary-process}
 A stochastic process $\{\ubf(t)\}_{t \in \rbb}$ is said to be weakly stationary if, for all $t, s\in\rbb$,
		\begin{itemize}
			\item[(a)] $\E \|\ubf(t) \ubf(t)^*\| <\infty$;
		\item[(b)] $\E[\ubf(t)]=\ubf$, for some constant vector $\ubf$ (we may assume $\ubf=\boldsymbol{0}$); and
		\item[(c)] the covariance matrix $\E\big[\ubf(t)\ubf(s)^* \big]$ only depends on the difference $t-s$.
		\end{itemize}
\end{definition}

\begin{definition} \label{def:weak-stationary-process:mean-squared-continous}
 A second order stochastic process $\{\ubf(t)\}_{t \in \rbb}$ is said to be mean squared continuous if, for all $t \in\rbb$, $\lim_{h\to 0}\E  (\ubf(t+h)-\ubf(t))^*(\ubf(t+h)-\ubf(t)) =0$.
\end{definition}

In the following theorem, we recall the fact that, under mild conditions, the covariance structure of a weakly stationary process is characterized by its so-named spectral measure (see also \cite{rozanov:1967}, \cite[Theorem 7.1]{lindgren2012stationary} and \cite[Chapter 4]{brockwell:davis:1991}).
\begin{theorem}\label{thm:weak-stationary-process:characterization} A mean squared continuous process $\{\ubf(t)\}_{t \in \bbR}$ is weakly stationary if and only if its matrix--valued covariance function has the representation
\begin{equation}\label{e:E[u(t)u(s)*]}
\E\big[\ubf(t)\ubf(s)^*\big]=\int_\rbb e^{\imag (t-s)\omega}\nu(\d\omega), \quad t,s \in \bbR.
\end{equation}
In \eqref{e:E[u(t)u(s)*]},
\begin{equation}\label{cond:weak-stationary-process:Hermite-positive-definite}
\nu(\d\omega)=(\nu_{ij}(\d\omega))_{1\le i,j\le d} \in {\mathcal H}_{\geq 0}(d,\bbC)
\end{equation}
is a matrix-valued Borel measure such that $\|\nu(\bbR)\|< \infty$.
\end{theorem}
\begin{remark} \label{rem:weak-stationary-process:characterization} If the matrix--valued measure $\nu(\d\omega)$ is entry--wise absolutely continuous with respect to the Lebesgue measure, then we can write
$$
\nu(\d\omega)=f(\omega)\d\omega
$$
for some entry-wise integrable function $f$ taking values in ${\mathcal H}_{\geq 0}(d,\bbC)$ a.e.\ (cf.\ \cite[Theorem 7.1]{lindgren2012stationary}). The function $f$ is called the \emph{spectral density} of $\ubf(t)$.
\end{remark}

Analogously, we briefly recall the notion of \emph{stationary random distributions}, a generalization of multivariate stationary processes, first introduced in~\cite{gelfand1955generalized,ito1954stationary}. So, let $\tau_y$ be the shift operator given by $\tau_y\varphi(x):=\varphi(x-y)$ for any $\f\in\Sc$. Also, let $L^2(\Omega)$ be the space of all complex--valued random variables with finite variance. We now provide the definition of a stationary distribution (see \cite[Section 1]{yaglom1957some}).
\begin{definition} \label{def:stationary-distribution} A linear functional $F:\Sc\to L^2(\Omega)^d$ given by $\la F,\f\ra=( \la F_1,\f \ra, \dots,\la F_d,\f\ra)^T$ is called a \emph{stationary random distribution} on $\Sc$ if the following two conditions hold.
\begin{itemize}
\item[(a)] For all $y\in\rbb$ and for all $\f\in\Sc$, $\E\la F,\tau_y\f\ra=\E\la F,\f\ra$; and

\item[(b)] for all $y\in\rbb$ and for all $\varphi_1,\varphi_2\in\Sc$,
$$\E\big[ \la F,\tau_y \varphi_1\ra  \la F,\tau_y\varphi_2  \ra^* \big]=\E\big[ \la F, \varphi_1\ra  \la F,\varphi_2  \ra^*\big].$$
\end{itemize}

\end{definition}

Analogously to Theorem~\ref{thm:weak-stationary-process:characterization}, we have the following characterization of the second order structure of a stationary distribution in terms of covariance functionals and spectral measures. See also \cite{lindgren2012stationary,yaglom1957some}.

\begin{theorem}\label{thm:weak-stationary-distribution:characterization}  A linear functional $F:\Sc\to L^2(\Omega)^d$ is a stationary random distribution on $\Sc$ if and only if its covariance matrix
$B(\f_1,\f_2)$ has the representation
\begin{align} \label{form:stationary-distribution:E[<G,phi1><G,phi2>]}
B(\f_1,\f_2):=\E\big[  \la F, \varphi_1\ra  \la F,\varphi_2  \ra^*\big]=\int_\rbb \widehat{\f_1}(\omega)\overline{\widehat{\f_2}(\omega)}\nu(\d\omega).
\end{align}
In \eqref{form:stationary-distribution:E[<G,phi1><G,phi2>]}, $\nu(\d\omega)=(\nu_{ij}(\d\omega))_{1\le i,j\le d}$ is a ${\mathcal H}_{\geq 0}(d,\bbC)$--valued measure such that, for some $p\in\rbb$, 
\begin{align} \label{cond:weak-stationary-distribution:int.nu(d.omega)/(1+omega^2)^p<infty}
\int_{\rbb} \frac{\|\nu(\d\omega)\|}{(1+\omega^2)^p}<\infty.
\end{align}
\end{theorem}
\begin{remark} \label{rem:weak-stationary-distribution:characterization} \emph{(a)} Note that, due to condition~\eqref{cond:weak-stationary-distribution:int.nu(d.omega)/(1+omega^2)^p<infty} and to the fact that $\f$ is a Schwartz function, $\la F, \varphi\ra$ is, indeed, an element of $L^2(\Omega)^d$.

\emph{(b)} Similarly to Remark~\ref{rem:weak-stationary-process:characterization}, in case the measure $\nu$ as in Theorem~\ref{thm:weak-stationary-distribution:characterization} has the form $\nu(\d\omega)=f(\omega)\d\omega$ for some a.e.\ ${\mathcal H}_{\geq 0}(d,\bbC)$--valued function $f$, then $f$ is called the spectral density of the stationary distribution $F$. Furthermore, in view of~\eqref{cond:weak-stationary-distribution:int.nu(d.omega)/(1+omega^2)^p<infty}, there exists $p\in\rbb$ such that
\begin{equation} \label{cond:weak-stationary-distribution:int.f(omega)/(1+omega^2)^p<infty}
\int_{\rbb} \frac{\|f(\omega)\|}{(1+\omega^2)^p}\d\omega<\infty.
\end{equation}
\end{remark}

Whereas Theorem~\ref{thm:weak-stationary-distribution:characterization} describes the spectral representation of the covariance structure of the stationary distribution $F$, a representation formula for the linear functional $F$ itself is provided next. For this purpose, we need the definition of a random measure.
\begin{definition} \label{def:random-measure}  Let $\nu$ be a matrix-valued Borel measure satisfying~\eqref{cond:weak-stationary-process:Hermite-positive-definite} and~\eqref{cond:weak-stationary-distribution:int.nu(d.omega)/(1+omega^2)^p<infty}. Let $\mathcal{B}_\nu$ be the collection of all Borel sets $E\subset\rbb$ such that $\|\nu(E)\| <\infty$. A map $Z:\mathcal{B}_\nu\to L^2(\Omega)^d$ is called a \emph{random measure} with respect to $\nu$ if for $E_1, E_2\in \mathcal{B}_\nu$,
\[\E\big[Z(E_1) Z(E_2)^*\big]=\nu(E_1\cap E_2).\]
\end{definition}
So, let $Z(\d \omega)$ be a random measure with respect to $\nu$ as in Definition \ref{def:random-measure}. The natural space of integrands for $Z(\d \omega)$ is given by
$$
L^2(\nu)  = \Big\{g: \bbR \rightarrow {\mathcal M}(d,\bbC): \Big\|\int_\rbb g(\omega)\nu (\d \omega)g(\omega)^* \Big\|<\infty \Big\}.
$$
In fact, for every $g_1,g_2\in L^2(\nu)$, the stochastic integral $\int_\rbb g(\omega)Z(\d \omega)$ is a well defined random vector such that
\begin{equation} \label{form:E|int.g(omega)Z(d.omega)|^2}
\E\bigg[\int_\rbb g_1(\omega)Z(\d \omega) \Big(\int_\rbb g_2(\omega')Z(\d \omega')\Big)^*\bigg]=\int_\rbb g_1(\omega)\nu(\d \omega)g_2(\omega)^*
\end{equation}
(see~\cite{ito1954stationary,yaglom1957some} for a detailed discussion). In the following theorem, $F$ is characterized by means of random measures (see also~\cite[Theorem 3]{yaglom1957some}).
\begin{theorem} \label{thm:weak-stationary-distribution:random-measure}
Let $F$ be a stationary random distribution with the spectral measure $\nu$ as in Theorem~\ref{thm:weak-stationary-distribution:characterization}. Then, there exists a random measure $Z$ corresponding to $\nu$ as in Definition~\ref{def:random-measure} such that, for all $\f\in\Sc$,
\begin{equation} \label{form:weak-stationary-distribution:random-measure}
\la F,\f\ra= \int_\rbb \widehat{\varphi}(\omega)\cdot I_d \hspace{0.5mm}Z(\d\omega),
\end{equation}
where $I_d$ is the identity matrix. Moreover, $Z$ is uniquely determined by $F$ and $\nu$.
\end{theorem}

Note that, as of now, the stationary distribution $F$ is a functional whose domain is restricted to $\Sc$. In order to define the process $\ubf(t)$ via $F$, it is necessary to extend the definition of $F$ to a subclass of tempered distributions $\Sc'$. For this purpose, we employ the approach introduced in~\cite{didier2020asymptotic,mckinley2018anomalous}.
\begin{definition} \label{def:weak-operator} Let $\nu$ be a matrix--valued Borel measure satisfying conditions~\eqref{cond:weak-stationary-distribution:int.nu(d.omega)/(1+omega^2)^p<infty} and~\eqref{cond:weak-stationary-process:Hermite-positive-definite}. Let $Z$ be the vector--valued random measure associated with $\nu$ as in Definition~\ref{def:random-measure}. Further suppose that $\nu$ is absolutely continuous with respect to Lebesgue measure. Then, we define an operator $\Phi \, : \, \Sc' \to L^2(\Omega)^d$ by means of the mapping
\begin{equation} \label{eqn:weak-operator}
g \in \Sc' \mapsto \la \Phi,g\ra = \int_\rbb \F{g}(\omega)\cdot I_d\, Z(\d\omega).
\end{equation}
The domain of $\Phi$, denoted by $\emph{Dom}(\Phi)$, is the set of tempered distributions $g$ such that its Fourier transform $\F{g}$ in $\Sc'$ is a function defined on $\rbb$ and that $\F{g}\in L^2(\nu)$.
\end{definition}

In the following lemma we establish that the absolute continuity of $\nu$ with respect to Lebesgue measure is a sufficient condition for the extension of $\Phi$ as in Definition \ref{def:weak-operator} to be well defined. This extends analogous results for one--dimensional settings \cite{didier2020asymptotic,mckinley2018anomalous}.
\begin{lemma} \label{lem:Weak-op-defined} Let $\Phi:\emph{\text{Dom}}(\Phi)\subset \Sc'\to L^2(\Omega)^d$ be the operator as in Definition~\ref{def:weak-operator}. Then, $\Phi$ is well defined.
\end{lemma}
The proof of Lemma~\ref{lem:Weak-op-defined} is essentially the same as that of~\cite[Lemma 2.15]{mckinley2018anomalous}. Since the argument is short, we include it here for the sake of completeness.
\begin{proof}[Proof of Lemma~\ref{lem:Weak-op-defined}]
 By the absolutely continuity of $\nu$ with respect to Lebesgue measure, we may write $\nu(\d\omega)=f(\omega)\d\omega$. We proceed to show that the right-hand side of \eqref{eqn:weak-operator} does not depend on the choice of $\F{g}$. To see that, suppose $\mathcal{F}_1[g]$ and $\mathcal{F}_2[g]$ are Fourier transforms of $g$ in $\Sc'$. Then, $\mathcal{F}_1[g] = \mathcal{F}_2[g]$ a.e.\ \cite{strichartz2003guide}. In view of~\eqref{form:E|int.g(omega)Z(d.omega)|^2}, this implies that
\begin{equation}\label{eqn:Weak-op-defined-1}
\begin{aligned}
&\E\Big[\Big(\int_\rbb\mathcal{F}_1[g](\omega)\cdot I_d \hspace{0.5mm}Z(\d\omega)-\int_\rbb \mathcal{F}_2[g](\omega)\cdot I_d \hspace{0.5mm}Z(\d\omega)\Big)\\
&
\times \Big(\int_\rbb\mathcal{F}_1[g](\omega)\cdot I_d \hspace{0.5mm} Z(\d\omega)-\int_\rbb \mathcal{F}_2[g](\omega)\cdot I_d \hspace{0.5mm}Z(\d\omega)\Big)^*\Big]\\
&\qquad = \int_\rbb \Big|\mathcal{F}_1[g](\omega)-\mathcal{F}_2[g](\omega)\Big|^2 f(\omega)\d \omega =0.
\end{aligned}
\end{equation}
It follows that the random vectors $\int_\rbb \mathcal{F}_1[g](\omega)\cdot I_d \hspace{0.5mm}Z(\d\omega)$ and $\int_\rbb \mathcal{F}_2[g](\omega)\cdot I_d \hspace{0.5mm}Z(\d\omega)$ are equal a.s., implying that $\Phi$ is well defined. This finishes the proof.
\end{proof}

Having obtained the extension $\Phi$ of $F$ to $\Sc'$, we are now ready to define the process $\ubf(t)$ via the action of $\Phi$ on Dirac functions as in the following definition.

\begin{definition}[The function--valued version of a stationary random operator] \label{def:form:u(t)}
Let $\delta_t$ be the Dirac $\delta$ distribution centered at $t$. If $\delta_t\in\text{\emph{Dom}}(\Phi)$, then we define
\begin{equation} \label{form:u(t)}
\ubf(t):= \la \Phi,\delta_t\ra.
\end{equation}	
\end{definition}
\begin{remark} \label{rem:weak-oprator:point-process} We note that the condition $\delta_t\in\text{\emph{Dom}}(\Phi)$ in Definition~\ref{def:form:u(t)} is equivalent to the assumption $\nu_{ii}$ are all finite nonnegative measures. To see this, by Definition~\ref{def:weak-operator} together with~\eqref{form:E|int.g(omega)Z(d.omega)|^2} and \eqref{eqn:weak-operator}, it holds that
\begin{align*}
\E\Big[\int_\rbb \F{\delta_t}(\omega)\cdot I_d\, Z(\d \omega)\Big(\int_\rbb \F{\delta_t}(\omega)\cdot I_d\, Z(\d \omega)\Big)^*\Big]=\Big\|\int_\rbb \nu(\d\omega)\Big\|^2<\infty,
\end{align*}
where the last implication above is equivalent to $\sum_{i=1}^d\nu_{ii}(\rbb)<\infty,$ since $\nu\in{\mathcal H}_{\geq 0}(d,\bbC)$. In view of Theorem~\ref{thm:weak-stationary-process:characterization}, $\ubf(t)=\la \Phi,\delta_t\ra $ is thus simply the ordinary stochastic process version for $\Phi$ (cf.\ Lemma~\ref{lem:x(t)-v(t):continuous}).
\end{remark}

\section{Fourier analysis of the memory kernel $K(t)$} \label{sec:Kcos-Ksin}
In this section, we collect several useful properties of Fourier transforms for $K(t)$ under Assumption~\ref{cond:K:general}. More details can be found in~\cite{didier2020asymptotic,mckinley2018anomalous,soni1974parseval,soni1975slowly,
soni1975slowlyII}.

In the following lemma, we state the fact that the Fourier transform of $K$ under Assumption~\ref{cond:K:general} is well defined in the sense of improper integrals.
\begin{lemma}  \label{lem:Kcos-Ksin-welldefined}
Suppose that $K$ satisfies Assumption~\ref{cond:K:general} (I) (a) and (b). Then, for $\omega\neq 0$, the improper integrals $\Kcos(\omega)=\int_0^\infty K(t)\cos(t\omega) \d t$ and $\Ksin(\omega)=\int_0^\infty K(t)\sin(t\omega) \d t$ are well defined, continuous in $\omega$, and
 \begin{equation} \label{lim:Kcos-Ksin:omega->infinity}
 \lim_{\omega\to\infty}\Kcos(\omega)=\lim_{\omega\to\infty} \Ksin(\omega)=0.	
 \end{equation}
\end{lemma}
\begin{proof} The proof is essentially the same as in \cite[Lemma 2.18]{mckinley2018anomalous}. See also \cite[Lemma 1]{soni1975slowly}.
\end{proof}

Next, we describe the Fourier transform of $K$ in the sense of distributions.
\begin{lemma}\label{lem:temper-distribution:Kcos} Suppose that $K$ satisfies~Assumption~\ref{cond:K:general}. Then, the following holds.
\begin{itemize}
\item [(a)] The Fourier transform of $K$ in the sense of tempered distributions is given by $2\Kcos$. In other words, for every $\f\in\Sc$,
\begin{align}\label{form:temper-distribution:Kcos}
\int_\rbb K(t)\widehat{\f}(t)\d t = \int_\rbb 2\Kcos(\omega)\f(\omega)\d\omega.
\end{align}

\item [(b)] For any $\varphi\in\Sc$, the Fourier transform of $K^+*\varphi$ in $\Sc'$ is given by
\begin{equation}\label{form:temper-distribution:K^+*phi}
\F{K^+*\varphi}(\omega) = \widehat{K^+}\cdot \widehat{\varphi}= \left(\Kcos(\omega)-\imag \Ksin(\omega)\right)\widehat{\varphi}(\omega),
\end{equation}
where  $K^+(t)=K(t)1_{[0,\infty)}(t)$.
\end{itemize}
\end{lemma}
\begin{proof} (a) If $K$ is integrable then~\eqref{form:temper-distribution:Kcos} is a consequence of Fubini's theorem. When $K$ satisfies the tail behavior $t^{-1}$ (see Assumption~\ref{cond:K:general} (II) (b)), then the argument can be found in the proof of~\cite[Proposition 17]{didier2020asymptotic}. Finally, if $K$ satisfies Assumption~\ref{cond:K:general} (II) (c), the argument is the same as in the proof of \cite[Proposition 2.19 (a)]{mckinley2018anomalous}.

(b) The proof of~\eqref{form:temper-distribution:K^+*phi} is essentially the same as the proof of \cite[Proposition 2.19 (b)]{mckinley2018anomalous}.
\end{proof}

In the following result, we provide the asymptotic behavior of the functions $\Kcos$ and $\Ksin$ near the origin. These properties play an import role in establishing the asymptotic growth of $\int_0^t (x(s),v(s))\d s$ as $t \rightarrow \infty$ in Theorem~\ref{thm:asymptotic-growth}. See also~\cite{inoue1995abel,soni1975slowly,soni1975slowlyII} for related results.

\begin{lemma}[Abelian direction]\label{lem:Kcos-Ksin:omega->0} Suppose that $K\in L^1_{\text{loc}}(0,\infty)$ satisfies Assumption~\ref{cond:K:general}. Then, the following holds.
\begin{itemize}
\item [(a)] If $K$ is integrable, then
\begin{equation} \label{lim:Kcos-Ksin:omega->0:diffusion}
\Kcos(\omega)\to\int_0^\infty\close K(t)\d t\quad\text{and}\quad\Ksin(\omega)\to0\quad\text{as } \omega \to 0.
\end{equation}

\item [(b)] If $K(t)\sim t^{-1}$ as $t\to\infty$, then
\begin{equation} \label{lim:Kcos-Ksin:omega->0:critial}
\frac{\Kcos(\omega) }{|\log(\omega)|}\to c_1, \quad\text{and}\quad\Ksin(\omega) \to c_1 \hspace{1mm}\frac{\pi}{2},\quad\text{as }\omega \to 0,
\end{equation}
where $c_1=\lim_{t\to\infty}t\,K(t)\in(0,\infty)$.

\item [(c)] If there exists $\alpha\in(0,1)$ such that $K(t)\sim t^{-\alpha}$ as $t\to\infty$, then
\begin{equation} \label{lim:Kcos-Ksin:omega->0:subdiffusion}
\omega^{1-\alpha}\Kcos(\omega)\to c_\alpha\int_0^\infty \frac{\cos(u)}{u^\alpha}\d u \quad  \text{and}  \quad\omega^{1-\alpha}\Ksin(\omega)\to c_\alpha\int_0^\infty \frac{\sin(u)}{u^\alpha}\d u\quad\text{as }\omega\to 0,
\end{equation}
where $c_\alpha=\lim_{t\to\infty}t^{\alpha}\,K(t)\in(0,\infty)$.
\end{itemize}
\end{lemma}
\begin{proof}
The limit~\eqref{lim:Kcos-Ksin:omega->0:diffusion} is a consequence of the Dominated Convergence Theorem. The limits~\eqref{lim:Kcos-Ksin:omega->0:critial} and~\eqref{lim:Kcos-Ksin:omega->0:subdiffusion} can be found in~\cite[Proposition 9]{didier2020asymptotic} and \cite[Proposition 3.1]{mckinley2018anomalous}, respectively.
\end{proof}

\section{Completely Monotonic Functions} \label{sec:CM}

In this section, we discuss two important properties of completely monotonic functions that are needed in the calculation of the second moment of $(x(t),v(t))$ (see Theorems~\ref{thm:equipartition-of-energy:1D:gamma=0:power-law.exponential} and~\ref{thm:equipartition-of-energy:2D:gamma>0:powerlaw.exponential}). First, we recall the following well--known theorem on the representation of the class $\CM$ in terms of Laplace transforms of Radon measures.
\begin{theorem}[Hausdorff--Bernstein--Widder Theorem]\label{thm:completely-monotone:Bernstein} A function $K$ is completely monotone as in Definition~\ref{def:complete-monotone} if and only if $K$ admits the formula
\begin{equation} \label{eq:complete-monotone:laplace}
K(t)=\int_0^\infty \!\!\! e^{-tx} \mu (\d x  ),
\end{equation}
for some positive Borel measure $\mu$ on $[0,\infty)$.
\end{theorem}

In Lemma~\ref{lem:complete-monotone:Kcos-Ksin}, stated and proven next, we compute Fourier transforms of completely monotonic functions based on their representation measures.
\begin{lemma} \label{lem:complete-monotone:Kcos-Ksin}
Suppose that $K\in\CM$ and that $K$ is locally integrable and is decreasing to $0$ as $t\to\infty$. Let $\mu$ be the representation measure as in~\eqref{eq:complete-monotone:laplace}. Then for every $\omega\neq 0$, we can write
\begin{equation} \label{form:complete-monotone:Kcos-Ksin}
\Kcos(\omega)\pm \imag\Ksin(\omega) = \int_0^\infty\close \frac{x\pm \imag \omega}{x^2+\omega^2}\mu( \d x  ).
\end{equation}
\end{lemma}
The proof of Lemma \ref{lem:complete-monotone:Kcos-Ksin} is essentially the same as that of \cite[Lemma 3.8]{mckinley2020h}. The only difference is that in \cite[Lemma 3.8]{mckinley2020h}, $K$ belongs to $\CM_b$, the class of completely monotone functions such that $K(0)$ is finite, whereas in Lemma~\ref{lem:complete-monotone:Kcos-Ksin}, we assume a slightly more general condition, namely, $K$ being locally integrable around the origin.
\begin{proof}[Proof of Lemma~\ref{lem:complete-monotone:Kcos-Ksin}] First note that, for all $\omega\neq 0$, the integrals in~\eqref{form:complete-monotone:Kcos-Ksin} are finite. Indeed, since $K$ is locally integrable, Fubini's theorem implies that
\begin{align*}
\int_0^1 K(t)\d t&=\int_0^1\int_0^\infty\close e^{-xt}\mu(\d x)\d t=\int_0^\infty \frac{1-e^{-x}}{x}\mu(\d x)<\infty.
\end{align*}
In particular,
\begin{align} \label{ineq:mu(0,1)<infty}
\mu([0,1])\le e\int_0^1\frac{1-e^{-x}}{x}&\mu(\d x)\le e\int_0^\infty\frac{1-e^{-x}}{x}\mu(\d x)<\infty,
\end{align}
and
\begin{equation}\label{ineq:int_1^infty 1/x.mu(dx)<infty}
\int_1^\infty\frac{1}{x}\mu(\d x)<\frac{1}{1-e^{-1}}\int_1^\infty \frac{1-e^{-x}}{x}\mu(\d x)<\infty.
\end{equation}
It follows that
\begin{align*}
\int_0^\infty\close \frac{x}{x^2+\omega^2}\mu( \d x  )=\Big\{\int_0^1+\int_1^\infty\Big\}\frac{x}{x^2+\omega^2}\mu( \d x  ) \le \frac{1}{\omega^2}\mu([0,1])+\int_1^\infty \frac{1}{x}\mu(\d x)<\infty.
\end{align*}
Likewise,
\begin{align*}
\int_0^\infty\close \frac{\omega}{x^2+\omega^2}\mu( \d x  )\le\frac{1}{\omega}\mu([0,1])+\omega\int_1^\infty \frac{1}{x^2}\mu(\d x)<\infty.
\end{align*}
Now, by the definition of improper integral,
\begin{align*}
\Kcos(\omega)-\imag \Ksin(\omega):= \lim_{A\to\infty}\int_0^A\close  K(t)e^{-\imag t\omega}\d t.
\end{align*}
Based on the representation $K(t)=\int_0^\infty e^{-tx}\mu( \d x    )$ (see \eqref{eq:complete-monotone:laplace}) and on Fubini's theorem, we obtain
\begin{align*}
\int_0^A\close K(t)e^{-\imag t\omega}\d t &= \int_0^A \close\int_0^\infty\close e^{-tx}\mu( \d x    ) e^{-\imag t\omega}\d t \\
&= \int_0^\infty \close\int_0^A\close e^{-(x+\imag \omega)t}\d t \mu( \d x    )\\
&= \int_0^\infty\frac{1-e^{-(x+\imag \omega)A}}{x+\imag \omega}\mu( \d x    )\\
&= \int_0^\infty \frac{\left(1-e^{-(x+\imag \omega)A}\right)x}{x^2+\omega^2}\mu( \d x    )-\imag \int_0^\infty \frac{\left(1-e^{-(x+\imag \omega)A}\right)\omega}{x^2+\omega^2}\mu( \d x    ).
\end{align*}
Also, since $K(t)$ decreases to 0 as $t\to\infty$, the Dominated Convergence Theorem implies that
\begin{equation} \label{eqn:mu(0)=0}
\mu(\{0\})=\lim_{t\to\infty}\int_0^\infty\close e^{-tx}\mu(\d x )=\lim_{t\to\infty}K(t)=0.
\end{equation}
It follows that $\mu-$a.e.\ on $x\in[0,\infty)$,
\begin{align*}
\lim_{A\to\infty}\frac{\left(1-e^{-(x+\imag \omega)A}\right)x}{x^2+\omega^2}=\frac{x}{x^2+\omega^2},\quad\text{and}\quad\lim_{A\to\infty}\frac{\left(1-e^{-(x+\imag \omega)A}\right)\omega}{x^2+\omega^2}=\frac{\omega}{x^2+\omega^2}.
\end{align*}
Again by the Dominated Convergence Theorem, we obtain
\begin{align*}
\lim_{A\to\infty}\Big[\int_0^\infty \frac{\left(1-e^{-(x+\imag \omega)A}\right)x}{x^2+\omega^2}\mu( \d x    )-&\imag \int_0^\infty \frac{\left(1-e^{-(x+\imag \omega)A}\right)\omega}{x^2+\omega^2}\mu( \d x    )\Big]\\
=\int_0^\infty\!\! \frac{x}{x^2+\omega^2}\mu( \d x    )-&\imag \int_0^\infty \close\frac{\omega}{x^2+\omega^2}\mu( \d x    )= \int_0^\infty \close\frac{x-\imag \omega}{x^2+\omega^2}\mu( \d x    ).
\end{align*}
This establishes~\eqref{form:complete-monotone:Kcos-Ksin} for $\Kcos(\omega)-\imag\Ksin(\omega)$. The formula for $\Kcos(\omega)+\imag\Ksin(\omega)$ can be derived using a similar argument.
\end{proof}

Next, we consider the case $K=\f(t^2)$, where $\f\in\CM$. Unlike in the situation where $K\in\CM$, computing the Fourier transform of $K=\f(t^2)$ is more complicated since it relies on delicate estimates for the error functions $\erf$ and $\erfc$ as well as for the Faddeeva function $w$ introduced in~\eqref{e:erfc}-\eqref{form:w(z)}.

\begin{lemma} \label{lem:complete-monotone:Kcos-Ksin:K(t)=phi(t^2)}
Suppose that $K(t)=\f(t^2)$, where $\f\in\CM$. Also suppose that $K$ is locally integrable and decreases to $0$ as $t\to\infty$. Let $\mu$ be the representation measure for $\f\in\CM$ as in~\eqref{eq:complete-monotone:laplace}. Then, for every $\omega\neq 0$, we can write
\begin{equation} \label{form:complete-monotone:Kcos-Ksin:K(t)=phi(t^2)}
\begin{aligned}
\Kcos(\omega) \pm \imag \Ksin(\omega) & =\frac{\sqrt{\pi}}{2}\int_0^\infty \close\frac{ 1}{\sqrt{x}} e^{-\frac{\omega^2}{4x}} \erfc\Big(\mp \imag\frac{ \omega}{2\sqrt{x}}\Big)\mu(\d x)\\
& = \frac{\sqrt{\pi}}{2}\int_0^\infty\close\frac{1}{\sqrt{x}} w\Big(\pm\frac{\omega}{2\sqrt{x}}\Big)\mu(\d x).
\end{aligned}
\end{equation}
\end{lemma}
\begin{proof} We will prove formula~\eqref{form:complete-monotone:Kcos-Ksin:K(t)=phi(t^2)} for $\Kcos(\omega)-\imag\Ksin(\omega)$. The formula for $\Kcos(\omega)+\imag\Ksin(\omega)$ can be derived using a similar argument.

In view of the expression~\eqref{eq:complete-monotone:laplace} for $\f\in\CM$, $K(t)$ admits the representation
\begin{align}\label{e:K(t)=int_exp(-t^2x)mu(dx)}
K(t)=\f(t^2)=\int_0^\infty\close e^{-t^2 x}\mu(\d x) \geq 0.
\end{align}
For $\mu(dx)$ as in \eqref{e:K(t)=int_exp(-t^2x)mu(dx)}, we claim that
\begin{equation}\label{ineq:int_1^infty 1/sqrt(x).mu(dx)<infty}
\int_1^\infty\close \frac{1}{\sqrt{x}}\mu(\d x)<\infty.
\end{equation}
To see this, first note that, due to the local integrability of $K$,
\begin{align}\label{e:infty>int^1_0_K(t)dt}
\infty>\int_0^1 K(t)\d t=\int_0^1\int_0^\infty\close e^{-t^2 x}\mu(\d x)\d t&=\int_0^\infty\close\frac{1}{\sqrt{x}} \int_0^{\sqrt{x}}\close e^{-t^2}\d t\,\mu(\d x),
\end{align}
where the second equality in \eqref{e:infty>int^1_0_K(t)dt} follows from a change of variable. Also,
\begin{align*}
\int_0^\infty\close\frac{1}{\sqrt{x}} \Big\{ \int_0^{\sqrt{x}}\close e^{-t^2}\d t\Big\}\,\mu(\d x)
&\geq \int_1^\infty\close\frac{1}{\sqrt{x}} \Big\{\int_0^{\sqrt{x}}\close e^{-t^2}\d t\Big\}\,\mu(\d x)\ge \int_1^\infty\close \frac{1}{\sqrt{x}}\mu(\d x)\cdot \int_0^1\close e^{-t^2}\d t,
\end{align*}
which proves~\eqref{ineq:int_1^infty 1/sqrt(x).mu(dx)<infty}.

Now fix $A>0$ and let $\omega \neq 0$. Fubini's Theorem and a change of variable imply that
\begin{align}\label{e:int^A_0_K(t)_exp(-iwt)dt}
\int_0^A\close K(t)e^{-\imag \omega t}\d t&=\int_0^\infty\close \int_0^A  \close e^{-t^2 x-\imag \omega t}\d t\,\mu(\d x)=\int_0^\infty \frac{e^{-\frac{\omega^2}{4x}}}{\sqrt{x}}\int_{\imag\frac{\omega}{2\sqrt{x}}}^{A\sqrt{x}+\imag
\frac{\omega}{2\sqrt{x}}}\close e^{-z^2}\d z\,\mu(\d x).
\end{align}
When considering the limit $A \to \infty$, we want to apply the Dominated Convergence Theorem in expression \eqref{e:int^A_0_K(t)_exp(-iwt)dt} so as to establish formula~\eqref{form:complete-monotone:Kcos-Ksin:K(t)=phi(t^2)}. To this end, it suffices to find a dominating $\mu-$integrable function for the family of integrands $\frac{e^{-\frac{\omega^2}{4x}}}{\sqrt{x}}\int_{\imag\frac{\omega}{2\sqrt{x}}}^{A\sqrt{x}+\imag
\frac{\omega}{2\sqrt{x}}}e^{-z^2}\d z$, $A > 0$.

We consider the contour integral on the rectangle curve
$$
D_1:0\dashrightarrow\imag\frac{\omega}{2\sqrt{x}}\dashrightarrow A\sqrt{x}+\imag
\frac{\omega}{2\sqrt{x}}\dashrightarrow A\sqrt{x}\dashrightarrow 0.
$$
By the analyticity of $e^{-z^2}$, $\int_{D_1} e^{-z^2}\d z=0$, whence
\begin{align}\label{e:int_D1_e^(-z^2)dz=0}
\int_{\imag\frac{\omega}{2\sqrt{x}}}^{A\sqrt{x}+\imag
\frac{\omega}{2\sqrt{x}}}\close e^{-z^2}\d z=\int_0^{A\sqrt{x}}\close e^{-t^2}\d t- \int_0^{\imag\frac{\omega}{2\sqrt{x}}}\close e^{-z^2}\d z+\int_{A\sqrt{x}}^{A\sqrt{x}+\imag\frac{\omega}{2\sqrt{x}}}\close e^{-z^2}\d z.
\end{align}
By making the changes of variable $z=\imag t$ and $z=A\sqrt{x}+\imag t$ in the second and last terms on the right-hand side of \eqref{e:int_D1_e^(-z^2)dz=0}, we obtain
\begin{align*}
\int_{\imag\frac{\omega}{2\sqrt{x}}}^{A\sqrt{x}+\imag
\frac{\omega}{2\sqrt{x}}}\close e^{-z^2}\d z=\int_0^{A\sqrt{x}}\close e^{-t^2}\d t-\imag \int_0^{\frac{\omega}{2\sqrt{x}}}\close e^{t^2}\d t+\imag \int_0^{\frac{\omega}{2\sqrt{x}}}\close e^{-(A\sqrt{x}+\imag t)^2}\d t.
\end{align*}
It follows that
\begin{align*}
\Big|\int_{\imag\frac{\omega}{2\sqrt{x}}}^{A\sqrt{x}+\imag
\frac{\omega}{2\sqrt{x}}}\close e^{-z^2}\d z\Big|&\le \frac{\sqrt{\pi}}{2}+2\int_0^{\frac{|\omega|}{2\sqrt{x}}}\close e^{t^2}\d t.
\end{align*}
It remains to show that
\begin{align}\label{e:int_1/sqrt(x)_exp_mu(dx)}
\int_0^\infty\close \frac{1}{\sqrt{x}}e^{-\frac{\omega^2}{4x}}\mu(\d x)<\infty
\end{align}
and
\begin{align}\label{e:int_(1/sqrt(x))_daw_mu(dx)}
0 \leq \int_0^\infty \close\frac{1}{\sqrt{x}}\daw\Big(\frac{|\omega|}{2\sqrt{x}}\Big)\mu(\d x)<\infty.
\end{align}
To show \eqref{e:int_1/sqrt(x)_exp_mu(dx)}, we employ~\eqref{ineq:int_1^infty 1/sqrt(x).mu(dx)<infty} and the elementary bound $e^{-t}<1/t$ for all $t>0$ to construct the estimate
\begin{align*}
\int_0^\infty\close \frac{1}{\sqrt{x}}e^{-\frac{\omega^2}{4x}}\mu(\d x)\le \frac{4}{\omega^2}\int_0^1\close\sqrt{x}\mu(\d x)+\int_1^\infty\close\frac{1}{\sqrt{x}}\mu(\d x)<\infty.
\end{align*}
In turn, to show \eqref{e:int_(1/sqrt(x))_daw_mu(dx)}, note that, for all real $t$, $\daw(t)$ as in~\eqref{form:dawson(z)} satisfies \cite[Section 7.8]{olver2010nist}
\begin{equation} \label{ineq:daw(t)<C/t}
\daw(|t|)\le \frac{C}{|t|}.
\end{equation}
Therefore,
\begin{align*}
\int_0^\infty\close \frac{1}{\sqrt{x}}\daw\Big(\frac{|\omega|}{2\sqrt{x}}\Big)\mu(\d x)\le \frac{c}{|\omega|}\mu([0,1])+c\int_1^\infty\close\frac{1}{\sqrt{x}}\mu(\d x)<\infty.
\end{align*}
This concludes the proof.
\end{proof}

We finish this section by the establishing the following useful estimate on $w(z)$. The result is employed in Section~\ref{sec:proofs} in establishing the equipartition of energy condition.
\begin{lemma} \label{lem:w(z)} Let $w(z)$ be the Faddeeva function as in~\eqref{form:w(z)}. For all $z=r e^{\imag\theta}$, $-\pi/8<\theta<9\pi/8$ and sufficiently large $r$,
\begin{equation} \label{ineq:w(z)}
|w(z)|\le \frac{C}{|z|}.
\end{equation}
\end{lemma}
\begin{remark} The interval $(-\pi/8,9\pi/8)$ in Lemma~\ref{lem:w(z)} can actually be any $(\theta_1,\theta_2)$ such that $-\pi/4<\theta_1<\theta_2<5\pi/4$.
\end{remark}
\begin{proof}[Proof of Lemma~\ref{lem:w(z)}] There are two situations to be considered, depending on the location of $z$ in $\bbC$.

We first consider the case where $\Im(z)\ge 0$. By writing $z=u+\imag v$, $v\ge 0$, in view of~\eqref{e:erfc},~\eqref{e:erf} and~\eqref{form:w(z)}, we can reexpress $w(z)$ as
\begin{align*}
w(z)=e^{-z^2}\Big(1-\frac{2}{\sqrt{\pi}}\int_0^{v-\imag u}\close e^{-z^2}\d z\Big).
\end{align*}
We consider the contour integral on the triangle curve
$$
D_2:0\dashrightarrow v\dashrightarrow v-\imag u\dashrightarrow 0.
$$
Since $\int_{D_2} e^{-z^2}\d z=0$, then
\begin{align*}
\int_0^{v-\imag u}\close e^{-z^2}\d z&=\int_0^v e^{-z^2}\d z+\int_v^{v-\imag u}\close e^{-z^2}\d z\\
&=\int_0^v e^{-t^2}\d t-\imag\int_0^u e^{-(v-\imag t)^2}\d t.
\end{align*}
It follows that
\begin{align*}
w(z)&=e^{-u^2+v^2-\imag 2uv}\Big(1-\frac{2}{\sqrt{\pi}}\int_0^v e^{-t^2}\d t+\imag\frac{2}{\sqrt{\pi}}\int_0^ue^{-v^2+t^2+\imag 2vt}\d t\Big)\\
&=\frac{2}{\sqrt{\pi}} e^{-u^2+v^2-\imag 2uv}\Big(\int_v^\infty\close e^{-t^2}\d t+\imag\int_0^ue^{-v^2+t^2+\imag 2vt}\d t\Big).
\end{align*}
Therefore,
$$
\frac{\sqrt{\pi}}{2}|w(z)| \le e^{-u^2+v^2}\int_v^\infty\close e^{-t^2}\d t+ e^{-u^2}\Big|\int_0^{u} e^{t^2+\imag 2vt}\d t\Big|=: I_1(u,v)+I_2(u,v).
$$
So, the bound~\eqref{ineq:w(z)} holds provided we can show that
\begin{equation}\label{e:sup_(|u|+v)(I1+I2)<infty}
\sup_{v\ge 0}(|u|+v)(I_1(u,v)+I_2(u,v))<\infty.
\end{equation}
We first consider $I_1(u,v)$. Note that there exists a positive $c>0$ such that, for all $|u|,\, v\ge 0$,
\begin{align*}
e^{-u^2}\le \frac{c}{|u|+1}\quad\text{and}\quad e^{v^2}\int_v^\infty\close e^{-t^2}\d t\le \frac{c}{v+1}.
\end{align*}
Hence, for all $u\in\rbb$ and $v\ge 0$,
\begin{align*}
I_1(u,v)=e^{-u^2+v^2}\int_v^\infty\close e^{-t^2}\d t\le \frac{c}{(|u|+1)(v+1)}\le \frac{c}{|u|+v+1},
\end{align*}
implying
\begin{align}\label{e:sup_(|u|+v)I1<infty}
\sup_{v\ge 0}(|u|+v)I_1(u,v)<\infty.
\end{align}
In regard to $I_2(u,v)$, we invoke~\eqref{ineq:daw(t)<C/t} to estimate
\begin{align}\label{e:|u|I2<infty}
|u|I_2(u,v)=|u|e^{-u^2}\Big|\int_0^{u} e^{t^2+\imag 2vt}\d t\Big|\le |u|e^{-u^2}\int_0^{|u|} e^{t^2}\d t\le c.
\end{align}
To bound $vI_2(u,v)$, it suffices to consider $v\ge 1$. Note that
\begin{align*}
I_2(u,v)&\le e^{-u^2}\Big|\int_0^{u}e^{t^2}\cos(2v t)\d t\Big|+e^{-u^2}\Big|\int_0^{u} e^{t^2}\sin(2v t)\d t\Big|\\
&=e^{-u^2}\Big|\int_0^{|u|}e^{t^2}\cos(2v t)\d t\Big|+e^{-u^2}\Big|\int_0^{|u|} e^{t^2}\sin(2v t)\d t\Big|.
\end{align*}
By Second Mean Value Theorem, for each $v\ge 1$, there exists $0<u_*<|u|$ such that
\begin{align*}
\Big|\int_0^{|u|}e^{t^2}\cos(2v t)\d t\Big|&=\Big|\int_0^{u_*}\close\cos(2v t)\d t+e^{u^2}\int_{u_*}^{|u|}\close\cos(2v t)\d t\Big|\le c\frac{1+e^{u^2}}{v}.
\end{align*}
Likewise,
\begin{align*}
\Big|\int_0^{|u|}e^{t^2}\sin(2v t)\d t\Big|\le c\frac{1+e^{u^2}}{v}.
\end{align*}
Therefore, still for $v\ge 1$,
\begin{align}\label{e:vI2<infty}
v I_2(u,v)\le c.
\end{align}
The bounds \eqref{e:sup_(|u|+v)I1<infty}, \eqref{e:|u|I2<infty} and \eqref{e:vI2<infty} imply \eqref{e:sup_(|u|+v)(I1+I2)<infty}. This establishes~\eqref{ineq:w(z)} for the case $v = \Im(z)\ge 0$.

Alternatively, consider the case $z=re^{\imag \theta}$, $\theta\in [-\pi/8,0]\cup[\pi,9\pi/8]$. In particular, $\Im(z) < 0$. By writing $z=u-\imag v$, $v>0$, we note that
\begin{align*}
\frac{v}{|u|}=|\tan(\theta)|\le \tan\Big(\frac{\pi}{8}\Big)<1.
\end{align*}
In other words, there exists $\varepsilon\in(0,1)$ such that $v\le \varepsilon|u|$. Note that $w(z)$ satisfies the property \cite[expression (3)]{fettis:caslin:cramer:1973}
\begin{align*}
w(-z)=e^{-2z^2}-w(z),\quad\forall z\in\cbb.
\end{align*}
Then, for $v\ge 0$,
\begin{align} \label{ineq:w(u-iv)<e^(-2u^2+2v^2)+|w(-u+iv)|}
w(u-\imag v)=w(-(-u+\imag v))=e^{-2(-u+\imag v)^2}-w(-u+\imag v)\le e^{-2(u^2-v^2)}+|w(-u+\imag v)|.
\end{align}
We invoke~\eqref{ineq:w(z)} for the first case $\Im(z)\ge0$ to see that for all $|u|,\,v\ge0$,
\begin{align} \label{ineq:(|u|+v)|w(-u+iv)|<c}
(|u|+v)|w(-u+\imag v)|\le c.
\end{align}
Also, since $0\le v\le \varepsilon |u|$, $\varepsilon\in(0,1)$, we infer the existence of a (possibly different) positive constant $c>0$ such that
\begin{align}\label{ineq:(|u|+v)e^(-2(u^2-v^2))<c}
(|u|+v)e^{-2(u^2-v^2)}\le(1+\varepsilon)|u|e^{-2(1-\varepsilon)u^2}<c .
\end{align}
We finally combine the estimates~\eqref{ineq:(|u|+v)e^(-2(u^2-v^2))<c} and~\eqref{ineq:(|u|+v)|w(-u+iv)|<c} with~\eqref{ineq:w(u-iv)<e^(-2u^2+2v^2)+|w(-u+iv)|} to establish the desired estimate~\eqref{ineq:w(z)} for the second case $z=re^{\imag \theta}$, $\theta\in [-\pi/8,0]\cup[\pi,9\pi/8]$. This concludes the proof.
\end{proof}

\bibliography{GLE-2D-bib}

\begin{thebibliography}{10}

\bibitem{argun2016non}
A.~Argun, A.-R. Moradi, E.~Pin{\c{c}}e, G.~B. Bagci, A.~Imparato, and G.~Volpe.
\newblock {Non-Boltzmann stationary distributions and nonequilibrium relations
  in active baths}.
\newblock {\em Phys. Rev. E}, 94(6):062150, 2016.

\bibitem{brockwell:davis:1991}
P.~J. Brockwell and R.~A. Davis.
\newblock {\em Time Series: Theory and Methods}.
\newblock Springer Science \& Business Media, 1991.

\bibitem{chandler1987introduction}
D.~Chandler.
\newblock {\em Introduction to Modern Statistical Mechanics}.
\newblock Oxford University Press, Oxford, U.K., 1987.

\bibitem{cramer1967stationary}
H.~Cram{\'e}r and R.~Leadbetter.
\newblock {\em {Stationary and related stochastic processes: Sample function
  properties and their applications}}.
\newblock Courier Corporation, 1967.

\bibitem{desposito2008memory}
M.~A. Desposito and A.~D. Vi{\~n}ales.
\newblock {Memory effects in the asymptotic diffusive behavior of a classical
  oscillator described by a generalized Langevin equation}.
\newblock {\em Phys. Rev. E}, 77(3):031123, 2008.

\bibitem{didier2012statistical}
G.~Didier, S.~A. McKinley, D.~B. Hill, and J.~Fricks.
\newblock Statistical challenges in microrheology.
\newblock {\em J. Time Series Anal.}, 33(5):724--743, 2012.

\bibitem{didier2020asymptotic}
G.~Didier and H.~Nguyen.
\newblock Asymptotic analysis of the mean squared displacement under fractional
  memory kernels.
\newblock {\em SIAM J. Math. Anal.}, 52(4):3818--3842, 2020.

\bibitem{didier2017asymptotic}
G.~Didier and K.~Zhang.
\newblock The asymptotic distribution of the pathwise mean squared displacement
  in single particle tracking experiments.
\newblock {\em J. Time Series Anal.}, 38(3):395--416, 2017.

\bibitem{fettis:caslin:cramer:1973}
H.~E. Fettis, J.~C. Caslin, and K.~R. Cramer.
\newblock Complex zeros of the error function and of the complementary error
  function.
\newblock {\em Math. Comp.}, pages 401--407, 1973.

\bibitem{fricks2009time}
J.~Fricks, L.~Yao, T.~C. Elston, and M.~G. Forest.
\newblock Time-domain methods for diffusive transport in soft matter.
\newblock {\em SIAM J. Appl. Math.}, 69(5):1277--1308, 2009.

\bibitem{gelfand1955generalized}
I.~M. Gelfand.
\newblock Generalized random processes.
\newblock {\em Dokl. Akad. Nauk SSSR}, 100(5):853–856, 1955.

\bibitem{glatt2020generalized}
N.~E. Glatt-Holtz, D.~P. Herzog, S.~A. McKinley, and H.~D. Nguyen.
\newblock {The generalized Langevin equation with power-law memory in a
  nonlinear potential well}.
\newblock {\em Nonlinearity}, 33(6):2820, 2020.

\bibitem{goychuk2009viscoelastic}
I.~Goychuk.
\newblock Viscoelastic subdiffusion: from anomalous to normal.
\newblock {\em Phys. Rev. E}, 80(4):046125, 2009.

\bibitem{herzog:mattingly:nguyen:2021}
D.~P. Herzog, J.~C. Mattingly, and H.~D. Nguyen.
\newblock {Gibbsian dynamics and the generalized {L}angevin equation}.
\newblock {\em arXiv preprint arXiv:2111.04187}, 2021.

\bibitem{hill1986introduction}
T.~L. Hill.
\newblock {\em An Introduction to Statistical Thermodynamics}.
\newblock Courier Corporation, 1986.

\bibitem{hohenegger2017equipartition}
C.~Hohenegger.
\newblock {On equipartition of energy and integrals of generalized Langevin
  equations with generalized Rouse kernel}.
\newblock {\em Commun. Math. Sci.}, 15(2):539--554, 2017.

\bibitem{hohenegger2017fluid}
C.~Hohenegger and S.~McKinley.
\newblock Fluid--particle dynamics for passive tracers advected by a thermally
  fluctuating viscoelastic medium.
\newblock {\em J. Comput. Phys.}, 340:688--711, 2017.

\bibitem{hohenegger2018reconstructing}
C.~Hohenegger and S.~McKinley.
\newblock Reconstructing complex fluid properties from the behavior of
  fluctuating immersed particles.
\newblock {\em SIAM J. Appl. Math.}, 78(4):2200--2226, 2018.

\bibitem{inoue1995abel}
A.~Inoue.
\newblock {On Abel-Tauber theorems for Fourier cosine transforms}.
\newblock {\em J. Math. Anal. Appl.}, 196(2):764--776, 1995.

\bibitem{ito1954stationary}
K.~It{\^o}.
\newblock Stationary random distributions.
\newblock {\em Mem. College Sci. Univ. Kyoto Ser. A Math.}, 28(3):209--223,
  1954.

\bibitem{kneller2011generalized}
G.~Kneller.
\newblock {Generalized Kubo relations and conditions for anomalous diffusion:
  physical insights from a mathematical theorem}.
\newblock {\em J. Chem. Phys.}, 134(22):224106, 2011.

\bibitem{kou2008stochastic}
S.~Kou.
\newblock Stochastic modeling in nanoscale biophysics: subdiffusion within
  proteins.
\newblock {\em Ann. Appl. Stat.}, pages 501--535, 2008.

\bibitem{kou2004generalized}
S.~Kou and X.~Xie.
\newblock {Generalized Langevin equation with fractional Gaussian noise:
  subdiffusion within a single protein molecule}.
\newblock {\em Phys. Rev. Lett.}, 93(18):180603, 2004.

\bibitem{kubo1966fluctuation}
R.~Kubo.
\newblock The fluctuation-dissipation theorem.
\newblock {\em Rep. Prog. Phys.}, 29(1):255, 1966.

\bibitem{kupferman2004fractional}
R.~Kupferman.
\newblock {Fractional kinetics in Kac--Zwanzig heat bath models}.
\newblock {\em J. Stat. Phys.}, 114(1):291--326, 2004.

\bibitem{lindgren2012stationary}
G.~Lindgren.
\newblock {\em Stationary Stochastic Processes: Theory and Applications}.
\newblock CRC Press, 2012.

\bibitem{maggi2014generalized}
C.~Maggi, M.~Paoluzzi, N.~Pellicciotta, A.~Lepore, L.~Angelani, and
  R.~Di~Leonardo.
\newblock Generalized energy equipartition in harmonic oscillators driven by
  active baths.
\newblock {\em Phys. Rev. Lett.}, 113(23):238303, 2014.

\bibitem{mason1995optical}
T.~Mason and D.~Weitz.
\newblock Optical measurements of frequency-dependent linear viscoelastic
  moduli of complex fluids.
\newblock {\em Phys. Rev. Lett.}, 74(7):1250, 1995.

\bibitem{mckinley2018anomalous}
S.~McKinley and H.~Nguyen.
\newblock {Anomalous diffusion and the generalized Langevin equation}.
\newblock {\em SIAM J. Math. Anal.}, 50(5):5119--5160, 2018.

\bibitem{mckinley2020h}
S.~A. McKinley and H.~D. Nguyen.
\newblock {On the H\"older regularity of a linear stochastic
  partial-integro-differential equation with memory}.
\newblock {\em J. Fourier Anal. Appl.}, 28(2):1--31, 2022.

\bibitem{morgado2002relation}
R.~Morgado, F.~Oliveira, G.~Batrouni, and A.~Hansen.
\newblock Relation between anomalous and normal diffusion in systems with
  memory.
\newblock {\em Phys. Rev. Lett.}, 89(10):100601, 2002.

\bibitem{mori1965continued}
H.~Mori.
\newblock A continued-fraction representation of the time-correlation
  functions.
\newblock {\em Prog. Theor. Phys.}, 34(3):399--416, 1965.

\bibitem{mori1965transport}
H.~Mori.
\newblock {Transport, Collective Motion, and Brownian Motion}.
\newblock {\em Prog. Theor. Phys.}, 33(3):423--455, 1965.

\bibitem{nichol2012equipartition}
K.~Nichol and K.~E. Daniels.
\newblock Equipartition of rotational and translational energy in a dense
  granular gas.
\newblock {\em Phys. Rev. Lett.}, 108(1):018001, 2012.

\bibitem{olver2010nist}
F.~Olver, D.~W. Lozier, R.~F. Boisvert, and C.~W. Clark.
\newblock {\em NIST Handbook of Mathematical Functions}.
\newblock Cambridge University Press, 2010.

\bibitem{ottobre2011asymptotic}
M.~Ottobre and G.~Pavliotis.
\newblock {Asymptotic analysis for the generalized Langevin equation}.
\newblock {\em Nonlinearity}, 24(5):1629, 2011.

\bibitem{pavliotis2014stochastic}
G.~Pavliotis.
\newblock {\em Stochastic Processes and Applications}.
\newblock Springer, 2014.

\bibitem{reif1965fundamentals}
F.~Reif.
\newblock {\em {Fundamentals of Statistical and Thermal Physics}}.
\newblock Waveland Press, 1965.

\bibitem{rosinski:zak:1996}
J.~Rosi{\'n}ski and T.~{\.Z}ak.
\newblock Simple conditions for mixing of infinitely divisible processes.
\newblock {\em Stochastic Process. Appl.}, 61(2):277--288, 1996.

\bibitem{rozanov:1967}
I.~A. Rozanov.
\newblock {\em Stationary Random Processes}.
\newblock Holden-Day, 1967.

\bibitem{soni1974parseval}
K.~Soni and R.~Soni.
\newblock {The Parseval relation and monotone functions}.
\newblock {\em J. Math. Anal. Appl.}, 48(3):633--645, 1974.

\bibitem{soni1975slowly}
K.~Soni and R.~Soni.
\newblock {Slowly varying functions and asymptotic behavior of a class of
  integral transforms I}.
\newblock {\em J. Math. Anal. Appl.}, 49(1):166--179, 1975.

\bibitem{soni1975slowlyII}
K.~Soni and R.~Soni.
\newblock {Slowly varying functions and asymptotic behavior of a class of
  integral transforms. II}.
\newblock {\em J. Math. Anal. Appl.}, 49(2):477--495, 1975.

\bibitem{spiechowicz2018quantum}
J.~Spiechowicz, P.~Bialas, and J.~{\L}uczka.
\newblock {Quantum partition of energy for a free Brownian particle: Impact of
  dissipation}.
\newblock {\em Phys. Rev. A}, 98(5):052107, 2018.

\bibitem{spiechowicz2019superstatistics}
J.~Spiechowicz and J.~{\L}uczka.
\newblock {On superstatistics of energy for a free quantum Brownian particle}.
\newblock {\em J. Stat. Mech. Theory Exp.}, 2019(6):064002, 2019.

\bibitem{spiechowicz2021energy}
J.~Spiechowicz and J.~{\L}uczka.
\newblock {Energy of a free Brownian particle coupled to thermal vacuum}.
\newblock {\em Sci. Rep.}, 11(1):1--12, 2021.

\bibitem{strichartz2003guide}
R.~S. Strichartz.
\newblock {\em {A Guide to Distribution Theory and Fourier Transforms}}.
\newblock World Scientific Publishing Company, 2003.

\bibitem{to2014boltzmann}
K.~To.
\newblock {Boltzmann distribution in a nonequilibrium steady state: Measuring
  local potential by granular Brownian particles}.
\newblock {\em Phys. Rev. E}, 89(6):062111, 2014.

\bibitem{vinales2006anomalous}
A.~D. Vi{\~n}ales and M.~A. Desposito.
\newblock {Anomalous diffusion: Exact solution of the generalized Langevin
  equation for harmonically bounded particle}.
\newblock {\em Phys. Rev. E}, 73(1):016111, 2006.

\bibitem{weideman:1994}
J.~A.~C. Weideman.
\newblock Computation of the complex error function.
\newblock {\em SIAM J. Numer. Anal.}, 31(5):1497--1518, 1994.

\bibitem{yaglom1957some}
A.~M. Yaglom.
\newblock Some classes of random fields in n-dimensional space, related to
  stationary random processes.
\newblock {\em Theory Probab. Appl.}, 2(3):273--320, 1957.

\bibitem{zwanzig2001nonequilibrium}
R.~Zwanzig.
\newblock {\em Nonequilibrium Statistical Mechanics}.
\newblock Oxford University Press, 2001.

\end{thebibliography}
\bibliographystyle{abbrv}

\end{document}